\documentclass[12pt]{amsart}
\usepackage{a4wide}
\usepackage{amssymb}
\usepackage{amsthm}
\usepackage{amsmath}
\usepackage{amscd}
\usepackage{verbatim}
\usepackage[all]{xy}

\numberwithin{equation}{section}

\theoremstyle{plain}
\newtheorem{theorem}{Theorem}[section]
\newtheorem{corollary}[theorem]{Corollary}
\newtheorem{lemma}[theorem]{Lemma}
\newtheorem{proposition}[theorem]{Proposition}

\newtheorem{conjecture}[theorem]{Conjecture}
\theoremstyle{definition}

\newtheorem{remark}[theorem]{Remark}

\theoremstyle{remark}

\newcommand{\OO}{\mathcal O}
\newcommand{\A}{\mathbb{A}}
\newcommand{\R}{\mathbb{R}}
\newcommand{\G}{\mathbb{G}}
\newcommand{\Q}{\mathbb{Q}}
\newcommand{\Z}{\mathbb{Z}}

\newcommand{\C}{\mathbb{C}}

\renewcommand{\H}{\mathbb{H}}
\newcommand{\F}{\mathbb{F}}
\newcommand{\D}{\mathbb{D}}

\newcommand{\zxz}[4]{\begin{pmatrix} #1 & #2 \\ #3 & #4 \end{pmatrix}}
\newcommand{\abcd}{\zxz{a}{b}{c}{d}}

\newcommand{\kzxz}[4]{\left(\begin{smallmatrix} #1 & #2 \\ #3 & #4\end{smallmatrix}\right) }
\newcommand{\kabcd}{\kzxz{a}{b}{c}{d}}

\newcommand{\calC}{\mathcal{C}}

\newcommand{\calE}{\mathcal{E}}
\newcommand{\calF}{\mathcal{F}}

\newcommand{\calL}{\mathcal{L}}

\newcommand{\calN}{\mathcal{N}}
\newcommand{\calO}{\mathcal{O}}
\newcommand{\calP}{\mathcal{P}}

\newcommand{\calS}{\mathcal{S}}

\newcommand{\calX}{\mathcal{X}}
\newcommand{\calZ}{\mathcal{Z}}

\newcommand{\fraka}{\mathfrak a}

\newcommand{\frake}{\mathfrak e}

\newcommand{\frakn}{\mathfrak n}

\newcommand{\eps}{\varepsilon}
\newcommand{\bs}{\backslash}
\newcommand{\norm}{\operatorname{N}}
\newcommand{\pr}{\operatorname{pr}}
\newcommand{\vol}{\operatorname{vol}}
\newcommand{\tr}{\operatorname{tr}}

\newcommand{\Cl}{\operatorname{Cl}}
\newcommand{\Sl}{\operatorname{SL}}
\newcommand{\Gl}{\operatorname{GL}}

\newcommand{\GSpin}{\operatorname{GSpin}}
\newcommand{\CT}{\operatorname{CT}}
\newcommand{\Mp}{\operatorname{Mp}}
\newcommand{\Orth}{\operatorname{O}}

\newcommand{\Hom}{\operatorname{Hom}}
\newcommand{\Aut}{\operatorname{Aut}}
\newcommand{\Char}{\operatorname{char}}
\newcommand{\Mat}{\operatorname{Mat}}
\newcommand{\Spec}{\operatorname{Spec}}

\newcommand{\End}{\operatorname{End}}

\newcommand{\GL}{\operatorname{GL}}
\newcommand{\SO}{\operatorname{SO}}

\newcommand{\Gal}{\operatorname{Gal}}

\newcommand{\supp}{\operatorname{supp}}

\newcommand{\fh}{\operatorname{ht}}
\newcommand{\dv}{\operatorname{div}}

\newcommand{\z}{\operatorname{Z}}
\newcommand{\za}{\operatorname{\widehat{Z}}}
\newcommand{\rata}{\operatorname{\widehat{Rat}}}

\newcommand{\cha}{\operatorname{\widehat{CH}}}

\newcommand{\diag}{\operatorname{diag}}
\newcommand{\ord}{\operatorname{ord}}
\newcommand{\kay}{k}
\newcommand{\Gspin}{\operatorname{GSpin}}
\newcommand{\boldbeta}{\text{\boldmath$\beta$\unboldmath}}

\newcommand{\ff}{\hbox{if }}
\newcommand{\SL}{\operatorname{SL}}

\begin{document}

\title[Faltings heights of CM cycles and derivatives of  $L$-functions]{Faltings heights of CM cycles and derivatives of  $L$-functions}

\author[Jan H.~Bruinier and Tonghai Yang]{Jan
Hendrik Bruinier and Tonghai Yang}
\address{Fachbereich Mathematik,
Technische Universit\"at Darmstadt, Schlossgartenstrasse 7, D--64289
Darmstadt, Germany}
\email{bruinier@mathematik.tu-darmstadt.de}
\address{Department of Mathematics, University of Wisconsin Madison, Van Vleck Hall, Madison, WI 53706, USA}
\email{thyang@math.wisc.edu} 

\thanks{The second author is partially supported by grants NSF DMS-0555503
and NSFC-10628103.}


\date{\today}

\begin{abstract}
  We study the Faltings height pairing of arithmetic Heegner divisors
  and CM cycles on Shimura varieties associated to orthogonal groups.
  We compute the Archimedian contribution to the height pairing and
  derive a conjecture relating the total pairing to the central
  derivative of a Rankin $L$-function.  We prove the conjecture in
  certain cases where the Shimura variety has dimension $0$, $1$, or
  $2$.  In particular, we obtain a new proof of the Gross-Zagier
  formula.
\end{abstract}

\maketitle

\section{Introduction}

Let $E$ be an elliptic curve over $\Q$. Assume that its $L$-function $L(E,s)$ has an odd functional equation so that the central critical value $L(E,1)$ vanishes. In this case the Birch and Swinnerton-Dyer conjecture predicts the existence of a rational point of infinite order on $E$. It is natural to ask if is possible to construct such a point explicitly.
The celebrated work of Gross and Zagier \cite{GZ} provides such a construction when $L'(E,1)\neq 0$. We briefly recall their main result, the Gross-Zagier formula, in a formulation which is convenient for the present paper.

Let $N$ be the conductor of $E$, and let $X_0(N)$ be the moduli space of cyclic isogenies  of degree $N$ of generalized elliptic curves. Let $K$ be an imaginary quadratic field such that $N$ is the norm of an integral ideal of $K$, and write $D$ for the discriminant of $K$.
We may consider the divisor $Z(D)$ on $X_0(N)$ given by elliptic curves with complex multiplication by the maximal order of $K$. By the theory of complex multiplication, this divisor is defined over $K$, and its degree $h$ is given by the class number of $K$. Hence the divisor $y(D)=\tr_{K/\Q }(Z(D)-h\cdot (\infty))$ has degree zero and is defined over $\Q$.
By means of the work of Wiles et al.~\cite{Wi}, \cite{BCDT}, we obtain a rational point $y^E(D)$ on $E$ using a modular parametrization $X_0(N)\to E$.
The Gross-Zagier formula states that the canonical height of $y^E(D)$ is given by the derivative of the $L$-function of $E$ over $K$ at $s=1$, more precisely
\[
\langle y^E(D),y^E(D)\rangle_{NT}
=C \sqrt{|D|} L'(E,1)L(E,\chi_D,1).
\]
Here $C$ is an explicit non-zero constant which is independent of $K$, and $L(E,\chi_D,s)$ denotes
the quadratic twist of $L(E,s)$ by the quadratic Dirichlet character $\chi_D$ corresponding to $K/\Q$.
 It is always possible to choose $K$ such that $L(E,\chi_D,1)$ is non-vanishing. So $y^E(D)$ has infinite order if and only if $L'(E,1)\neq 0$.

The work of Gross and Zagier triggered a lot of further research on height pairings of algebraic cycles on Shimura varieties. For instance,
Zhang considered heights of Heegner type cycles on Kuga-Sato fiber varieties over modular curves in \cite{Zh1}, and
the heights of Heegner points on compact Shimura
curves over totally real fields in \cite{Zh2}.
Gross and Keating discovered a connection between arithmetic intersection numbers of Hecke correspondences on
the product of two copies of the modular curve $X(1)$ over $\Z$ and the coefficients of the derivative of
the Siegel Eisenstein series of genus three and weight two \cite{GK}. This inspired an extensive program of
Kudla, Rapoport and Yang relating Arakelov intersection numbers on Shimura varieties of orthogonal type to derivatives of Siegel Eisenstein series and modular $L$-functions, see e.g.~\cite{Ku2}, \cite{Ku:MSRI}, \cite{KRY2}.

In these works the connection between a height pairing and the derivative of an automorphic $L$-function comes up in a rather indirect way. The idea is to identify the local height pairings in the Fourier coefficients of a suitable integral kernel function (often given by an Eisenstein series), which takes an automorphic form $\phi$ to the special value of the derivative of an $L$-function associated to $\phi$.

In the present paper we consider a different approach to obtain
identities between certain height pairings on Shimura varieties of
orthogonal type and derivatives of automorphic $L$-functions. It is
based on the Borcherds lift \cite{Bo1} and its generalization in
\cite{Br}, \cite{BF}. We propose a conjecture for the Faltings
height pairing of arithmetic Heegner divisors and CM cycles. We
compute the Archimedian contribution to the height pairing. Using
this result we prove the conjecture in certain low dimensional
cases. 
We now describe the content of this paper in more detail.

\medskip


Let $(V,Q)$ be a quadratic space over $\Q$ of signature $(n,2)$,
and let $H=\GSpin(V)$. We realize the hermitian  symmetric space
corresponding to $H(\R)$ as the Grassmannian $\D$ of oriented
negative definite two-dimensional subspaces of $V(\R)$. For a
compact open subgroup $K\subset H(\A_f)$ we consider the Shimura
variety
\begin{align*}
X_K=H(\Q) \bs \big( \D\times H(\A_f)  / K\big).
\end{align*}
It is a quasi-projective variety of dimension $n$, which is defined over $\Q$.

We define CM cycles on $X_K$ following \cite{Scho}.
Let $U\subset V$ be a negative definite
two-dimensional {\em rational} subspace of $V$. It determines a two point
subset $\{z_U^\pm\}\subset \D$ given by $U(\R)$ with the two possible choices of orientation.
Let $V_+\subset V$ be the orthogonal
complement of $U$.
Then $V_+$ is a positive definite
subspace of dimension $n$, and we have the rational splitting
$V=V_+\oplus U$.
Let $T=\GSpin(U)$, which we view as a subgroup of $H$ acting trivially
on $V_+$,  and put $K_T=K\cap T(\A_f)$. We obtain the CM cycle
\begin{align*}
Z(U)=T(\Q)\bs \big( \{ z_U^\pm\} \times T(\A_f)/K_T\big)
\longrightarrow X_K.
\end{align*}

We aim to compute the Faltings height pairing of $Z(U)$ with arithmetic Heegner divisors on $X_K$ that are constructed
by means of a regularized theta lift.
We use a similar setup as in \cite{Ku4}.
Let $L\subset V$ be an even lattice, and write $L'$ for the dual of $L$.
The discriminant group $L'/L$ is finite. We consider the
space $S_L$ of Schwartz functions on $V(\A_f)$ which are
supported on $L'\otimes \hat \Z$ and which are constant on cosets of
$\hat{L}=L\otimes \hat \Z$. The characteristic functions $\phi_\mu=\Char(\mu+\hat L)$ of the cosets $\mu\in L'/L$ form a basis of $S_L$.
We write $\Gamma'=\Mp_2(\Z)$ for the full inverse image of
$\Sl_2(\Z)$ in the two fold metaplectic covering
of $\Sl_2(\R)$.
Recall that there is a Weil representation $\rho_L$ of $\Gamma'$ on $S_L$, see \eqref{eq:rhol}.


Let $k\in \frac{1}{2}\Z$. We write $M^!_{k,\rho_L}$ for the space of $S_L$-valued
weakly holomorphic modular forms  of weight $k$ for $\Gamma'$
with representation $\rho_L$.
Recall that weakly holomorphic modular forms are those meromorphic modular forms whose poles are supported at the cusps.
The space of weakly holomorphic modular forms is contained in the space $H_{k,\rho_L}$ of harmonic weak Maass forms of weight $k$ for $\Gamma'$
with representation $\rho_L$  (see Section \ref{sect:3} for precise definitions).
An element $f\in H_{k,\rho_L}$ has a Fourier expansion of the form
\begin{align*}
f(\tau)= \sum_{\mu\in L'/L}\sum_{\substack{n\in \Q\\ n\gg-\infty}} c^+(n,\mu) q^n\phi_\mu
+\sum_{\mu\in L'/L} \sum_{\substack{n\in \Q\\
n< 0}} c^-(n,\mu) \Gamma(1-k, 4\pi |n| v) q^n \phi_\mu,
\end{align*}
where $\Gamma(a,t)$ denotes the incomplete Gamma function, and $v$ is the imaginary part of  $\tau\in \H$.
Note that $c^\pm (n,\mu)=0$ unless $n\in Q(\mu)+\Z$, and that there are only finitely many $n<0$ for which $c^+(n,\mu)$ is non-zero.
There is an antilinear  differential operator $\xi:H_{k,\rho_L}\to S_{2-k,\bar\rho_L}$
to the space of cusp forms of weight $2-k$ with dual representation.
It is surjective and its kernel is equal to $M^!_{k,\rho_L}$.

Assume that $K\subset H(\A_f)$ acts trivially on $L'/L$.
Recall that for any $\mu\in L'/L$ and for any positive $m\in Q(\mu)+\Z$ there is a Heegner divisor
$Z(m,\mu)$ on $X_K$, see Section~\ref{sect4}. An arithmetic divisor on $X_K$ is a pair $(x,g_x)$
consisting of a divisor $x$ on $X_K$ and a Green function $g_x$ of logarithmic type for $x$. For the divisors $Z(m,\mu)$ we obtain such Green functions
by means of the regularized theta lift of harmonic weak Maass forms.
For $\tau\in \H$, $z\in \D$ and $h\in H(\A_f)$, let $\theta_L(\tau,z,h)$ be the Siegel theta function associated to the lattice $L$.
Let $f\in H_{1-n/2,\bar\rho_L}$ be a harmonic weak Maass form of
weight $1-n/2$,
and denote its Fourier expansion as above.
We consider
the regularized theta integral
\begin{align*}
\Phi(z,h,f)=\int_{\calF}^{reg} \langle f(\tau),
\theta_L(\tau,z,h)\rangle\,d\mu(\tau).
\end{align*}
This theta lift was studied in \cite{Br}, \cite{BF},
generalizing the Borcherds lift of weakly holomorphic modular forms \cite{Bo1}.
It turns out that $\Phi(z,h,f)$  is a logarithmic
Green function for the divisor
\[
Z(f)=\sum_{\mu\in L'/L}\sum_{m>0} c^+ (-m,\mu)
Z(m,\mu)
\]
in the sense of Arakelov
geometry (see \cite{SABK}).
It is harmonic when $c^+(0,0)=0$. The pair $\hat Z(f)=(Z(f),\Phi(\cdot,f))$
defines an arithmetic divisor on $X_K$.
We obtain a linear map
\[
H_{1-n/2,\bar\rho_L}\longrightarrow \hat\z{}^1(X_K)_\C,\quad f\mapsto \hat Z(f)
\]
to the group of arithmetic divisors on $X_K$.
Using the Borcherds lift \cite{Bo1}, we see that
this map takes weakly holomorphic modular forms with vanishing constant term to arithmetic divisors which are rationally equivalent to zero.

Let $\calX\to \Spec(\Z)$ be a regular
scheme which is projective and flat over $\Z$, of relative
dimension $n$.
An arithmetic divisor on $\calX$ is a pair $(x,g_x)$ of a divisor $x$ on $\calX$ and a logarithmic Green function
$g_x$ for the divisor $x(\C)$  induced by $x$ on the complex variety $\calX(\C)$, see \cite{SABK}.
Recall from \cite{BGS} that there is a height pairing
\[
\cha^1(\calX)\times \z^n(\calX) \longrightarrow \R
\]
between the first arithmetic Chow group of $\calX$ and the group of codimension $n$ cycles.
When $\hat x=(x,g_x)\in \cha^1(\calX)$ and $y\in \z^n(\calX)$ such that $x$ and $y$ intersect properly on the generic fiber,
it is defined by
\[
\langle \hat x, y\rangle_{Fal} = \langle x,y\rangle_{fin} + \langle \hat x,y\rangle_\infty,
\]
where
$\langle \hat x,y\rangle_\infty = \frac{1}{2}g_x(y(\C))$,
 and $\langle x,y\rangle_{fin}$ denotes the intersection
pairing at the finite places. The
quantity $\langle \hat x, y\rangle_{Fal}$ is called the Faltings
height of $y$ with respect to $\hat x$.

We now give a conjectural formula for the Faltings height pairing of arithmetic Heegner divisors and CM cycles
(see Section \ref{sect:fal} for details).
We are quite vague here and ignore
various difficult  technical problems regarding regular models.
Assume that there is a regular scheme $\calX_K\to \Spec\Z$,
projective and flat over $\Z$, whose associated complex variety is a
smooth compactification
of $X_K$. Let $\calZ(m ,\mu)$ and $\calZ(U)$ be suitable
extensions to $\calX_K$ of the cycles $Z(m, \mu)$ and $Z(U)$,
respectively. Such extensions can be found in many cases using a
moduli interpretation of $\calX_K$, see e.g. \cite{Ku:MSRI},
\cite{KRY2}, or by taking flat closures as in \cite{BBK}.
%
For an $f \in H_{1-n/2, \bar{\rho}_L}$, we set $\calZ(f) =\sum_\mu \sum_{m>0} c^+(-m, \mu) \calZ(m,
\mu)$. Then
the pair
\[
\hat \calZ(f)=(\calZ(f),\Phi(\cdot,f))
\]
defines
an arithmetic divisor in $\cha^1(\calX_K)_\C$.
The pairing of this divisor with the CM cycle $\calZ(U)$ should be
given by the central derivative of a certain Rankin type
$L$-function which we now describe.


Using the splitting $V=V_+\oplus U$, we obtain definite lattices
$N=L\cap U$ and $P=L\cap V_+$.
 Let
\begin{align*}
\theta_P(\tau)&=\sum_{\mu\in P'/P} \sum_{m\geq 0} r(m,\mu)
q^m\phi_\mu
\end{align*}
be the Fourier expansion of the $S_P$-valued theta series
associated to the positive definite lattice $P$.
For a cusp form  $g\in S_{1+n/2,\rho_L}$ with Fourier expansion $g=\sum_{\mu} \sum_{m>0} b(m,\mu) q^m\phi_\mu$,
we consider the Rankin type $L$-function
\begin{align}
L(g, U,
s)=(4\pi)^{-(s+n)/2}\Gamma\left(\tfrac{s+n}{2}\right)\sum_{m>0}\sum_{\mu\in
P'/P}r(m,\mu) \overline{b(m,\mu)}  m^{-(s+n)/2},
\end{align}
where $g$ is considered as an $S_{P\oplus N}$-valued cusp form in a natural way (via Lemma \ref{sublattice}).
This $L$-function can be written as a Rankin-Selberg convolution against an incoherent Eisenstein series
$E_N(\tau,s;1)$ of weight $1$ associated to the negative definite lattice $N$, see
Section~\ref{sect:cmgreen}.
Under mild assumptions on $U$, the completed $L$-function
$L^*(g, U, s) :=\Lambda(\chi_D, s+1) L(g, U, s)$
satisfies the functional equation
\[
L^*(g, U, s) =-L^*(g, U, -s).
\]
Consequently, it vanishes at $s=0$, the center of symmetry, and it is of interest to describe the derivative
$L'( g,U,0)$.


\begin{conjecture} \label{conj:intro}
Let $f\in H_{1-n/2, \bar{\rho}_L}$, and assume that the constant term $c^+(0,0)$ of $f$ vanishes. Then
\begin{align}
\label{eq:fhintro} \langle \hat \calZ(f), \calZ(U) \rangle_{Fal}
=\frac{2}{\vol(K_T)}L'( \xi(f),U,0).
\end{align}
\end{conjecture}

In Section \ref{sect4} we compute the Archimedian contribution to the height pairing, see Theorem~\ref{thm:fund}.
\begin{theorem}
\label{thm:fundintro}
The Archimedian height pairing $\langle \hat \calZ(f), \calZ(U)
\rangle_{\infty}$ is given by
\begin{align*}
\frac{1}{2}\Phi(Z(U),f)&=\frac{2}{\vol(K_T)}\left( \CT\left(\langle
f^+,\, \theta_P \otimes \calE_N\rangle\right) +
L'(\xi(f), U,0)\right).
\end{align*}

\end{theorem}

Here $f^+$ denotes the ``holomorphic part'' of the harmonic weak Maass form $f$ and $\calE_N(\tau)$ is the holomorphic part of the derivative $E_N'(\tau,0;1)$
of the Eisenstein series associated to $N$, see \eqref{eq:calE}.
Moreover, $\CT(\,\cdot\,)$ means the constant term of a holomorphic Fourier series.
In the proof we combine the approach of Kudla and Schofer  to evaluate regularized theta integrals on special cycles (see \cite{Ku4}, \cite{Scho})
with results on harmonic weak Maass forms and automorphic Green functions obtained in \cite{BF}. The basic idea is to view
the evaluation of $\Phi(z,h,f)$ on $Z(U)$ as an integral over $T(\Q)\bs T(\A_f)/K_T$. Then the CM value $\Phi(Z(U),f)$ can be computed using a see-saw identity,
the Siegel-Weil formula, and the properties of the Maass lowering and raising operators on Eisenstein series and harmonic weak Maass forms.

When $f$ is actually weakly holomorphic then $\xi(f)=0$ and
Theorem \ref{thm:fundintro} reduces to the main result of
\cite{Scho}. Moreover, the Borcherds lift of $f$ gives rise to a
relation which shows that the arithmetic divisor $\hat\calZ(f)$ is
rationally equivalent to zero. Hence the Faltings height in
Conjecture \ref{conj:intro} vanishes. Therefore the Archimedian
contribution to the height pairing must equal the negative of the
contribution from the finite places. This leads to a general
conjecture for the finite intersection pairing of $\calZ(m,\mu)$
and $\calZ(U)$ (see Conjecture \ref{conj4.8}) which motivates
Conjecture \ref{conj:intro}:

\begin{conjecture}
\label{conj1.3}
Let $\mu\in L'/L$,  and let $m\in Q(\mu)+\Z$ be
positive. Then $ \langle \calZ(m, \mu), \calZ(U)\rangle_{fin}$ is
equal to $-\frac{2}{\vol(K_T)}$ times the $(m,\mu)$-th Fourier
coefficient of $\theta_P\otimes \calE_N$.
\end{conjecture}

In view of Theorem \ref{thm:fundintro}, this conjecture is essentially equivalent to Conjecture  \ref{conj:intro}. We discuss this in detail in  Section
\ref{sect:fal}, where we also give a slight generalization and
derive some consequences.


In Section \ref{sect6} we consider the case
 $n=0$ where $V$ is negative definite of
dimension $2$. Then we have $U=V$.
The even Clifford algebra of
$V$ is an imaginary quadratic field $\kay=\mathbb Q(\sqrt D)$,
and $H=\Gspin(V) =
\kay^*$.
For
simplicity we assume that the lattice $L$ is isomorphic to  a fractional
ideal $\fraka \subset k$ with the scaled norm $-\norm(\cdot)/\norm(\fraka)$ as the quadratic form.
We take
$K=\hat{\OO}_\kay^*$,
which acts on $L'/L$ trivially.  Then $X_K$
is the union of  two copies of the ideal class group $\Cl(k)$.
An integral model over $\Z$ can be found by slightly varying the setup of \cite{KRY1}. It is given as
the moduli stack $\mathcal C$ over $\Z$
of elliptic curves with complex multiplication by the ring of integers
of $k$.
The Heegner divisors can be defined on $\mathcal C$ by considering CM elliptic curves whose endomorphism ring is larger, and therefore equal to
an order of a quaternion algebra.
They are supported in finite characteristic.

In this case the lattice $P$ is zero-dimensional and the $L$-function $L(\xi(f),U,s)$ vanishes identically.
Therefore Conjecture \ref{conj:intro} reduces to the statement that the arithmetic degree of the Heegner divisor $\calZ(f)$ on $\mathcal C$ should be given
by the negative of the average of the regularized theta lift of $f$. We prove this identity using Theorem \ref{thm:fundintro} and the results obtained in \cite{KRY1}, respectively their generalization in \cite{KY1}. More precisely we show (see Theorem \ref{thm:n=0}):

\begin{theorem}
\label{thm:n=0intro}
Let $f\in H_{1,\bar\rho_L}$ and assume that the
constant term of $f$ vanishes. Then
\begin{align*}
\widehat{\deg}(\calZ(f))=-\frac{1}{2} \sum_{(z,h)\in X_K}\Phi(z,h,f).
\end{align*}
\end{theorem}

In Section \ref{sect:5} we consider the case $n=1$. We let $V$ be the
rational quadratic space of signature $(1,2)$ given by the trace zero
$2\times 2$ matrices with the quadratic form $Q(x)=N\det(x)$, where
$N$ is a fixed positive integer.  In this case $H\cong \Gl_2$. We
chose the lattice $L\subset V$ and the compact open subgroup $K\subset
H(\A_f)$ such that $X_K$ is isomorphic to the modular curve
$\Gamma_0(N)\bs \H$.
The Heegner divisors $Z(m,\mu)$ and the CM cycles $Z(U)$ are both supported on CM points and therefore closely related.

The space $S_{3/2,\rho_L}$ can be identified with the space of Jacobi cusp
forms of weight $2$ and index $N$. Recall that there is a Shimura
lifting from this space to cusp forms of weight $2$ for $\Gamma_0(N)$,
see \cite{GKZ}.  Let $G$ be a normalized newform of weight $2$
for $\Gamma_0(N)$ whose Hecke $L$-function $L(G,s)$ satisfies an odd
functional equation.  There exists a newform $g\in S_{3/2,\rho_L}$
corresponding to $G$ under the Shimura correspondence.
It turns out that the $L$-function $L(g,U,s)$
is proportional to $L(G,s+1)$, see Lemma \ref{prop:l1}.

We may choose
$f\in H_{1/2,\bar\rho_L}$ with vanishing constant term such that
$\xi(f)=\|g\|^{-2} g$ and such that the principal
part of $f$ has coefficients in the number field generated by the
eigenvalues of $G$.
Then $Z(f)$ defines an explicit point in the Jacobian of $X_0(N)$, which lies in the $G$ isotypical component,
see Theorem \ref{thm:isogen}.
In this case Conjecture~\ref{conj:intro} essentially reduces to the following Gross-Zagier type formula for the Neron-Tate height of $Z(f)$ (Theorem~\ref{GZ1}).

\begin{theorem}
\label{GZintro}
The Neron-Tate height of $Z(f)$ is given by
\[
\langle Z(f), Z(f)\rangle_{NT}=  \frac{2\sqrt{N}}{\pi\|g\|^2}
L'\big(G,1).
\]
\end{theorem}

The proof of this result which we give in Section \ref{sect:gz}
is quite different
from the original proof of  Gross and Zagier  and uses {\it minimal}
information on finite intersections between  Heegner divisors.
Instead, we derive it from Theorem \ref{thm:fundintro},
modularity of the generating series of Heegner divisors (Borcherds'
approach to the Gross-Kohnen-Zagier theorem \cite{Bo2}), and
multiplicity one for the subspace of newforms in $S_{3/2,\rho_L}$
\cite{SZ}.  Another crucial ingredient is the non-vanishing result for
coefficients of weight $2$ Jacobi cusp forms by Bump, Friedberg,
and Hoffstein \cite{BFH}.
Employing in addition the Waldspurger type formula for the coefficients of
$g$ \cite{GKZ}, we also obtain the Gross-Zagier formula
as stated at the beginning.

We conclude Section~\ref{sect7} by giving an alternative proof of Conjectures \ref{conj:intro} and \ref{conj1.3} in this case.
It relies on the computation of the finite intersection pairing of
$\calZ(f)$ and $\calZ(U)$ by pulling back to $\calZ(U)$ and employing
the results for the $n=0$ case obtained in Section~\ref{sect6}.
Finally, in Section~\ref{sect:8} we use the same idea to prove
Conjecture~\ref{conj1.3} in certain special cases for $n=2$. Here we consider the case where the CM $0$-cycle lies on the diagonal in a Hilbert modular surface. The normalization of the Hirzebruch-Zagier divisor given by the diagonal is the modular curve of level $1$. We may pull back the divisor $\calZ(f)$ to this modular curve and compute the intersection there using the results of Section~\ref{sect7} (see Theorem~\ref{thm:n=2}).

The paper is organized as follows. In Section~\ref{sect:2} we collect important facts on theta series, Eisenstein series and the Siegel-Weil formula.
In Section~\ref{sect:3} we recall some results on vector valued modular forms and harmonic weak Maass forms.
In Section~\ref{sect4} we define the regularized theta lift and compute the CM values of automorphic Green functions.
Section~\ref{sect:fal} contains the conjectures on Faltings heights. In Section~\ref{sect6} we consider the case $n=0$, in Section~\ref{sect:5} the case $n=1$, and in Section~\ref{sect:8} the case $n=2$.

We would like to thank W. Kohnen and S. Kudla for very helpful conversations.
Part of this paper was written, while the first author was visiting the Max-Planck Institute for Mathematics in Bonn.
He would like to thank the institute for providing a stimulating environment.
The second author thanks  the AMSS and the Morningside Center
of Mathematics at Beijing for providing a wonderful working environment
during his visits there, where  part of this work is done.

\section{Theta series and Eisenstein series}
\label{sect2}  \label{sect:2}

Here we fix the basic setup. We present some facts on theta series,
Eisenstein series, and the Siegel-Weil formula. We refer to
\cite{Ku1}, \cite{Ku4} for details.

Let $(V,Q)$ be a quadratic space over $\Q$ of signature $(n,2)$. Let
$H=\GSpin(V)$, and $G=\Sl_2$, viewed as an algebraic groups over $\Q$.
Recall that there is an exact sequence of algebraic groups
\[
1\longrightarrow \G_m \longrightarrow H\longrightarrow \SO(V)
\longrightarrow 1.
\]
Let $\A$ be the ring of adeles of $\Q$. We write $G'_\A$ for the
twofold metaplectic cover of $G(\A)$. We frequently identify
$G'_\R$, the full inverse image in $G'_\A$ of $G(\R)$, with the
group of pairs
\[
(g,\phi(\tau))
\]
where $g=\kabcd\in \Sl_2(\R)$ and $\phi(\tau)$ is a holomorphic
function on the upper complex half plane $\H$ such that
$\phi(\tau)^2=c\tau+d$. The multiplication is given by
$(g_1,\phi_1(\tau))(g_2,\phi_2(\tau))= (g_1
g_2,\phi_1(g_2\tau)\phi_2(\tau))$.

Let $K'$ be the full inverse image in $G'_\A$ of
$K=\Sl_2(\hat\Z)\subset G(\A_f)$. Let $K'_\infty$ be the full
inverse image in $G'_\R$ of $K_\infty=\SO(2,\R)\subset G(\R)$. We
write $G_\Q'$ for the image in $G_\A'$ of $G(\Q)$ under the
canonical splitting. We have $G'_\A=G'_\Q G'_\R K'$ and
\[
\Gamma:=\Sl_2(\Z)\cong G'_\Q\cap G'_\R K'.
\]
We write $\Gamma'=\Mp_2(\Z)$ for the full inverse image of
$\Sl_2(\Z)$ in $G'_\R$. Then for every $\gamma'\in \Gamma'$ there
are unique elements $\gamma\in \Gamma$ and $\gamma''\in K'$ such
that
\begin{align}
\label{eq:gamma} \gamma=\gamma'\gamma''.
\end{align}
The assignment $\gamma'\mapsto \gamma''$ defines a homomorphism
$\Gamma'\to K'$. Let $\psi$ be the standard non-trivial additive
character of $\A/\Q$. The groups $G'_\A$ and $H(\A)$ act on the
space $S(V(\A))$ of Schwartz-Bruhat functions of $V(\A)$ via the
Weil representation $\omega=\omega_\psi$.

For $\varphi\in S(V(\A))$ we have the usual theta function
\[
\vartheta(g,h;\varphi)= \sum_{x\in V(\Q)} (\omega(g,h)\varphi)(x),
\]
where $g\in G'_\A$ and $h\in H(\A)$. It is left invariant under
$G'_\Q$ by Poisson summation, and it is trivially left invariant
under $H(\Q)$.

Here we consider the following specific Schwartz functions. We
realize the hermitean symmetric space corresponding to $H(\R)$ as
the Grassmannian
\[
\D=\{ z\subset V(\R);\quad \text{$\dim(z)=2$ and $Q\mid_z<0$}\}
\]
of oriented negative definite $2$-dimensional subspaces of $V(\R)$. For any
$z\in \D$, we may consider the corresponding majorant
\[
(x,x)_z= (x_{z^\perp}, x_{z^\perp})-(x_z,x_z),
\]
which is a positive definite quadratic form on the vector space
$V(\R)$. The Gaussian
\[
\varphi_\infty(x,z)=\exp(-\pi (x,x)_z )
\]
belongs to $S(V(\R))$. It has the invariance property
$\varphi_\infty(hx,hz)= \varphi_\infty(x,z)$ for any $h\in H(\R)$.
Moreover, it has weight $n/2-1$ under the action of the maximal
compact subgroup $K'_\infty\subset G'_\R$.
Let $\varphi_f\in S(V(\A_f))$. We obtain a theta function on
$G'_\A\times H(\A)$ by putting
\begin{align}
\label{theta1}
\theta(g,h;\varphi_f)=\vartheta\big(g,h;\varphi_\infty(\cdot,z_0
)\otimes \varphi_f(\cdot)\big),
\end{align}
where $z_0\in \D$ denotes a fixed base point. This theta
function can be written as a theta function on $\H\times \D$ in the
usual way. For $z\in \D$ we chose a $h_z\in H(\R)$ such that $h_z z_0
=z$. Notice that
$\omega(h_z)\varphi_\infty(\cdot,z_0)=\varphi_\infty(\cdot,z)$.
Moreover, choosing $i$ as a base point for $\H$, we put
\[
g_\tau=\zxz{1}{u}{0}{1}\zxz{v^{1/2}}{0}{0}{v^{-1/2}}
\]
for $\tau=u+iv\in \H$ and write $g_\tau'=(g_\tau,1)\in G'_\R$. So we
have $g'_\tau i =\tau$. We then obtain the theta function
\begin{align}
\label{theta1.5}
\theta(\tau,z,h_f;\varphi_f) & = v^{-n/4+1/2}\vartheta\big(g'_\tau,(h_z,h_f);\varphi_\infty(\cdot,z_0 )\otimes \varphi_f(\cdot)\big)\\
\nonumber & = v^{-n/4+1/2} \sum_{x\in V(\Q)} \omega(g_\tau') \big(
\varphi_\infty(\cdot,z) \otimes \omega(h_f)\varphi_f\big)(x)
\end{align}
for $h_f\in H(\A_f)$. Using the fact that
\[
v^{-n/4+1/2} \omega(g_\tau')\big(\varphi_\infty(\cdot,z)\big) (x)= v
\,e\big( Q(x_{z^\perp})\tau+Q(x_{z})\bar\tau\big),
\]
we find more explicitly
\begin{align}
\label{theta2} \theta(\tau,z,h_f;\varphi_f)  = v  \sum_{x\in V(\Q)}
e\big( Q(x_{z^\perp})\tau+Q(x_{z})\bar\tau\big)\otimes
\varphi_f(h_f^{-1} x).
\end{align}

By means of the argument of \cite[Lemma 1.1]{Ku4}, we obtain the
following transformation formula for $\theta(\tau,z,h_f;\varphi_f)$
under $\Gamma'$. Let $\gamma'=(\kabcd,\phi)\in \Gamma'$, and write
$\gamma=\gamma'\gamma''$ as in \eqref{eq:gamma}. Then we have
\begin{align}
\label{theta3} \theta\left(\abcd\tau,z,h_f;\varphi_f\right)=
\phi(\tau)^{n-2}
\theta\left(\tau,z,h_f;\omega_f(\gamma'')^{-1}\varphi_f\right).
\end{align}
If we view $\theta\left(\tau,z,h_f;\cdot\right)$ as a function on $\H$ with values in the dual space $S(V(\A_f))^\vee$ of $S(V(\A_f))$, we see that $\theta\left(\tau,z,h_f;\cdot\right)$ transforms as a (non-holomorphic) modular form of weight $n/2-1$ with representation $\omega_f^\vee$.

Let $L\subset V$ be an even lattice and write $L'$ for the dual
lattice. The discriminant group $L'/L$ is finite. We consider the
subspace $S_L$ of Schwartz functions  in $S(V(\A_f))$ which are
supported on $L'\otimes \hat \Z$ and which are constant on cosets of
$\hat{L}=L\otimes \hat \Z$. For any $\mu\in L'/L$, the
characteristic function
\[
\phi_\mu=\Char(\mu+\hat{L})
\]
belongs to $S_L$, and we have
\[
S_L=\bigoplus_{\mu\in L'/L}\C\phi_\mu \subset S(V(\A_f)).
\]
In particular, the dimension of $S_L$ is equal to $|L'/L|$. The
space $S_L$ is stable under the action of $K'$ via the Weil
representation (see \cite{Ku4}).

We define a $S_L$-valued theta function by putting
\begin{align}
\label{theta2.5} \theta_L(\tau,z,h_f)=\sum_{\mu\in L'/L}
\theta(\tau,z,h_f;\phi_\mu) \phi_\mu.
\end{align}
If we identify $S_L$ with the group ring $\C[L'/L]$ by mapping
$\phi_\mu$ to the standard basis element $\frake_\mu$ of $\C[L'/L]$,
then $\theta_L(\tau,z,1)$ is exactly the Siegel theta function
$\Theta_L(\tau,z)$ considered by Borcherds in \cite{Bo1} \S4 for the
polynomial $p=1$. (Under this identification of $S_L$ with
$\C[L'/L]$ the $L^2$ scalar product on $S_L$ corresponds to the
standard scalar product on $\C[L'/L]$. The convolution product
corresponds to the usual product in $\C[L'/L]$.) Let
$\gamma'=(\kabcd,\phi)\in \Gamma'$. We write
$\gamma=\gamma'\gamma''$ as in \eqref{eq:gamma} and put
\begin{align}
\label{eq:rhol} \rho_L(\gamma')=\bar\omega_f(\gamma'').
\end{align}
Then $\rho_L$ defines a representation of $\Gamma'$ on $S_L$. The
transformation formula \eqref{theta3} implies that
\begin{align}
\label{theta4} \theta_L\left(\abcd\tau,z,h_f\right)=
\phi(\tau)^{n-2} \rho_L(\gamma') \theta_L\left(\tau,z,h_f\right).
\end{align}
Let $T= \left( \kzxz{1}{1}{0}{1}, 1\right)$, and $S= \left(
\kzxz{0}{-1}{1}{0}, \sqrt{\tau}\right)$ denote the standard
generators of $\Gamma'$. Recall that the action of $\rho_L$ is given
by
\begin{align}
\label{eq:weilt}
\rho_L(T)(\phi_\mu)&=e(\mu^2/2)\phi_\mu,\\
\label{eq:weils}
\rho_L(S)(\phi_\mu)&=\frac{e((2-n)/8)}{\sqrt{|L'/L|}} \sum_{\nu\in
L'/L} e(-(\mu,\nu)) \phi_\nu,
\end{align}
see e.g. \cite{Bo1}, \cite{Ku4}, \cite{Br}.

\subsection{The Siegel-Weil formula}

Here we briefly recall the Siegel-Weil formula in our setting (see
\cite{Weil}, \cite{KR1}, \cite{KR2}, \cite{Ku1}). We assume that $n$
is even, which is sufficient for our purposes. In this case the
dimension of $V$ is even so that the Weil representation factors
through $G=\Sl_2$.

For $a\in \G_m$ we put $m(a)=\kzxz{a}{0}{0}{a^{-1}}$, and for $b\in
\G_a$ we put $n(b)= \kzxz{1}{b}{0}{1}$. Let $P=MN\subset G$ be the
parabolic subgroup of upper triangular matrices, where
\begin{align}
M &=\left\{ m(a);\; a\in \G_m\right\}, \\
N &=\left\{ n(b);\; b\in
\G_a\right\}.
\end{align}
Let $\chi_V$ denote the quadratic character of $\A^\times/\Q^\times$
associated to $V$ given by
\[
\chi_V(x)=(x,(-1)^{\dim V/2} \det(V))_\A.
\]
Here $\det(V)$ denotes the Gram determinant of $V$ and
$(\cdot,\cdot) $ is the Hilbert symbol on $\A^*$. For $s\in \C$ we
denote by $I(s,\chi_V)$ the principal series representation of
$G(\A)$ induced by $\chi_V  |\cdot|^s$. It consists of  all smooth
functions $\Phi(g,s)$ on $\G(\A)$ satisfying
\[
\Phi(n(b)m(a)g,s)=\chi_V(a) |a|^{s+1}\Phi(g,s)
\]
for all $b\in \A$, $a\in \A^\times$, and the action of $G(\A)$ is
given by right translations. There is a $G(\A)$-intertwining map
\begin{equation}
\lambda: S(V(\A))\longrightarrow I(s_0,\chi_V),\quad
\lambda(\varphi)(g)=(\omega(g)\varphi)(0),
\end{equation}
where $s_0=\dim(V)/2-1$.
A section $\Phi(s)\in I(s,\chi_V)$ is called standard, if its
restriction to $K_\infty K$ is independent of $s$. Using the Iwasawa
decomposition $G(\A)=N(\A)M(\A)K_\infty K$, we see that the function
$\lambda(\varphi)\in I(s_0,\chi_V)$ has a unique extension to a
standard section $\lambda(\varphi,s)\in I(s,\chi_V)$ such that
$\lambda(\varphi,s_0)=\lambda(\varphi)$.

We give an example at the Archimedian place. For $\ell\in \Z$, let
$\chi_\ell$ be the character of $K_\infty$ defined by
\[
\chi_\ell(k_\theta)=e^{i\ell\theta},
\]
where
$k_\theta=\kzxz{\cos\theta}{\sin\theta}{-\sin\theta}{\cos\theta}\in
K_\infty$. Let $\Phi_\infty^\ell(s)\in I_\infty(s,\chi_{V})$  be the
unique standard section such that
\[
\Phi_\infty^\ell(k_\theta,s)=\chi_\ell(k_\theta)=e^{i\ell\theta}.
\]
In terms of the Iwasawa decomposition we have
\begin{align}
\label{iwphi} \Phi_\infty^\ell(n(b)m(a)k_\theta,s)=\chi_V(a)
|a|^{s+1}e^{i\ell\theta}.
\end{align}
Then it is easily seen that for the Gaussian we have
\begin{align}
\label{gaussian}
\lambda_\infty(\varphi_\infty(\cdot,z))=\Phi_\infty^{n/2-1}(s_0).
\end{align}

For any standard section $\Phi(s)$, the Eisenstein series
\[
E(g,s;\Phi)=\sum_{\gamma\in P(\Q)\bs G(\Q)} \Phi(\gamma g, s)
\]
converges for $\Re(s)>1$ and defines an automorphic form on $G(\A)$.
It has a meromorphic continuation in $s$ to the whole complex plane
and satisfies a functional equation relating $E(g,s;\Phi)$ and
$E(g,-s;M(s)\Phi)$. In our special case, the Siegel Weil formula
says the following (see \cite{Weil}, \cite[Theorem 4.1]{Ku4}).

\begin{theorem}
\label{Siegel-Weil} Let $V$ be a rational quadratic space of
signature $(n,2)$ as above. Assume that $V$ is anisotropic or that
$\dim(V)-r_0>2$, where $r_0$ is the Witt index of $V$. Then
$E(g,s;\lambda(\varphi))$ is holomorphic at $s_0$ and
\[
\frac{\alpha}{2} \int_{\SO(V)(\Q)\bs \SO(V)(\A)}
\vartheta(g,h;\varphi)\,dh=E(g,s_0; \lambda(\varphi)).
\]
Here $dh$ is the Tamagawa measure on $\SO(V)(\A)$, and $\alpha=2$ if
$n=0$, and $\alpha=1$ if $n>0$.
\end{theorem}

Note that the theta integral on the right hand side converges
absolutely by Weil's convergence criterion.

\subsection{Quadratic spaces of signature $(0,2)$.}

\label{sect:2.2}

Here, as in \cite{Scho}, we are interested in the special case,
where $V$ is a definite space of signature $(0,2)$. Then $(V,Q)$ is
isometric to $(k,-c\norm(\cdot))$ for an imaginary quadratic field
$k$ with the negative of the norm form scaled by
a constant $c\in \Q_{>0}$.
The group $H(\Q)$ can be identified with the multiplicative group
$k^*$ of $k$, and $\SO(V)$ is the group of norm $1$ elements in
$k$. The homomorphism $H\to \SO(V)$ is given by $h\mapsto h\bar
h^{-1}$, and $\SO(V)$ acts on $k$ by multiplication. Moreover, the
Grassmannian $\D$ consists of the two points $z_V^+$ and $z_V^-$
given by $V(\R)$ with positive and negative orientation,
respectively. We want to compute the integral of the theta function
$\theta_L(\tau,z_V,h_f)$ in \eqref{theta2.5}, where $z_V\in \D$. To
this end, for $\ell\in \Z$, we define a $S_L$-valued Eisenstein
series of weight $\ell$ by putting
\begin{align}
\label{eisvector} E_L(\tau,s;\ell) & = v^{-\ell/2} \sum_{\mu\in
L'/L} E(g_\tau,s;\Phi_\infty^{\ell}\otimes
\lambda_f(\phi_\mu))\phi_\mu.
\end{align}

We normalize the measure on $\SO(V)(\R)\cong \SO(2,\R)$ such that
$\vol(\SO(V)(\R))=1$. This determines the normalization of the measure
$dh_f$ on $\SO(V)(\A_f)$. Note that  in this
normalization we have $\vol(\SO(V)(\Q)\bs \SO(V)(\A_f))=2$.
By Theorem \ref{Siegel-Weil} and
\eqref{gaussian} we obtain:

\begin{proposition}
\label{sw2} We have
\[
\int_{\SO(V)(\Q)\bs \SO(V)(\A_f)} \theta_L(\tau,z_V,h_f)\,dh=E_L(\tau,0; -1).
\]
\end{proposition}

Following \cite[\S IV.2]{Ku1},  we write down this Eisenstein series
more classically. It is easily seen that $P(\Q)\bs
G(\Q)=\Gamma_\infty \bs \Gamma$, where $\Gamma_\infty=P(\Q)\cap
\Gamma$. Hence we have
\[
E(g_\tau, s;\Phi)= \sum_{\gamma\in \Gamma_\infty\bs \Gamma}
\Phi(\gamma g_\tau,s).
\]
For $\gamma=\kabcd\in \Gamma$ we consider the Iwasawa decomposition
$\gamma g_\tau = n m(\alpha) k_\theta$, where $\alpha\in \R_{>0}$
and $\theta\in \R$. A calculation shows that
\begin{align*}
\alpha&=v^{1/2}|c\tau+d|^{-1},\\
e^{i\theta} &= \frac{c\bar\tau+d}{|c\tau+d|}.
\end{align*}
Inserting this into \eqref{iwphi}, we see that
\[
\Phi_\infty^\ell(\gamma g_\tau,s) = v^{s/2+1/2}(c\tau+d)^{-\ell}
|c\tau+ d|^{\ell-s-1}.
\]
Therefore we obtain
\begin{align*}
E(g_\tau,s,\Phi^\ell_\infty\otimes \lambda_f(\phi_\mu))
&=\sum_{\gamma\in \Gamma_\infty\bs\Gamma}
(c\tau+d)^{-\ell}\frac{v^{s/2+1/2}}{|c\tau+d|^{s+1-\ell}}\cdot \lambda_f(\phi_\mu)(\gamma)\\
&=\sum_{\gamma\in \Gamma_\infty\bs\Gamma}
(c\tau+d)^{-\ell}\frac{v^{s/2+1/2}}{|c\tau+d|^{s+1-\ell}}\cdot
\langle \phi_\mu,(\omega_f^{-1}(\gamma))\phi_0\rangle .
\end{align*}
Here $\langle\cdot ,\cdot \rangle$ denotes the $L^2$ scalar product
on $S_L$. Consequently, for the vector valued Eisenstein series we
find
\begin{align}
E_L(\tau,s; \ell)=\sum_{\gamma\in \Gamma_\infty\bs\Gamma}
\left[\Im(\tau)^{(s+1-\ell)/2}\phi_0\right]\mid_{\ell,\rho_L}
\gamma,
\end{align}
where $\mid_{\ell,\rho_L}$ is the usual Petersson slash operator of
weight $\ell$ for the representation $\rho_L$. In particular we see
that this Eisenstein series coincides up to a shift in the argument
with the Eisenstein series considered in \cite[\S3]{BK}.

Let $L_\ell=-2iv^2\frac{\partial}{\partial \bar \tau}$ be the Maass
lowering operator, and let $R_\ell=2i\frac{\partial}{\partial
\tau}+\ell v^{-1}$ be the Maass raising operator in weight $\ell$.
It is easily seen that
\begin{align*}
L_\ell
E_L(\tau,s; \ell) &= \frac{1}{2}(s+1-\ell) E_L(\tau,s; \ell-2),\\
R_\ell E_L(\tau,s; \ell) &= \frac{1}{2}(s+1+\ell) E_L(\tau,s;
\ell+2)
\end{align*}
(see also \cite[ Lemma 2.7]{Ku4}). In particular, we see that
\begin{align}
\label{eis1} L_1 E_L(\tau,s; 1) = \frac{s}{2} E_L(\tau,s; -1).
\end{align}
Since $E_L(\tau,s; -1)$ is holomorphic at $s=0$ by Theorem
\ref{Siegel-Weil}, we find that  $E_L(\tau,s; 1)$ vanishes at $s=0$,
the center of symmetry. This corresponds to the fact that
$E_L(\tau,s; 1)$ is an incoherent Eisenstein series, see \cite{Ku2},
because it is constructed at all finite places from the data
corresponding to the quadratic space $(V,Q)$, but at the Archimedian
place one takes $(V,-Q)$. In particular $E_L(\tau,s; 1)$ satisfies
an odd functional equation under $s\mapsto -s$, which explains
the vanishing at $s=0$ (see also Proposition \ref{prop:fueq}).
The identity \eqref{eis1} implies that
\begin{align}
\label{eis2} L_1 E_L'(\tau,0; 1) = \frac{1}{2} E_L(\tau,0; -1),
\end{align}
where  $E_L'(\tau,s; 1)$ denotes the derivative of  $E_L(\tau,s; 1)$
with respect to $s$. This identity can be written in terms of
differential forms as follows.

\begin{lemma}
\label{eis3} We have
\[
-2  \bar\partial\left(E_L'(\tau,0; 1)\, d\tau\right) = E_L(\tau,0;
-1)\, d\mu(\tau).
\]
\end{lemma}

As in \cite{Scho} we write the Fourier expansion of the Eisenstein
series in the form
\begin{align}
\label{E1} E_L(\tau,s; 1) = \sum_{\mu\in L'/L} \sum_{m\in \Q}
A_\mu(s,m,v) q^m \phi_\mu,
\end{align}
where $q=e^{2\pi i \tau}$ as usual. The coefficients $A_\mu(s,m,v)$
are computed in \cite{KY2}, \cite{KRY1}, \cite{Scho}, and \cite{BK}.
The formulas we will need later are summarized in Theorem
\ref{thm:kappaexplicit} below.
 Notice that
$A_\mu(s,m,v)=0$ unless $m\in Q(\mu)+\Z$. Since the Eisenstein
series vanishes at $s=0$, the coefficients have a Laurent expansion
of the form
\begin{align}
\label{A1} A_\mu(s,m,v)=b_\mu(m,v)s+O(s^2)
\end{align}
 at $s=0$, and we have
\begin{align}
\label{E1'} E_L'(\tau,0; 1)=\sum_{\mu\in L'/L} \sum_{m\in \Q}
b_\mu(m,v) q^m \phi_\mu.
\end{align}
For the evaluation of an automorphic Green function at a CM cycle,
the following quantities play a key role:
\begin{align}
\label{kappa1} \kappa(m,\mu):=\begin{cases}
\lim_{v\to \infty} b_\mu(m,v),& \text{if $m\neq 0$ or $\mu\neq 0$,}\\
\lim_{v\to \infty} b_0(0,v)-\log(v),& \text{if $m=0$ and $\mu=0$.}\\
\end{cases}
\end{align}
According to \cite[Proposition 2.20 and Lemma 2.21]{Scho},  (see
also \cite[Theorem 2.12]{Ku4}), the limits exist. If $m>0$, then
$b_\mu(m,v)$ is actually independent of $v$ and equal to
$\kappa(m,\mu)$. We also have $\kappa(m, \mu)=0$ for $m<0$ or $m=0$,
$\mu \ne 0$.
Using the
quantities $\kappa(m,\mu)$ we define a holomorphic $S_L$ valued
function on $\H$ by
\begin{align}
\label{eq:calE} \calE_L(\tau)=\sum_{\mu\in L'/L} \sum_{m\in \Q}
\kappa(m,\mu) q^m \phi_\mu.
\end{align}
This function is clearly periodic, but it is not invariant under the
$\mid_{1,\rho_L}$-action of $S\in \Gamma'$.

\begin{remark}
Another way of interpreting \eqref{eis2} is that $E'_L(\tau,0;1)$ is
a harmonic weak Maass form (see Section \ref{sect:3.1}) of weight
$1$ which is mapped to $v^{-1} \overline{E_L(\tau,0;-1)}$ under
$\xi$. The function $\calE_L(\tau)$ is simply the holomorphic part
of $E'_L(\tau,0;1)$.
\end{remark}

Now we assume that  $(L, Q)=(\mathfrak a,
-\frac{\norm}{\norm(\mathfrak a)})$ where $\mathfrak a$ is a
fractional ideal of an imaginary quadratic field $k=\mathbb Q(\sqrt
D)$ with fundamental discriminant $D \equiv 1 \pmod{4}$.
We denote by $\calO_k$ the ring of integers in $k$, and write $\partial$ for the different of $k$.
In this case,
$V=k$,
the dual lattice is given by
$L'=\partial^{-1}\mathfrak a$, and
$$
L'/L =\partial^{-1}\mathfrak a/\mathfrak a\cong
\partial^{-1}/\calO_k\cong\Z/D\Z.
$$

\begin{proposition}
\label{prop:fueq}
Let the notation be as above. Let $\chi_D$ be the
quadratic Dirichlet character associated to $k/\mathbb Q$, and let
$$
\Lambda(\chi_D, s) =|D|^{\frac{s}2} \Gamma_{\mathbb R}(s+1) L( \chi_D,
s), \quad \Gamma_{\mathbb R}(s) =\pi^{-\frac{s}2} \Gamma(\frac{s}2)
$$
be its complete $L$-function. Let
$$
E_L^*(\tau, s) = \Lambda(\chi_D, s+1) E_L(\tau, s),
$$
then
$$
E_L^*(\tau, s) =-E_L^*(\tau, -s).
$$
\end{proposition}
\begin{proof}  It is equivalent to prove the
equation for each $E(\tau, s, \mu)=E(\tau, s, \Phi_\infty^1 \otimes
\lambda_f(\phi_\mu)$. By Langlands' general theory of Eisenstein
series, one has
$$
E(g, s, \Phi) = E(g, -s, M(s)\Phi).
$$
Here $M(s) =\prod_{p\le \infty} M_p(s): I(s, \chi_D) \rightarrow
I(-s, \chi_D)$ is the usual intertwining operator given by (when
$\Re (s) \gg 0$)
\begin{equation} \label{eqN4.20}
 M_p(s) \Phi_p(g, s) = \int_{\mathbb Q_p} \Phi_p(wn(b)g, s) db
 \end{equation}
for $\Phi_p \in I_p(s, \chi_D)$, where $I_p(s, \chi_D)$ is the local
principal series.

When $p=\infty$, it is well-known that
\begin{equation}  \label{eqN4.21}
M_\infty(s) \Phi_\infty^1(g, s) = C_\infty(s) \Phi_\infty^1(g, -s)
\end{equation}
with
\begin{equation}
C_\infty(s) =M_\infty(s) \Phi_\infty^1(1,s).
\end{equation}
It is also known (see for example \cite[Proposition 2.6]{KRY1}) that
\begin{equation} \label{eqN4.23}
C_\infty(s) = -\gamma_\infty(V) \frac{\Gamma_\mathbb
R(s+1)}{\Gamma_\mathbb R(s+2)}
\end{equation}
where $\gamma_\infty(V) =-\gamma_\infty(-V)=i$ is the local Weil
index associated to the local Weil representation $\omega_V$ of
$\SL_2(\mathbb R)$ on the Schwartz space $S(V\otimes \mathbb R)$
with respect to the dual pair $(\Orth(V), \SL_2)$. The fact we need
here is that $\gamma_\infty(V)=-\gamma_\infty(-V)$, not the explicit
formula, and $\Phi_\infty^1$ coming from the dual pair $(\Orth(-V),
\SL_2)$.

When $p \nmid D\infty$, $L$ is unimodular, and it is  well-known
(see \cite[Section 2]{KRY1}) that
\begin{equation} \label{eqN4.24}
M_p(s) \Phi_\mu(g, s) = C_p(s) \Phi_\mu(g, -s)
\end{equation}
with
\begin{equation}
C_p(s) =\frac{L_p(\chi_D, s)}{L_p(\chi_D, s+1)}.
\end{equation}
When $p|D$ and $\mu=0$, the intertwining operator is also computed
in  \cite[Sections 2, 3]{KRY1}. We do it here in general. Let
$$
K_0(p)=\{ g=\abcd \in \SL_2(\mathbb Z_p) ;  \quad c \equiv 0 \pmod{p}\}.
$$
Every $g \in K_0(p)$ can be written as product
$$
g=n_-(pc)n(b) m(a), \quad  n_-(c) =\kzxz {1} {0} {c} {1}, \quad n(b)
=\kzxz {1} {b} {0} {1}, \quad m(a) =\kzxz {a} {0} {0} {a^{-1}}
$$
with $a \in \mathbb Z_p^*$, $b, c \in \mathbb Z_p$. Let $w=\kzxz {0}
{-1} {1} {0}$, then
$$
\SL_2(\mathbb Z_p) =K_0(p) \cup B(\mathbb Z_p) w N(\mathbb Z_p),
$$
where
$$
B (\mathbb Z_p)=\{ m(a) ; \; a \in \mathbb Z_p^*\}, \quad N(\mathbb
Z_p)=\{ n(b); \; b \in \mathbb Z_p\}.
$$
Notice also that $\SL_2(\mathbb Q_p) =B(\mathbb Q_p) \SL_2(\mathbb
Z_p)$. So $\Phi_p \in I(s, \chi_D)$ is determined by its value in
$K_p=\SL_2(\mathbb Z_p)$. Recall locally that
$$
\Phi_\mu(g, s)=|a(g)|^s(\omega_V(g)\phi_\mu)(0),
$$
where $\phi_\mu = \hbox{char}(\mu +L_p)$ is the $p$-part of
$\phi_\mu$ and $|a(g)| =|a|$ if $g =n(b) m(a) k$ with $k \in
\SL_2(\mathbb Z_p)$. One can check that
\begin{align}
\Phi_\mu(g n(b)) &=\psi(bQ(\mu)) \Phi_\mu(g), \quad b \in \mathbb
Z_p,  \notag
\\
 \Phi_\mu(g m(a)) &= \chi_D(a) \Phi_{a^{-1}\mu}(g), \quad a \in
 \mathbb Z_p^*,
 \\
  \Phi_\mu(g n_-(pc)) &= \Phi_\mu(g), \quad c \in \mathbb Z_p,
  \notag
  \\
   \Phi_\mu(gw) &=\gamma_p(V) \vol(L_p) \sum_{\lambda \in L_p'/L_p}
   \psi_p(-(\mu, \lambda)) \Phi_\lambda(g), \notag
   \end{align}
   since $\omega_V(g)\phi_\mu$ satisfies similar equations.
Here $\psi=\prod_p \psi_p$ is the `canonical' additive character of
$\mathbb Q_\A$, and $\vol(L_p)=[L_p':L_p]^{-\frac{1}2}$ is the
measure of $L_p$ with respect to the self-dual Haar measure on $L_p$
(with respect to $\psi_p$). By definition, it is easy to see that
$\tilde\Phi_\mu(g, -s) = M(s) \Phi_\mu(g, s) \in I_p(-s, \chi_D)$
satisfies the same equations. So both $\Phi_\mu$ and
$\tilde\Phi_\mu$ are determined by their values at $g=1$.
A  simple computation gives
  $$
  \tilde\Phi_\mu(1) = \gamma_p(V) \vol(L_p) \delta_{\mu,
  0}=\gamma_p(V) \vol(L_p) \Phi_\mu(1).
  $$
  Since both functions satisfy the same set of equations and are
  determined by their values at $g=1$, one has
 \begin{equation} \label{eqN4.28}
  M_p(s)\Phi_{\mu}(g, s)=\tilde\Phi_\mu(g, -s) =\gamma_p(V) \vol(L_p) \Phi_\mu(g, -s).
  \end{equation}
Combining (\ref{eqN4.21})--(\ref{eqN4.28})  together with
\begin{align*}
\prod_{p|D} \vol(L_p) =D^{-\frac{1}2}, \qquad  \prod_{p\le \infty}
\gamma_p(V) =1,
\end{align*}
one sees that
$$
E(g, s, \Phi_\infty^1 \otimes \lambda_f(\phi_\mu))
 =\frac{\Lambda(\chi_D, s)}{\Lambda(\chi_D, s+1)} E(g, -s,
 \Phi_\infty^1\otimes \lambda_f(\phi_\mu)).
 $$
 This proves the proposition since
 $
 \Lambda(\chi_D, s) =\Lambda(\chi_D,1- s).
 $
 \end{proof}

We end this section with a theorem of Schofer \cite[Theorem
4.1]{Scho}, which will be used later. 

\begin{theorem}\label{theoY2.5}
\label{thm:kappaexplicit} Let the notation be as above, and  let
$h_k$ be the class number of $k=\mathbb Q(\sqrt D) $. Write
$\chi=(D, \cdot)_\A=\prod_p \chi_p$ as a product of local quadratic
characters. Let $\mu \in L'/L$ and $m >0$ such that $m \in Q(\mu)
+\mathbb Z$. Then
\begin{align*}
-\Lambda(\chi_D,1) \kappa(m, \mu)
 &=  \eta_0(m, \mu)  \sum_{\text{$p$ inert} } (\ord_p(m) +1) \rho(m|D|/p) \log p
  \\
   &\qquad +  \rho(m|D|) \sum_{p|D} \eta_p(m, \mu)
(\ord_p(m) +1) \log p
  \end{align*}
  where
\begin{align*}
\eta_p(m, \mu) &=(1-\chi_p(-m\norm(\mathfrak a))) \prod_{\substack{q|D, q\ne p\\
\mu_q =0}} (1+ \chi_q(-m\norm(\mathfrak a))),\\
\eta_0(m, \mu) &= \prod_{\substack{q|D\\
\mu_q =0}} (1+ \chi_q(-m\norm(\mathfrak a))).
\end{align*}
Here  we take $\eta_0(m, \mu)=1$ and $\eta_p(m, \mu)=0$ if $\mu_q\ne
0$ for all $q|D$. Finally,
$$
\rho(n) = \#\{ \mathfrak b \subset \calO_k; \; \norm(\mathfrak b)
=n\}.
$$
 We also have
$$
\kappa(0, 0)= \log|D| - 2 \frac{\Lambda'( \chi_{D}, 0)}{\Lambda(
\chi_{D}, 0)}.
$$
\end{theorem}
\begin{proof} The formula for $\kappa(0, 0)$ follows from
\cite[Lemma 2.21]{Scho}. The other formula is \cite[Theorem
4.1]{Scho} when $D<-3$. Looking into his proof, the formula is true
in general for $D \equiv 1 \pmod{4}$ if we replace $h_k$ by
$\Lambda(\chi_{D}, 1)$. Notice that
\begin{equation} \label{neweq2.33}
\Lambda(\chi_D, 1) = \frac{\sqrt{|D|}}{\pi} L(\chi_D, 1)=
\frac{2}{w_k} h_k.
\end{equation}
Here $w_k$ is the number of roots of unity in $k$.
\end{proof}

\section{Vector valued modular forms}
\label{sect:3}

Let $(V,Q)$ be a quadratic space as in Section \ref{sect:2}, and let
$L\subset V$ be an even lattice.
In this section we make no restriction on the signature $(b^+,b^-)$ of $V$.
We consider the subspace $S_L$ of Schwartz functions  in
$S(V(\A_f))$ which are supported on $\hat{L'}=L'\otimes \hat \Z$ and
which are constant on cosets of $\hat L$. For any $\mu\in L'/L$, we
write $\phi_\mu$ for the characteristic function of $\mu+\hat L$.

We use $\tau$ as a standard variable on $\H$ and write $u$ for its
real and $v$ for its imaginary part. If $f:\H\to S_L$ is a function,
we write $f=\sum_{\mu\in L'/L} f_\mu \phi_\mu$ for its decomposition
in components with respect to the standard basis $(\phi_\mu)$ of
$S_L$. Let $k\in \frac{1}{2}\Z$, and assume for simplicity that
$k\equiv \frac{b^+-b^-}{2}\pmod{2}$.
Let $\rho=\rho_L$ be the Weil
representation of $\Gamma'=\Mp_2(\Z)$ on $S_L$, see \eqref{eq:rhol}.
We denote by  $A_{k,\rho}$ the space of $S_L$-valued $C^\infty$
modular forms of weight $k$ for $\Gamma'$ with representation
$\rho$. The subspaces of weakly holomorphic modular forms
(resp.~holomorphic modular forms, cusp forms) are denoted by
$M^!_{k,\rho}$  (resp.~$M_{k,\rho}$, $S_{k,\rho}$). We need a few
facts about such vector valued modular forms.

If $f\in A_{k,\rho}$ and  $g\in A_{-k,\bar\rho}$, then the scalar
valued function
\begin{align}
\langle f(\tau), g(\tau)\rangle = \sum_{\mu\in L'/L} f_\mu(\tau)
g_\mu(\tau)
\end{align}
is invariant under $\Gamma'$. For $f,g\in A_{k,\rho}$, we define the
Petersson scalar product by
\begin{align}
(f, g)_{Pet}= \int_{\calF}\langle f,\bar g \rangle v^k \,
d\mu(\tau),
\end{align}
provided the integral converges. Here $d\mu (\tau)=\frac{du\,
dv}{v^2}$ is the invariant measure on $\H$, and $\calF= \{\tau\in
\H; \; \text{$|u|\leq 1/2$ and $|\tau|\geq 1$}\}$ denotes the
standard fundamental domain for the action of $\Gamma$ on $\H$.

Let $K$ and $L$ be even lattices. Then the Weil representation of
$K\oplus L$ is isomorphic to the tensor product of $\rho_K$ and
$\rho_L$. Moreover, if $f= \sum_{\mu\in K'/K} f_\mu
\phi_\mu \in A_{k,\rho_K}$ and $g= \sum_{\nu\in L'/L} g_\nu
\phi_\nu\in A_{l,\rho_L}$, then
\[
f\otimes g = \sum_{\mu,\nu} f_\mu g_\nu \phi_{\mu+\nu}\in A_{k+l,\rho_{K\oplus L}}.
\]

Let $M\subset L$ be a sublattice of finite index, then a vector
valued modular form $f \in A_{k,\rho_L}$ can be naturally viewed as
a vector valued modular form in $A_{k, \rho_M}$. Indeed, we have the
inclusions $M\subset L\subset L'\subset M'$ and therefore
\[
L/M\subset L'/M\subset M'/M.
\]
We have the natural map $L'/M\to L'/L$, $\mu\mapsto \bar \mu$.
\begin{lemma}
\label{sublattice} There are  two natural maps
$$
\operatorname{res}_{L/M}:A_{k,\rho_L} \rightarrow  A_{k,\rho_M},\quad
f\mapsto f_M
$$
and
$$  \tr_{L/M}: A_{k,\rho_M}\rightarrow  A_{k,\rho_L}, \quad g
\mapsto g^L
$$
such that for any $f \in A_{k,\rho_L}$ and $g\in A_{k,\rho_M}$
$$
\langle f, \bar g^L\rangle =\langle f_M, \bar g \rangle.
$$
They are given as follows. For  $\mu\in M'/M$ and $f \in
A_{k,\rho_L}$,
\[
(f_M)_\mu=\begin{cases} f_{\bar\mu},&\text{if $\mu\in L'/M$,}\\
0,&\text{if $\mu\notin L'/M$.}
\end{cases}
\]
For any $\bar\mu \in L'/L$, and $g \in A_{k,\rho_M}$, let $\mu$ be a
fixed preimage of $\bar\mu$ in $L'/M$. Then
$$
(g^L)_{\bar\mu} =\sum_{\alpha \in L/M} g_{\alpha +\mu}.
$$
\end{lemma}

\begin{proof}
See \cite[Proposition 6.9]{Sche} for the map $\operatorname{res}_{L/M}$. The assertion for $\tr_{L/M}$ can be proved analogously.
\end{proof}

\begin{remark}
The following fact about the trace map and theta functions, which is
easy to check, will be used in Section \ref{sect4}:
\begin{equation} \label{eqY3.3}
\theta_{L} = (\theta_M)^L.
\end{equation}
\end{remark}

\subsection{Harmonic weak Maass forms}
\label{sect:3.1}

Now assume that
$k\leq 1$.  A twice continuously differentiable function $f:\H\to
S_L$ is called a {\em harmonic weak Maass form} (of weight $k$ with
respect to $\Gamma'$ and $\rho_L$) if it satisfies:
\begin{enumerate}
\item[(i)]
$f \mid_{k,\rho_L} \gamma'= f$
for all $\gamma'\in \Gamma'$;
\item[(ii)]
there is a $S_L$-valued Fourier polynomial
\[
P_f(\tau)=\sum_{\mu\in L'/L}\sum_{n\leq 0} c^+(n,\mu) q^n\phi_\mu
\]
such that $f(\tau)-P_f(\tau)=O(e^{-\eps v})$ as $v\to \infty$ for
some $\eps>0$;
%
\item[(iii)]
$\Delta_k f=0$, where
\begin{align*}
\Delta_k := -v^2\left( \frac{\partial^2}{\partial u^2}+
\frac{\partial^2}{\partial v^2}\right) + ikv\left(
\frac{\partial}{\partial u}+i \frac{\partial}{\partial v}\right)
\end{align*}
is the usual weight $k$ hyperbolic Laplace operator (see \cite{BF}).
\end{enumerate}
The Fourier polynomial $P_f$  is called the {\em principal part} of
$f$. We denote the vector space of these harmonic weak Maass forms
by  $H_{k,\rho_L}$ (it was called $H^+_{k,\rho_L}$ in \cite{BF}).
Any weakly holomorphic modular form is a harmonic weak Maass form.
The Fourier expansion of any $f\in H_{k,\rho_L}$ gives a unique
decomposition $f=f^++f^-$, where
\begin{subequations}
\label{deff}
\begin{align}
\label{deff+}
f^+(\tau)&= \sum_{\mu\in L'/L}\sum_{\substack{n\in \Q\\ n\gg-\infty}} c^+(n,\mu) q^n\phi_\mu,\\
\label{deff-} f^-(\tau)&= \sum_{\mu\in L'/L} \sum_{\substack{n\in \Q\\
n< 0}} c^-(n,\mu) W(2\pi nv) q^n \phi_\mu,
\end{align}
\end{subequations}
and $W(a)=W_k(a):= \int_{-2a}^\infty e^{-t}t^{-k}\,
dt=\Gamma(1-k,2|a|)$ for $a<0$. We refer to $f^+$ as the {\em
holomorphic part} and to $f^-$ as the {\em non-holomorphic part} of
$f$.

Recall that there is an antilinear  differential operator $\xi=
\xi_k:H_{k,\rho_L}\to S_{2-k,\bar\rho_L}$, defined by
\begin{equation}
\label{defxi} f(\tau)\mapsto \xi(f)(\tau):=v^{k-2} \overline{L_k
f(\tau)}.
\end{equation}
Here $L_k$ is the Maass lowering operator. The kernel of $\xi$ is
equal to $M^!_{k,\rho_L}$. By  \cite[Corollary~3.8]{BF}, the
sequence
\begin{gather}
\label{ex-sequ}
\xymatrix{ 0\ar[r]& M^!_{k,\rho_L} \ar[r]& H_{k,\rho_L}
\ar[r]^{\xi}& S_{2-k,\bar\rho_L} \ar[r] & 0 }
\end{gather}
is exact.
%

There is a bilinear pairing between the spaces $M_{2-k,\bar\rho_L}$
and $H_{k,\rho_L}$ defined by the Petersson scalar product
\begin{equation}\label{defpair}
\{g,f\}=\big( g,\, \xi(f)\big)_{Pet}
\end{equation}
for $g\in M_{2-k,\bar\rho_L}$ and $f\in H_{k,\rho_L}$. If $g$ has
the Fourier expansion $g=\sum_{\mu,n} b(n,\mu) q^n\phi_\mu$, and we
denote the Fourier expansion of $f$  as in \eqref{deff}, then by
 \cite[Proposition 3.5]{BF} we have
\begin{equation}\label{pairalt}
\{g,f\}= \sum_{\mu\in L'/L} \sum_{n\leq 0}  c^+(n,\mu) b(-n,\mu).
\end{equation}
Hence $\{g,f\}$ only depends on the principal part of $f$. The
exactness of \eqref{ex-sequ} implies that the induced pairing
between $S_{2-k,\bar\rho_L}$ and $H_{k,\rho_L}/M^!_{k,\rho_L}$ is
non-degenerate.

\begin{lemma}
\label{lem:principal}
Let $f\in H_{k,\rho_L}$ and assume that $P_f$ is constant. Then $f\in M_{k,\rho_L}$.
\end{lemma}

\begin{proof}
It follows from the assumption and \eqref{pairalt} that
\[
 \big( \xi(f),\, \xi(f)\big)_{Pet}
=\{\xi(f),f\}
=0.
\]
Hence $\xi(f)=0$ and $f$ is weakly holomorphic. Since $P_f$ is
constant we find that $f\in M_{k,\rho_L}$.
\end{proof}

\begin{lemma} \label{lemY3.4}
\label{lem:n11}
Let $\mu\in L'/L$, and let $m\in \Q_{>0}$ such that $m\equiv -Q(\mu)\pmod{1}$.
There exists a harmonic weak Maass form
$f_{m,\mu}\in H_{k,\rho_L}$ whose Fourier expansion
starts as
\[
f_{m,\mu}(\tau)=\frac{1}{2}(
q^{-m}\phi_{\mu}+q^{-m}\phi_{-\mu}) + O(1) ,\quad v\to
\infty.
\]
\end{lemma}

\begin{proof}
This is an immediate consequence of \cite[Proposition 3.11]{BF}.
\end{proof}

\section{Regularized theta integrals}
\label{sect4}

Let $(V,Q)$ be a quadratic space over $\Q$ of signature $(n,2)$. We
use the setup of Section~\ref{sect:2}. In particular, $L\subset V$
is an even lattice. Let $K\subset H(\A_f)$ be a compact open
subgroup acting trivially on $S_L$. We consider the Shimura variety
\begin{align}
X_K=H(\Q) \bs \big( \D\times H(\A_f)  / K\big).
\end{align}
It is a quasi-projective variety of dimension $n$ defined over $\Q$.

On $X_K$ we consider the following Heegner divisors (cf.~\cite{Bo1},
\cite{Br}, \cite{Ku4}). We follow the description in \cite[pp.~304]{Ku4}. Let $x\in V(\Q)$ be a vector of positive norm. We write
$V_x$ for the orthogonal complement of $x$ in $V$ and $H_x$ for the
stabilizer of $x$ in $H$. So $H_x\cong \GSpin(V_x)$. The
sub-Grassmannian
\begin{align}
\label{eq:dx} \D_{x}=\{ z\in \D;\;\text{$z\perp x$}\}
\end{align}
defines an analytic divisor of $\D$. For $h\in H(\A_f)$ we consider
the natural map
\begin{align} \label{eqY4.3}
H_x(\Q)\bs \D_x\times H_x(\A_f)/(H_x(\A_f)\cap hKh^{-1})
\longrightarrow X_K,\quad (z,h_1)\mapsto (z,h_1 h).
\end{align}
Its image defines a divisor $Z(x,h)$ on $X_K$, which is rational
over $\Q$. For $m\in \Q_{>0}$ let
\begin{align}
\Omega_m=\left\{x\in V;\; Q(x)=m\right\}
\end{align}
 be the corresponding quadric in $V$. If $\Omega_m(\Q)$ is
non-empty, then by Witt's theorem, we have $\Omega_m(\Q)=H(\Q)x_0$
and $\Omega_m(\A_f)=H(\A_f)x_0$ for a fixed element $x_0\in
\Omega_m(\Q)$. For a Schwartz function $\varphi\in S_L$, we may
write
\begin{align}
\supp(\varphi)\cap \Omega_m(\A_f) = \coprod_j K\xi_j^{-1} x_0
\end{align}
as a finite disjoint union, where  $\xi_j\in H(\A_f)$. This follows
from the fact that $\supp(\varphi)$ is compact and $\Omega_m(\A_f)$
is a closed subset of $V(\A_f)$. We define a composite Heegner
divisor by putting
\begin{align}
\label{eq:zvarphi} Z(m,\varphi)=\sum_j \varphi(\xi_j^{-1} x_0)
Z(x_0,\xi_j).
\end{align}
This definition is independent of the choice of $x_0$ and the
representatives
 $\xi_j$. For
$\mu\in L'/L$ we briefly write $ Z(m,\mu):=Z(m,\phi_\mu)$. The
following lemma is a special case of \cite[Proposition 5.4]{KuDuke}.

\begin{lemma} \label{lemY4.1}
\label{heegnercomp} Assume that $H(\A_f)=H(\Q) K$ and put
$\Gamma_K=H(\Q)\cap K$. Then
\[
Z(m,\varphi) = \sum_{x\in \Gamma_K\bs \Omega_m(\Q)}\varphi(x)
\pr(\D_x,1),
\]
where $\pr:\D\times H(\A_f)\to X_K$ denotes the natural projection.
\hfill$\square$
\end{lemma}

Let $f\in H_{1-n/2,\bar\rho_L}$ be a harmonic weak Maass form of
weight $1-n/2$ with representation $\bar\rho_L$ for $\Gamma'$, and
denote its Fourier expansion as in \eqref{deff}. Throughout we
assume that $c^+(m,\mu)\in \mathbb Z$ for all $m\leq 0$. We consider
the regularized theta integral
\begin{align}
\label{reg1} \Phi(z,h,f)=\int_{\calF}^{reg} \langle f(\tau),
\theta_L(\tau,z,h)\rangle\,d\mu(\tau)
\end{align}
for $z\in \D$ and $h\in H(\A_f)$. The integral is regularized as in
\cite{Bo1}, \cite{BF}, that is, $\Phi(z,h,f)$ is defined as the
constant term in the Laurent expansion at $s=0$ of the function
\begin{align}
\label{reg2} \lim_{T\to \infty}\int_{\calF_T} \langle f(\tau),
\theta_L(\tau,z,h)\rangle\,v^{-s} d\mu(\tau).
\end{align}
Here $\calF_T=\{\tau\in \H; \; \text{$|u|\leq 1/2$, $|\tau|\geq 1$,
and $v\leq T$}\}$ denotes the truncated fundamental domain. The
following theorem summarizes some properties of the function
$\Phi(z,h,f)$ in the setup of the present paper (see \cite{Br},
\cite{BF}).

\begin{theorem} \label{theoY4.2}
The function $\Phi(z,h,f)$ is smooth on $X_K\bs Z(f)$, where
\begin{align}
Z(f)=\sum_{\mu\in L'/L}\sum_{m>0} c^+ (-m,\mu)
Z(m,\mu).
\end{align}
It has a logarithmic singularity along the divisor $-2Z(f)$. The
$(1,1)$-form $dd^c \Phi(z,h,f)$ can be continued to a smooth form on
all of $X_K$. We have the Green current equation
\begin{align} \label{eqY4.10}
dd^c[\Phi(z,h,f)]+\delta_{Z(f)}=[dd^c \Phi(z,h,f)],
\end{align}
where $\delta_Z$ denotes the Dirac current of a divisor $Z$.
Moreover, if $\Delta_z$ denotes the invariant Laplace operator on
$\D$, normalized as in \cite{Br}, we have
\begin{align}
\Delta_z  \Phi(z,h,f) = \frac{n}{4}\cdot c^+(0,0).
\end{align}
\end{theorem}

In particular, the theorem implies that $\Phi(z,h,f)$ a
Green function for the divisor $Z(f)$ in the sense of Arakelov
geometry in the normalization of \cite{SABK}. (If the constant term
$c^+(0,0)$ of $f$ does not vanish, one actually has to work with the
generalization of Arakelov geometry given in \cite{BKK}.) Moreover,
we see that $\Phi(z,h,f)$ is harmonic when $c^+(0,0)=0$.
Therefore, it is called the {\em automorphic Green function}
associated with $Z(f)$.

In the special case when $f$ is weakly holomorphic, $\Phi(z,h,f)$ is
essentially equal to the logarithm of the Petersson metric of a
Borcherds product $\Psi(z,h,f)$ on $X_K$. Note that $\Phi(z,h,f)$
has a finite value for {\em every} $z\in \D$, even on $Z(f)$, where
it is not smooth, see \cite{Scho}. Similar Green functions are
investigated from the point of view of spherical functions on real
Lie groups in \cite{OT}.
The following theorem gives a characterization of $\Phi(z,h,f)$.
Although it is not needed in the rest of the paper, we include it here
to provide some background.

\begin{theorem}
Assume that the Witt rank of $V$ over $\Q$ is smaller than $n$.
Let $G$ be a smooth real valued function on $X_K\bs Z(f)$ with the properties:
\begin{enumerate}
\item[(i)] $G$ has a logarithmic singularity along $-2Z(f)$,
\item[(ii)] $\Delta_z G = \text{constant}$,
\item[(iii)] $G\in L^{1+\eps}(X_K,d\mu(z))$ for some $\eps>0$.
\end{enumerate}
Then $G(z,h)$ differs from $\Phi(z,h,f)$ by a constant.
\end{theorem}

Here $d\mu(z)$ is the measure on $X_K$ induced from the Haar measure on the group $H(\A)$.
If the Witt rank of $V$ is equal to $n$, one can obtain a similar characterization by also requiring growth conditions at the boundary of $X_K$.
The constant could be fixed, for instance, by adding a condition on the value of the integral $\int_{X_K} G\, d\mu(z)$.

\begin{proof}[Idea of the proof]
First, we notice that $\Phi(z,h,f)$ satisfies the properties (i)--(iii). The first two are contained in Theorem \ref{theoY4.2}.
The third can be proved using the Fourier expansion of $\Phi(z,h,f)$
(see \cite{Br}) and the `curve lemma' as in \cite[Theorem 2]{Brcurve}.

Hence the difference $G(z,h)-\Phi(z,h,f)$ is a smooth subharmonic function on the complete Riemann manifold $X_K$ which is contained in $L^{1+\eps}(X_K,d\mu(z))$.
By a result of Yau, such a function must be constant (see e.g. \cite[Corollary 4.22]{Br}).
\end{proof}

\subsection{CM values of automorphic Green functions}
\label{sect:cmgreen}

We define CM cycles on $X_K$ as follows. Let $U\subset V$ be a negative definite
$2$-dimensional rational subspace of $V$. It determines a two point
subset $\{z_U^\pm\}\subset \D$ given by $U(\R)$ with the two possible choices of orientation.
Let $V_+\subset V$ be the orthogonal
complement of $U$ over $\Q$. Then $V_+$ is a positive definite
subspace of dimension $n$, and we have the rational splitting
\begin{align}
\label{split} V=V_+\oplus U.
\end{align}
Let $T=\GSpin(U)$, which we view as a subgroup of $H$ acting trivially
on $V_+$,  and put $K_T=K\cap T(\A_f)$. We obtain the CM cycle
\begin{align}
\label{eq:zu} Z(U)=T(\Q)\bs \big( \{ z_U^\pm\} \times T(\A_f)/K_T\big)
\longrightarrow X_K.
\end{align}
Here each point in the cycle is counted with multiplicity
$\frac{2}{w_{K, T}}$, where  $w_{K,T}=\# (T(\mathbb Q) \cap K_T)$.

It is our goal to compute the value of $\Phi(z,f)$ on $Z(U)$.
In the special case when $f$ is weakly holomorphic this
was done by Schofer \cite{Scho}, whose argument we will extend here.
The related problem of computing the
integrals of logarithms of Petersson norms of Borcherds products is
considered in \cite{Ku4}, \cite{BK}.

We fix the Tamagawa Haar measure on
$\SO(U)(\A)$ so that $\vol(\SO(U)(\mathbb R)) =1$, and
$\vol(\SO(U)(\Q)\backslash \SO(U)(\A_f))=2$. We also fix the usual
Haar measure on $\A_f^*$ so that $\vol(\mathbb Z_p^*)=1$. So
$\vol(\hat{\mathbb Z}^*) =1$, and $\vol(\Q^*\backslash
\A_f^*)=1/2$. We then use the exact sequence
$$
1  \rightarrow \A_f^* \rightarrow T(\A_f) \rightarrow \SO(U)(\A_f)
\rightarrow 1
$$
to define the Haar measure on $T(\A_f)$.

\begin{lemma} \label{lemY1.1} With the notation as above, one has
\begin{align}
\label{tocompute}
\Phi(Z(U),f)&= \frac{2}{w_{K, T}}\sum_{z\in \supp( Z(U))} \Phi(z,f)\\
\nonumber &=\frac{2}{\vol(K_T)}\int_{h\in \SO(U)(\Q)\bs
\SO(U)(\A_f)} \Phi(z_U^+,h,f)\, dh.
\end{align}
\end{lemma}
\begin{proof}For any $T(\mathbb Q)K_T$-invariant everywhere defined $L^1$-function $F$
on $T(\A_f)$, one has
\begin{align*}
 \int_{T(\mathbb Q) \backslash T(\A_f)} F(h) dh
  &= \int_{T(\mathbb Q) \backslash T(\A_f)/K_T}\int_{(T(\mathbb Q) \cap K_T)\backslash K_T} F(hk
  )dk \, dh
 \\
  &= \vol((T(\mathbb Q) \cap K_T)\backslash K_T) \sum_{h \in T(\mathbb Q) \backslash T(\A_f)/K_T}
  F(h)
  \\
   &=\frac{\vol(K_T)}{w_{K, T}} \sum_{h \in T(\mathbb Q) \backslash
   T(\A_f)/K_T} F(h).
   \end{align*}
On the other hand, the exact sequence
$$
1  \rightarrow \mathbb G_m \rightarrow T \rightarrow \SO(U)
\rightarrow 1
$$
implies that if $F$ is also $\A_f^*$-invariant,  we have
$$
\int_{T(\mathbb Q)\backslash T(\A_f)} F(h) dh=\frac{1}2
\int_{\SO(U)(\Q) \backslash \SO(U)(\A_f)} F(h) dh.
$$
So
$$
\sum_{h \in T(\Q) \backslash T(\A_f) /K_T} F(h) =\frac{w_{K,
T}}{2\vol(K_T)}\int_{\SO(U)(\Q) \backslash \SO(U)(\A_f)} F(h) dh.
$$
Taking $F(h) = \Phi(z_U^\pm, h, f)$, and noticing that $\Phi(z_U^+,
h, f)=\Phi(z_U^-, h, f)$, one proves the lemma.
\end{proof}

\begin{remark}
Taking $F=1$ in the proof of the previous lemma, one sees
\begin{equation}
\deg Z(U) =\frac{4}{\vol(K_T)}.
\end{equation}
\end{remark}

Using the splitting \eqref{split}, we obtain definite lattices
\begin{align*}
N=L\cap U,\quad  P=L\cap V_+.
\end{align*}
Then  $N\oplus P\subset L$ is a sublattice of finite index. Since $
\theta_{P\oplus N}=\theta_{P} \otimes \theta_N$, and $\theta_L
=(\theta_{P \oplus N})^L$ by (\ref{eqY3.3}), Lemma \ref{sublattice}
 implies that
$$
\langle f, \theta_L \rangle =\langle f_{P\oplus N}, \theta_P \otimes
\theta_N \rangle.
$$
So we may assume in the following calculation $L=P\oplus N$ if we
replace $f$ by $f_{P\oplus N}$.

 For $z=z_U^\pm$ and $h\in
T(\A_f)$, the Siegel theta function $\theta_L(\tau,z,h)$ splits up
as a product
\begin{align}
\label{splittheta} \theta_L(\tau,z_U^\pm,h)= \theta_P(\tau)\otimes
\theta_N(\tau,z_U^\pm,h).
\end{align}
Here $\theta_P(\tau)=\theta_P(\tau,1)$ is the holomorphic
$S_P$-valued theta function of weight $n/2$ associated to the
positive definite lattice $P$.

For the computation of the CM value $\Phi( Z(U),f)$ it is
convenient to write the regularized theta integral as a limit of
truncated integrals by means of the following lemma.
If $S(q)=\sum_{n\in \Z} a_n q^n$ is a Laurent series in $q$ (or a
holomorphic Fourier series in $\tau$), we write
\begin{align}
\CT(S)=a_0
\end{align}
for the constant term in the $q$-expansion.

\begin{lemma}
\label{limit} If we define
\begin{align}
A_0 =\CT(\langle f^+(\tau),\, \theta_P(\tau)\otimes \phi_{0+N}
\rangle),
\end{align}
we have
\[
\Phi(z_U^\pm,h,f)
=\lim_{T\to \infty}\left[ \int_{\calF_T} \langle f(\tau),\,
\theta_P(\tau)\otimes \theta_N(\tau,z_U^\pm,h) \rangle\,d\mu(\tau) -A_0
\log(T)\right].
\]
\end{lemma}

\begin{proof}
We use the splitting $f=f^++f^-$ of $f$ into its holomorphic and
non-holomorphic part. If we insert it into the definition
\eqref{reg1}, we obtain
\begin{align}
\label{h1} \Phi(z_U^\pm,h,f) = \int_{\calF}^{reg} \langle f^+(\tau),
\theta_L(\tau,z_U^\pm,h)\rangle\,d\mu(\tau) +\int_{\calF} \langle
f^-(\tau), \theta_L(\tau,z_U^\pm,h)\rangle\,d\mu(\tau).
\end{align}
Since $f^-$ is rapidly decreasing as $v\to \infty$, the second
integral on the right hand side converges absolutely. For the first
integral on the right hand side
we insert the factorization \eqref{splittheta} of
$\theta_L(\tau,z_U^\pm,h)$ and argue as in
\cite[Proposition~2.19]{Scho} or   \cite[Proposition~2.5]{Ku4}. We find that it is
equal to
\[
\lim_{T\to \infty}\left[ \int_{\calF_T} \langle f^+(\tau),\,
\theta_P(\tau)\otimes \theta_N(\tau,z_U^\pm,h) \rangle\,d\mu(\tau) -A_0
\log(T)\right].
\]
Adding the two contributions, we obtain the assertion.
\end{proof}

\begin{lemma}
\label{tocompute2} We have
\begin{align*}
\Phi(Z(U),f)&=\frac{2}{\vol(K_T)} \lim_{T\to \infty}\left[
\int_{\calF_T} \langle f(\tau),\, \theta_P(\tau)\otimes
E_N(\tau,0;-1) \rangle\,d\mu(\tau) -2A_0  \log(T)\right],
\end{align*}
where $E_N(\tau,0;-1)$ denotes the Eisenstein series defined in
\eqref{eisvector}.
\end{lemma}

\begin{proof}
We insert the formula of Lemma \ref{limit} into the definition
\eqref{tocompute}. The evaluation of $\Phi(z,h,f)$ at the CM cycle
is a finite sum, which may be interchanged with the limit.
Consequently, the lemma follows from the Siegel-Weil formula, see
Proposition~\ref{sw2}.
\end{proof}

For any $g\in S_{1+n/2,\rho_L}$ we define an $L$-function by means
of the convolution integral
\begin{align}
\label{eq:L} L(g, U,s)=\big( \theta_P(\tau)\otimes E_N(\tau,s;1),\,
g(\tau)\big)_{Pet}.
\end{align}
The meromorphic continuation of the Eisenstein series
$E_N(\tau,s;1)$ leads to the meromorphic continuation of $L(g, U,
s)$ to the whole complex plane. At $s=0$, the center of symmetry,
$L(g, U, s)$ vanishes because the Eisenstein series $E_N(\tau,s;1)$
is incoherent.  Proposition \ref{prop:fueq} gives the following simple
functional equation for
$$
L^*(g, U, s) :=\Lambda(\chi_D, s+1) L(g, U, s)
$$
when $N\cong (\mathfrak a ,
-\frac{\norm}{\norm(\mathfrak a)})$ for a fractional ideal $\mathfrak a$
 of $k=\mathbb Q(\sqrt D)$:
\begin{equation}
L^*(g, U, s) =-L^*(g, U, -s).
\end{equation}

 Let
\begin{align}
g(\tau)&=\sum_{\mu\in L'/L} \sum_{m>0} b(m,\mu) q^m\phi_\mu,\\
\theta_P(\tau)&=\sum_{\mu\in P'/P} \sum_{m\geq 0} r(m,\mu)
q^m\phi_\mu
\end{align}
be the Fourier expansion of $g$ and $\theta_P$, respectively. Using
the usual unfolding argument, we obtain the Dirichlet series
expansion
\begin{align}
\label{eq:luf} L(g, U,
s)=(4\pi)^{-(s+n)/2}\Gamma\left(\tfrac{s+n}{2}\right)\sum_{m>0}\sum_{\mu\in
P'/P}r(m,\mu) \overline{b(m,\mu)}  m^{-(s+n)/2}.
\end{align}

\begin{theorem} \label{theoY4.6}
\label{thm:fund} The value of the automorphic Green function
$\Phi(z,h,f)$ at the CM cycle $Z(U)$ is given by
\begin{align*}
\Phi(Z(U),f)&=\frac{4}{\vol(K_T)}\left( \CT\left(\langle
f^+(\tau),\, \theta_P(\tau)\otimes \calE_N(\tau)\rangle\right) +
L'(\xi(f), U,0)\right).
\end{align*}
Here $\calE_N(\tau)$ denotes the function defined in
\eqref{eq:calE}, and $L'(\xi(f),U,s)$ the derivative  with respect
to $s$ of the $L$-series \eqref{eq:L}.
\end{theorem}

\begin{proof}
In view of Lemma \ref{tocompute2} we have
\begin{align}
\label{hlim} \Phi(Z(U),f)&=\frac{2}{\vol(K_T)} \lim_{T\to
\infty}\left[ I_T(f) -2A_0  \log(T)\right],
\end{align}
where
\[
I_T(f):=\int_{\calF_T} \langle f(\tau),\, \theta_P(\tau)\otimes
E_N(\tau,0;-1) \rangle\,d\mu(\tau).
\]
We compute $I_T(f)$  combining the ideas of \cite{Scho} and
\cite{BF}. According to Lemma \ref{eis3}, we have
\begin{align*}
I_T(f)&=-2\int_{\calF_T} \langle f(\tau),\, \theta_P(\tau)\otimes
\bar \partial E_N'(\tau,0;1)\, d\tau
\rangle\\
&=-2\int_{\calF_T} d\langle f(\tau),\, \theta_P(\tau)\otimes
E_N'(\tau,0;1)\, d\tau \rangle +2 \int_{\calF_T} \langle
(\bar\partial f ),\, \theta_P(\tau)\otimes E_N'(\tau,0;1)\, d\tau
\rangle.
\end{align*}
Using Stokes' theorem and the definition of the Maass lowering
operator, we get
\begin{align*}
I_T(f) &=-2\int_{\partial\calF_T} \langle f(\tau),\,
\theta_P(\tau)\otimes E_N'(\tau,0;1)\, d\tau
\rangle\\
&\phantom{=}{}+2 \int_{\calF_T} \langle L_{1-n/2} f  ,\,
\theta_P(\tau)\otimes E_N'(\tau,0;1)
\rangle\, d\mu(\tau)\\
&=2\int_{\tau=iT}^{iT+1} \langle f(\tau),\, \theta_P(\tau)\otimes
E_N'(\tau,0;1)\rangle\, d\tau
\\
&\phantom{=}{}+2 \int_{\calF_T} \langle \overline{\xi(f)}  ,\,
\theta_P(\tau)\otimes E_N'(\tau,0;1) \rangle\,v^{1+n/2} d\mu(\tau).
\end{align*}
If we insert this formula into \eqref{hlim}, we obtain
\begin{align*}
\Phi(Z(U),f)&=\frac{4}{\vol(K_T)} \lim_{T\to \infty}\left[
\int_{\tau=iT}^{iT+1} \langle f(\tau),\, \theta_P(\tau)\otimes
E_N'(\tau,0;1)\rangle\, d\tau
-A_0  \log(T)\right]\\
&\phantom{=}{}+\frac{4}{\vol(K_T)}\int_{\calF} \langle
\overline{\xi(f)}  ,\, \theta_P(\tau)\otimes E_N'(\tau,0;1)
\rangle\,v^{1+n/2} d\mu(\tau).
\end{align*}
The second summand on the right hand side leads to $L'(\xi(f),U, 0)$
via the integral representation \eqref{eq:L}. For the first summand,
we may replace $f$ by its holomorphic part $f^+$, since $f^-$ is
rapidly decreasing as $v\to \infty$. Inserting the Fourier expansion
of $E_N'(\tau,0;1)$ and the definition of $A_0 $, we get
\begin{align*}
&\lim_{T\to \infty}\left[ \int_{\tau=iT}^{iT+1} \langle f(\tau),\,
\theta_P(\tau)\otimes E_N'(\tau,0;1)\rangle\, d\tau
-A_0  \log(T)\right]\\
&= \lim_{T\to \infty} \int_{\tau=iT}^{iT+1} \left\langle
f^+(\tau),\, \theta_P(\tau)\otimes \sum_{\mu\in N'/N}\sum_{m\in \Q}
\big(b_\mu(m,v)
-\delta_{\mu,0}\delta_{m,0}\log(v)\big)q^m\phi_\mu\right\rangle\,
d\tau.
\end{align*}
Here $\delta_{*,*}$ denotes the Kronecker delta. The limit is equal
to
\[
\CT\left(\langle f^+(\tau),\, \theta_P(\tau)\otimes
\calE_N(\tau)\rangle\right).
\]
This concludes the proof of the theorem.
\end{proof}

In the special case when $f$ is weakly holomorphic we have
$\xi(f)=0$. Hence the $L$-function term vanishes and the above
formula reduces to \cite[Theorem~1.1]{Scho}.

\begin{remark}
\label{rem:holmf}
If the principal part $P_f$ is constant, then
\[
\Phi(Z(U),f)=\frac{4}{\vol(K_{T})} c^+(0,0)\kappa(0,0).
\]
\end{remark}

\begin{proof}
In view of Lemma \ref{lem:principal} the assumption implies that $f\in M_{1-n/2,\bar\rho_L}$.
Hence $\xi(f)=0$ and the assertion follows.
\end{proof}

\section{Faltings heights of CM cycles}
\label{sect:fal}

Let $\calX\to \Spec(\Z)$ be an arithmetic variety, that is, a regular
scheme which is projective and flat over $\Z$, of relative
dimension $n$.
Let $\z^d(\calX)$ denote the group of codimension $d$ cycles  on $\calX$.
Recall that an arithmetic divisor on $\calX$ is a pair $(x,g_x)$ of a divisor $x$ on $\calX$ and a Green function
$g_x$ for the divisor $x(\C)$  induced by $x$ on the complex variety $\calX(\C)$.
So $g_x$ is a smooth real function on $\calX(\C)\setminus x(\C)$ with a logarithmic singularity on $x(\C)$ satisfying the current equation
\[
dd^c [g_x] +\delta_{x(\C)}=[\omega_x]
\]
with a smooth $(1,1)$-form $\omega_x$ on $\calX(\C)$.
We write $\cha^1(\calX)$ for the first arithmetic Chow group of $\calX$, that is, the
free abelian group generated by the arithmetic divisors on $\calX$ modulo rational equivalence, see \cite{SABK}.
Moreover, if $F\subset \C$ is a subfield we put
\[
\cha^1(\calX)_F=\cha^1(\calX)\otimes_\Z F.
\]
Recall from \cite{BGS} that there is a height pairing
\[
\cha^1(\calX)\times \z^n(\calX) \longrightarrow \R.
\]
When $\hat x=(x,g_x)\in \cha^1(\calX)$ and $y\in \z^n(\calX)$ such that $x$ and $y$ intersect properly on the generic fiber,
it is defined by
\[
\langle \hat x, y\rangle_{Fal} = \langle x,y\rangle_{fin} + \langle \hat x,y\rangle_\infty,
\]
where
\[
\langle \hat x,y\rangle_\infty = \frac{1}{2}g_x(y(\C)),
\]
 and $\langle x,y\rangle_{fin}$ denotes the intersection
pairing at the finite places. When $x$ and $y$ do not intersect
properly, one defines the pairing by replacing $\hat x$ by a
suitable arithmetic divisor which is rationally equivalent. The
quantity $\langle \hat x, y\rangle_{Fal}$ is called the Faltings
height of $y$ with respect to $\hat x$ (see also \cite[\S6.3]{BKK}).

Theorem \ref{thm:fund} and the examples of the next sections lead to
the following conjectures. We are quite vague here and ignore
various difficult  technical problems regarding regular models.
Assume that there is a regular scheme $\calX_K\to \Spec\Z$,
projective and flat over $\Z$, whose associated complex variety is a
smooth compactification $X_K^c$ of $X_K$. Let $\calZ(m ,\mu)$ and
$\calZ(U)$ be suitable extensions to $\calX_K$ of the cycles $Z(m,
\mu)$ and $Z(U)$, respectively. Such extensions can be found in many
cases using a moduli interpretation of $\calX_K$, see e.g.
\cite{Ku:MSRI}, \cite{KRY2}. (When $n>0$
one can often also take the flat closures
in $\calX_K$ of $Z(m, \mu)$ and $Z(U)$, respectively.)
%
For an $f \in H_{1-n/2, \bar{\rho}_L}$, the function
$\Phi(\cdot,f)$ is a Green function for the divisor
$Z(f)$. Set $\calZ(f) =\sum_\mu \sum_{m>0} c^+(-m, \mu) \calZ(m,
\mu)$. Then
the pair
\[
\hat \calZ(f)=(\calZ(f),\Phi(\cdot,f))
\]
defines an arithmetic divisor in $\cha^1(\calX_K)_\C$. (When $X_K$
is non-compact, one has to add suitable components to the divisor
$\calZ(f)$ which are supported at the boundary, see
Section~\ref{sect:5}. Moreover, if the constant term $c^+(0,0)$ of
$f$ does not vanish, one actually has to work with the generalized
arithmetic Chow groups defined in \cite{BKK}.) Theorem
\ref{thm:fund} provides a formula for the quantity
\begin{align}
\label{eq:4.23} \langle \hat \calZ(f), \calZ(U)
\rangle_{\infty}=\frac{1}{2}\Phi(Z(U),f).
\end{align}
If $f$ is weakly
holomorphic with constant term $c^+(0,0)=0$, then $\hat \calZ(f)$
should be rationally equivalent to a torsion element,
the relation being given by
the Borcherds lift of $f$. Assuming this, we  would have
$$
0=\langle\hat\calZ(f), \calZ(U)\rangle_{Fal} =\langle \calZ(f),
\calZ(U)\rangle_{fin} + \frac{1}2 \Phi(Z(U), f).
$$
 Theorem
\ref{thm:fund} then implies that
\begin{align}
\langle \calZ(f), \calZ(U)\rangle_{fin} =
-\frac{2}{\vol(K_T)}\CT\left(\langle f^+(\tau),\,
\theta_P(\tau)\otimes \calE_N(\tau)\rangle\right).
\end{align}
Expanding both sides would suggest the following conjecture on the
arithmetic intersection.

\begin{conjecture}
\label{conj4.8} Let $\mu\in L'/L$,  and let $m\in Q(\mu)+\Z$ be
positive. Then $\langle \calZ(m, \mu),
\calZ(U)\rangle_{fin}$ is equal to $-\frac{2}{\vol(K_T)}$ times
the $(m,\mu)$-th Fourier
coefficient of $\theta_P\otimes \calE_N$, that is,
$$
\langle \calZ(m, \mu), \calZ(U)\rangle_{fin} =
-\frac{2}{\vol(K_T)}
\sum_{\substack{\mu_1\in P'/P\\\mu_2\in N'/N\\ \mu_1+\mu_2\equiv \mu\;(L) }}
\sum_{\substack{\\ m_i\in \Q_{\geq 0} \\ m_1+m_2=m}}
 r(m_1, \mu_1) \kappa(m_2, \mu_2).
$$
Here $r(m, \mu)$ is the
 $(m,\mu)$-coefficient of $\theta_P$, and $\kappa(m, \mu)$ is
 the $(m, \mu)$-th coefficient of $\calE_N$.
 \end{conjecture}
This conjecture and Theorem \ref{thm:fund} would imply the following
conjecture.

\begin{conjecture} \label{conjY5.2}
For any $f\in H_{1-n/2, \bar{\rho}_L}$,
one has
\begin{align}
\label{eq:fh} \langle \hat \calZ(f), \calZ(U) \rangle_{Fal}
=\frac{2}{\vol(K_T)}\left( c^+(0,0)\kappa(0,0)+L'( \xi(f),U,0)\right).
\end{align}
\end{conjecture}

In view of Lemma \ref{lem:n11}, for $\mu \in L'/L$ and $m \in Q(\mu)+\Z$ positive, there is an $f_{m, \mu} \in H_{1-n/2, \bar{\rho}_L}$ such that
$\calZ(f_{m, \mu})= \calZ(m, \mu)$.
Evaluating Conjecture \ref{conjY5.2} for $\hat \calZ(f_{m,\mu})$ and using
Theorem \ref{thm:fund}, we see that the two conjectures are equivalent.

Conjecture \ref{conjY5.2}  has also the following consequence. Let
$Q_- \subset  P(V(\mathbb C))$ be defined by
\begin{equation}
Q_-=\{ w \in V(\mathbb C);\; (w, w)=0, \, (w, \bar w)<0\}/\mathbb
C^*.
\end{equation}
It is isomorphic to $\mathbb D$ via $w=v_1 -v_2i$ maps to the
oriented negative $2$-plane $z$ with oriented $\mathbb R$-basis
$\{v_1, v_2\}$ (see e.g. \cite{Bo1}, \cite{Br}, and
\cite{Ku4}). The restriction to $Q_-$ of the tautological line
bundle on $ P(V(\mathbb C))$ induces a line bundle $\omega$ on
$X_K$, the Hodge bundle.
%
We define the
Petersson metric on $\omega$ via
\begin{equation}
\label{eq5.5}
\|w\|_{Pet}^2 := -\frac{a}2 (w, \bar{w}).
\end{equation}
Here  $a:=2\pi e^{-\gamma}$ is a normalizing factor which turns out to be convenient, and $\gamma=-\Gamma'(1)$ is Euler's constant.
Rational sections of $\omega^k$  can be identified with meromorphic
modular
forms  $\Psi(z, h)$ of weight $k$ and level $K$ for $\SO(V)$. The
Petersson metric coincides with the usual Petersson metric for a
modular form (up to the normalizing factor $a$).
Let us assume that it has an integral model which we still denote
by $\omega$. Using the Petersson  metric, we get a metrized line
bundle $\hat\omega=(\omega, \|\cdot \|_{Pet})$. An integral
modular form $\Psi$ of weight $k$ can be viewed as a section of
$\omega$, and we have
$$
k \hat{c}_1(\hat{\omega}) = (\dv \Psi, -\log\| \Psi
\|_{Pet}^2) \in \cha^1(\mathcal X_K).
$$
Now let $f$ be a weakly holomorphic modular form for $\Gamma'$
with $c^+(0,0) \ne 0$ and $c^+(m, \mu) \in \mathbb Z$ for $m \le
0$.
Let $\Psi(z, h,
f)$ be its Borcherds lifting which is a meromorphic modular form of weight
$c^+(0, 0)/2$ for $\SO(V)$ of level $K$, see \cite{Bo1}.
Then we have
$$
\Phi(z,h, f) =- 2 \log\|\Psi(z, h, f)\|_{Pet}^2,
$$
see \cite[Theorem 13.3]{Bo1}.
Consequently,
\[
c^+(0,0) \hat{c}_1(\hat{\omega})=\hat\calZ(f).
\]
So Conjecture \ref{conjY5.2} says that
\begin{align*}
 \langle \hat{\omega}, \calZ(U) \rangle_{Fal}
=\frac{2}{\vol(K_T)}\kappa(0,0).
\end{align*}
Hence we obtain the following conjecture.

\begin{conjecture}
One has
$$
\frac{1}{\deg Z(U)}\langle \hat{\omega}, \mathcal Z(U)\rangle_{Fal}
=\frac{1}2 \kappa(0, 0).
$$
\end{conjecture}

It is interesting that the right hand side depends only on
$K_T$.  When $K_T \cong \hat{\OO}_D^*$ for a fundamental
discriminant $D <0$,
 $$ \kappa(0, 0)= \log|D| - 2
\frac{\Lambda'(\chi_{D}, 0)}{\Lambda( \chi_{D}, 0)} = 4 h_{Fal}(E)
$$
is four times the Faltings height of an elliptic curve  with complex multiplication  by $\OO_D$. This follows from
the Chowla-Selberg formula as reformulated by Colmez \cite{Co}.
When  $n=1$ and $X_K=Y_0(N)$,
the sections
of $\omega^k$ actually correspond to weight $2k$ modular forms in
the usual sense, and $\mathcal Z(U)$ is the moduli stack of CM
elliptic curves  with CM by $\OO_D$ (see Section \ref{sect7}).
In this case,  the conjecture is simply the  the celebrated
Chowla-Selberg formula just mentioned. When $n=2$, and
 $X_K$ is a Hilbert modular surface,  sections of $\omega^k$ correspond to weight
 $k$ Hilbert modular
 forms, and the left hand side of the conjecture is the Faltings height
 of a  CM abelian surface of the  CM type $(K, \Phi)$ where
 $K=\mathbb Q(\sqrt\Delta, \sqrt D)$ is a biquadratic CM quartic
 field with real quadratic subfield $F=\mathbb Q(\sqrt\Delta)$, and
 $\Phi =\Gal(K/k_D)$ as a CM type of $K$. In this case, the
 conjecture is a special case of Colmez' conjecture and follows
 from the Chowla-Selberg formula (see for example \cite[Proposition 3.3]{YaColmez}).

\section{The $n=0$ case}
\label{sect6}

Here we consider the case $n=0$ where $V$ is negative definite of
dimension $2$. Then $U=V$ and the even Clifford algebra $C^0(V)$ of
$V$ is an imaginary quadratic field $\kay=\mathbb Q(\sqrt D)$. For
simplicity we assume that
$(L,Q)\cong(\fraka,-\frac{\norm}{\norm(\fraka)})$ for a fractional
ideal $\fraka \subset k$ as in the end of Section \ref{sect2}. So
$L'=\partial^{-1} \mathfrak a$. In this case $H=T=\Gspin(V) =
\kay^*$. We take
\begin{equation}
K=K_T=\hat{\OO}_\kay^*,
\end{equation}
which acts on $L'/L$ trivially.  So
$$
X_K=Z(U)= k^* \backslash \{z_U^\pm\} \times k_f^*/\hat{\OO}_k^*  =\{
z_U^\pm\} \times \Cl(k)
$$
is the union of  two copies of the ideal class group $\Cl(k)$ (a
finite collection of points). It has the following integral model
over $\mathbb Z$.

Let $\mathcal C$ be the moduli stack over $\mathbb Z$ representing the
moduli problem which assigns to every scheme $S$ over $\mathbb Z$
the set $\mathcal C(S)$ of the CM elliptic curves $(E, \iota)$ where
$E$ is an elliptic curve over $S$ and $\iota: \OO_k \hookrightarrow
\End_S(E)=:\OO_E$ is an $\OO_k$-action on $E$ such that the main
involution on $\OO_E$ gives the complex conjugation on $k$. Indeed,
let $\mathcal C^+$ be the moduli stack over $\OO_\kay$ defined in
\cite{KRY1}, representing the moduli problem which assigns to
every scheme $S$ over $\OO_\kay$ the set $\mathcal C^+(S)$ of CM
elliptic curves $(E, \iota)$ over $S$ such that the CM action
$\iota: \OO_\kay \hookrightarrow \OO_E$ gives rise to the
structure map $\OO_\kay \rightarrow \mathcal O_S$ on the lie algebra
$\operatorname{Lie}(E)$. Then $\mathcal C$ is the restriction of
coefficients of $\mathcal C^+$ in the sense of Grothendieck, i.e., it
is $\mathcal C^+$ but viewed as a stack over $\mathbb Z$: $\mathcal
C= (\mathcal C^+\rightarrow \Spec (\OO_k) \rightarrow \Spec (\mathbb
Z))$.

\begin{lemma}
One has a bijective map between   $\mathcal C(\mathbb C)$  and
$X_K$.
\end{lemma}

\begin{proof}
It is well-known that every elliptic curve  with CM by $\OO_k$ over
$\mathbb C$ is isomorphic to $E_\mathfrak a=\mathbb C/\mathfrak a$
for some fractional ideal $\mathfrak a$ of $k$, and that the
isomorphism class of $E_{\mathfrak a}$ depends only on the ideal
class of $\mathfrak a$. On the other hand, $E_\mathfrak a$ has two
$\OO_k$-actions induced by
$$
\iota_+(r) z =rz, \quad \iota_-(r) z=\bar r z,
$$
respectively. So $(z_U^\pm, [\mathfrak a]) \mapsto (E_{\mathfrak a},
\iota_\pm)$ gives a bijection between $X_K$ and $\mathcal C(\mathbb
C)$.
\end{proof}

For $(E, \iota) \in \mathcal C(S)$, let
\begin{equation}
V(E, \iota) = \{ x \in \OO_E;\;  \text{$\iota(\alpha)  x = x
\iota(\bar\alpha)$ for all $\alpha\in \OO_k$, and $\tr x =0$}\}
\end{equation}
be the space of `special endomorphisms' with the definite quadratic
form $\norm(x) :=\deg x = -x^2$. When $S =\Spec (F)$ for an
algebraically closed field $F$, then $V(E, \iota)$ is empty if
$F=\C$ or $F=\bar{\mathbb F}_p$ for a prime $p$ which is split in
$k$. When $p$ is non-split in $k$, then $\calO_E$ is a maximal order
of the unique quaternion algebra $\mathbb B$ which is ramified
exactly at $p$ and $\infty$. In this case $V(E,\iota)$ is a positive
definite lattice of rank $2$ and $\norm(x)$ is the reduced norm  of
$x$.

For $\mu \in L'/L=\partial^{-1}\mathfrak a/\mathfrak a$ and $m \in
\mathbb Q_{>0}$, consider the moduli problem which assigns to every
scheme $S$ (over $\mathbb Z$) the set $\mathcal Z(S)$ of triples
$(E, \iota, \boldbeta)$ where
\begin{itemize}
\item[(i)]
$(E, \iota) \in \mathcal C(S)$, and
\item[(ii)]
$\boldbeta \in V(E, \iota)\partial^{-1} \mathfrak a$ such that
\begin{equation*}
\norm \boldbeta =m \norm \mathfrak a, \quad \mu + \boldbeta \in \OO_E
\mathfrak a.
\end{equation*}
\end{itemize}
It is empty unless $m \in  Q(\mu) +\mathbb Z$.

\begin{lemma}
Let the notation be as above, and assume that $m \in
Q(\mu)+\mathbb Z$. Then the above moduli problem is represented by an algebraic stack $\mathcal Z(m, \mathfrak a, \mu)$ of dimension $0$.
Furthermore, the forgetful map $(E, \iota, \boldbeta) \mapsto (E,
\iota)$ is a finite  \'etale  map from $\mathcal Z(m, \mathfrak a,
\mu)$ into $\mathcal C$.
\end{lemma}

We will view $\mathcal Z(m, \mathfrak a, \mu)$ as a cycle in
$\mathcal C$ by identifying it with its direct image under the
forgetful map. It is supported at finitely many primes which are
non-split in $k$.

\begin{proof}
Consider the similar moduli problem which assigns to
each scheme $S$ over $\OO_k$ the set $\mathcal Z^+(S)$ of the
triples $(E, \iota, \boldbeta)$ where $(E, \iota)\in \mathcal
C^+(S)$ and $\boldbeta$ satisfies the same conditions as above.
Choose a $\lambda \in \mathfrak a^{-1}/\partial \mathfrak a^{-1}$
such that the multiplication by $\lambda$ gives an isomorphism
$$
\partial^{-1}\mathfrak a/\mathfrak a  \cong  \partial^{-1}/\OO_\kay,
\quad x \mapsto \lambda x.
$$
Then $\mathcal Z^+(S)$ consists of the triples $(E, \iota,
\boldbeta)$ where $(E, \iota) \in \mathcal C^+(S)$,
$$
\boldbeta \in V(E, \iota) (\partial \mathfrak a^{-1})^{-1}, \quad
\norm (\partial \mathfrak a^{-1})\norm \boldbeta= m|D|,
$$
and
$$
\lambda \mu + \bar\lambda \boldbeta \in \OO_E.
$$
It is proved in  \cite{KY1} that this moduli problem is represented
by a DM-stack $\mathcal Z^+(m, \mathfrak a, \mu)$ (denoted there by
$\mathcal Z(m|D|,
\partial \mathfrak a^{-1}, \bar\lambda, \lambda \mu)$).
Let $\mathcal Z(m, \mathfrak a, \mu)$ be the restriction of
coefficients of $\mathcal Z^+(m, \mathfrak a, \mu)$, then $\mathcal
Z(m, \mathfrak a, \mu)$ represents the moduli problem $S \mapsto
\mathcal Z(S)$. The forgetful map is clearly a finite  \'etale map.
\end{proof}

Following \cite[Section 2]{KRY2}, we define the arithmetic  degree
of a $0$-dimensional DM-stack $\mathcal Z$ as
\begin{equation}
\widehat{\deg} (\mathcal Z)=\sum_{p} \sum_{x\in \mathcal
Z(\bar{\mathbb F}_p)} \frac{1}{\# \hbox{Aut}(x)} i_p(\mathcal Z, x)
\log p.
\end{equation}
Here $i_p(\mathcal Z, x)$ is defined as follows.
Let $
\tilde{\OO}_{\mathcal Z, x}$ be the strictly local Henselian ring of
$\mathcal Z$ at $x$, then
$$
i_p(x) =\hbox{Length}(\tilde{\OO}_{\mathcal Z, x})
$$
is the length of the local Artin ring $\tilde{\OO}_{\mathcal Z, x}$.
It is well-known that
\begin{equation}
\widehat{\deg} (\mathcal Z)
=\widehat{\deg} (\operatorname{cRes}_{\OO_K/\mathbb Z} \mathcal Z)
\end{equation}
for a DM-stack $\mathcal Z$ over the ring of integers $\OO_K$ of
some number field $K$, where $\operatorname{cRes}_{\OO_K/\mathbb Z}
\mathcal Z$ is the restriction of coefficients of $\mathcal Z$. In
particular, one has
\begin{equation} \label{neweq6.6}
\widehat{\deg}(\mathcal Z(m, \mathfrak a, \mu)) =
\widehat{\deg}(\mathcal Z^+(m,\mathfrak a, \mu)).
\end{equation}

 It is also well-known that
\begin{equation}
\widehat{\deg}(\mathcal Z) =\frac{1}{ [K:\mathbb Q]}
\widehat{\deg}(\mathcal Z\otimes_{\mathbb Z} \OO_K)
\end{equation}
for a DM-stack $\mathcal Z$ over $\mathbb Z$.

\begin{lemma}
\label{lemY6.1}
\label{volkt}
Let $w_\kay=\# \OO_\kay^*$. We have
\[
\frac{1}{\vol(K_T)}= \frac{h_\kay}{w_{k}}= \frac{\sqrt{|D|}}{2\pi} L(
\chi_D, 1).
\]
\end{lemma}

\begin{proof}
Recall that
 $T=\kay^*$ and $K=\hat{\OO}_\kay^*$ in our case. Hence $w_{K, T}=w_\kay$. Moreover, recall our Haar
 measure choice just before Lemma \ref{lemY1.1}.
Since
$$
1 \rightarrow \mathbb Q^* \backslash \Q_f^* \rightarrow
\kay^*\backslash \kay_f^* \rightarrow \kay^1 \backslash \kay_f^1
\rightarrow 1
$$
is exact, and $\vol (\mathbb Q^* \backslash \mathbb
Q_f^*)=1/2$, we see that
$$\vol(\kay^* \backslash \kay_f^*)=\vol (\mathbb Q^* \backslash \mathbb
Q_f^*) \vol(\kay^1 \backslash \kay_f^1)=1. $$
On the other hand, we have
$$
\int_{\kay^* \backslash \kay_f^*}  d^*x = \int_{\kay^* \backslash
\kay_f^*/\hat{\OO}_\kay^*} \int_{\OO_\kay^* \backslash
\hat{\OO}_\kay^*} d^*x = \frac{h_\kay}{w_\kay}
\vol(\hat{\OO}_\kay^*).
$$
Hence
$
\vol(K_T)=\vol(\hat{\OO}_\kay^*) = \frac{w_\kay}{h_\kay}$,
and the assertion follows from (\ref{neweq2.33}).
\end{proof}

Conjecture \ref{conj4.8} is just the following theorem in this
special case, which is a reformulation of a result in \cite{KY1}.

\begin{theorem}\label{theoY2.1}
Let the notation be  as above and assume that $D$ is odd. Then
$$
\widehat{\deg} (\calZ (m, \mathfrak a, \mu)) =-\frac{2} {\vol(K_T)}
\kappa(m, \mu).
$$
\end{theorem}

\begin{proof}[Sketch of the proof]
For a prime $p$ which is inert or ramified in $\kay$, let $\mathbb B$ be
the unique quaternion algebra over $\mathbb Q$ ramified exactly at
$p$ and $\infty$.  Choose a prime $p_0 \nmid 2 pD$ (depending on
$p$) satisfying
\begin{equation} \label{eqY6.8}
\operatorname{inv}_l \mathbb B = \begin{cases}
  (D, -p_0 p)_l, &\ff  p \hbox{ is inert in } \kay,
  \\
   (D, -p_0)_l,  &\ff p \hbox{ is ramified in }\kay
  \end{cases}
\end{equation}
for every prime $l$. Here $\operatorname{inv}_l \mathbb B =\pm
1$ depends on whether $\mathbb B$ is  a matrix algebra or a
division algebra.

In particular, $p_0=\mathfrak p_0 \bar{\mathfrak p_0}$ is split in
$\kay$.  For an ideal $\mathfrak b$ of $\kay$, let $[\mathfrak b]$
be the ideal class of $\mathbb B$ and $[[\mathfrak b]]$ be its
associated genus, i.e., the set of (fractional) ideals $\alpha
\mathfrak c^2 \mathfrak b$. Moreover, let
\begin{equation}
\rho(n, [[\mathfrak b]]) =\#\{ \mathfrak c \subset \OO_\kay;\;
\mathfrak a \in [[\mathfrak b]],  \quad \norm\mathfrak c =n\}.
\end{equation}
Notice that it is equal to
$$
\rho(n) =\#\{\mathfrak c \subset \OO_\kay; \; \norm\mathfrak c=n\}
$$
if it is non-zero. In \cite{KY1}, Kudla and the second author proved
the following formula:
\begin{align*}
\widehat{\deg} (\mathcal Z(m, \mathfrak a, \mu))
 &=2^{o(\mu)}\left[\sum_{\text{$p$ inert} }(\ord_{p}m +1)
 \rho(m|D|/p, [[\mathfrak p_0\partial \bar{\mathfrak a}]]) \log
 p\right.
 \\
 &\qquad + \left. \sum_{\substack{p|D,\;\mu_p=0}} (\ord_p m+1) \rho(m|D|/p,
 [[\mathfrak p_0 \mathfrak p^{-1} \partial \bar{\mathfrak a}]])
 \right].
\end{align*}
Here we decompose
$$
\partial^{-1}\mathfrak a/\mathfrak a = \oplus_{p|D}
(\partial^{-1}\mathfrak a/\mathfrak a)\otimes  \mathbb Z_p, \quad
\mu =(\mu_p)_{p|D},
$$
and $o(\mu) =\#\{ p|D;\; \mu_p=0\}$.
Comparing this with Theorem
\ref{theoY2.5}, one sees that it suffices to verify that for positive $ m \in
Q(\mu) +\mathbb Z = -\frac{\mu \bar\mu}{\norm\mathfrak a}
+\mathbb Z$ we have
\begin{equation} \label{Yeq1.6}
2^{o(\mu)} \rho(m|D|/p, [[\mathfrak p_0 \partial \bar{\mathfrak
a}]]) = \eta_0(m, \mu) \rho(m|D|/p)
\end{equation}
when $p$ is inert in $\kay$, and
\begin{equation} \label{Yeq1.7}
2^{o(\mu)} \rho(m|D|/p, [[\mathfrak p_0 \partial^{-1}
\partial \bar{\mathfrak a}]]) = \eta_p(m, \mu) \rho(m|D|)
\end{equation}
when $p$ is ramified in $\kay$ and $\mu_p=0$. For $p|D$, let $\xi_p$
be the genus character of $\Cl(\kay)/\Cl(\kay)^2$ given
by
$$
\xi_p([\mathfrak b]) = \chi_p(\norm\mathfrak b) =(D, \norm\mathfrak b)_p.
$$
Then $\mathfrak c \in [[\mathfrak b]]$ if and only if
$\xi_p(\norm\mathfrak c) = \xi_p(\norm\mathfrak b)$ for all $p|D$. So just
as
$$
\rho(n) = \prod_{l <\infty} \rho_l(n),
$$
one has
$$
\rho(n, [[\mathfrak b]])= \prod_{l\nmid D\infty} \rho_l(n)
\prod_{l|D} \rho_l(n, [[\mathfrak b]]),
$$
where $ \rho_l(n, [[\mathfrak b]]) =1$ or $0 $ depending on whether
there is an integral ideal $\mathfrak c $ such that $\norm\mathfrak c=n$
and $\xi_l(n) =\xi_l(\norm\mathfrak b)$, and
$$
\rho_l(n) =\begin{cases}
  1, &\ff l |D,
  \\
  \frac{1+(-1)^{\ord_l n}}2, &\ff \chi_p(l)=-1,
  \\
   \ord_ln +1, &\ff \chi_p(l) =1.
   \end{cases}
   $$

To see (\ref{Yeq1.6}), we may assume that there is an integral ideal
$\mathfrak c$ with $\norm\mathfrak c =m|D|/p$ (otherwise both sides are
zero). For any $l|D$, one has by (\ref{eqY6.8})
\begin{align}
\xi_l (\frac{m|D|}p \norm(\mathfrak p_0 \partial \bar{\mathfrak a}))
&=(D, m pp_0 \norm\mathfrak a)_l \notag
\\
 &=(D, -m \norm\mathfrak a)_l. \label{Yeq1.8}
 \end{align}
 When $\mu_l \ne 0$, $\mu \bar\mu \notin \mathbb Z_l$, and
 $-m \norm\mathfrak a  \in \mu \bar\mu +\mathbb Z_l$. So
 $$
 -m\norm \mathfrak a |D| \in \mu \bar\mu |D|  +\mathbb
 Z_l |D|
 $$
 and  $\mu \bar\mu |D| \in \mathbb Z_l^*$
(note that $l\ne 2$, since $D$ is odd). We may assume that $\mathfrak a$ is prime to $\partial$,
  so $\norm\mathfrak a$ does not interfere here. Hence
 $$
\xi_l (\frac{m|D|}p \norm(\mathfrak p_0 \partial \bar{\mathfrak a}))
=(D, -m \norm\mathfrak a |D|)_l =(D, \mu \bar\mu |D|)_l=1.$$ That is
$\rho_l(m|D|/p, [[\mathfrak p_0
\partial \bar{\mathfrak a}]])=1$ when $\mu_l\ne 0$. When
$\mu_l=0$, (\ref{Yeq1.8}) implies that
$$
\rho_l(m|D|/p,  [[\mathfrak p_0 \partial \bar{\mathfrak a}]])
=\frac{1}{2} (1+ \chi_l(-m \norm\mathfrak a)).
$$
This proves (\ref{Yeq1.6}). The verification of $(\ref{Yeq1.7})$ is
the same plus the fact $\rho(m|D|) =\rho(m|D|/p)$ for
$p|\hbox{gcd}(m|D|, |D|)$.
\end{proof}

Notice that  the $L$-function $L(\xi(f), U,s)$ vanishes identically,
since it is given as the Petersson scalar product of a cusp form and
an Eisenstein series. The lattice $N$ is equal to $L$. So Conjecture
\ref{conjY5.2} is simply the following theorem in our special case.

\begin{theorem}
\label{thm:n=0} Let $f\in H_{1,\bar\rho_L}$ and assume that the
constant term $c^+(0,0)$ of $f$ vanishes. Then
\begin{align*}
\widehat{\deg}(\calZ(f))=-\frac{1}2 \Phi(Z(U),f).
\end{align*}
\end{theorem}

\begin{proof}
Since
$$
\calZ(f) =\sum_{\mu \in L'/L} \sum_{m>0} c^+(-m, \mu) \calZ(m,
\mathfrak a, \mu),
$$
one has by Theorem \ref{theoY2.1} that
$$
\widehat{\deg}(\calZ(f)) =-\frac{2}{\vol(K_T)} \sum_{\mu \in L'/L}
\sum_{m>0} c^+(-m, \mu) \kappa(m,\mu).
$$
On the other hand, Theorem \ref{theoY4.6} asserts in this case
\begin{align*}
 \Phi(Z(U), f)
 &=\frac{4}{\vol(K_T)} \CT\left(\langle f^+(\tau),\,
\calE_L(\tau)\rangle\right)=\frac{4}{\vol(K_T)} \sum_{\mu \in
L'/L} \sum_{m>0} c^+(-m, \mu) \kappa(m,\mu).
\end{align*}
Comparing the two equalities, we obtain the assertion.
\end{proof}

\section{Height pairings on modular curves}
\label{sect7}
\label{sect:5}

Throughout this section we  assume that $(V,Q)$ has signature
$(1,2)$. Then $X_K$ is a modular or a Shimura curve defined over $\Q$.
The Heegner
divisors $Z(m,\mu)$ and the CM cycles are both divisors on $X_K$
(both supported on CM points).
Moreover, the Faltings height pairing is closely related to the
Neron-Tate height pairing.
Here we compute the heights of Heegner divisors
 employing Theorem \ref{thm:fund},
modularity of the generating series of Heegner divisors,
and
multiplicity one for  newforms in $S_{3/2,\rho_L}$.
Another crucial ingredient is the non-vanishing result for
coefficients of weight $2$ Jacobi cusp forms by Bump, Friedberg,
and Hoffstein \cite{BFH}.
This leads to a proof of the Gross-Zagier formula which uses minimal information
on the intersections of Heegner divisors at the finite places.
Moreover, we also prove Conjectures \ref{conj4.8} and \ref{conjY5.2} by pulling back Heegner divisors to the moduli space $\calC$ defined in Section \ref{sect6}.

\subsection{The modular curve $X_0(N)$}
\label{sect:5.1}

In this example we chose $L$ such that $X_K=Y_0(N)$. Then the
compactification of $X_K$ by the cusps is isomorphic to the modular
curve $X_0(N)$. The basic setup is the same as in \cite[Section
2.4]{BO} with the difference that the quadratic form is replaced by its
negative (which is slightly more convenient for the present paper).

Let $N$ be a positive integer. We consider the rational quadratic
space
\begin{align}
\label{defV} V:=\{x\in \Mat_2(\Q);\; \tr(x)=0\}
\end{align}
with the quadratic form $Q(x):=N\det(x)$. The corresponding bilinear
form is given by $(x,y)=-N\tr(xy)$ for $x,y\in V$. The signature of
$V$ is $(1,2)$.
The group $\Gl_2(\Q)$ acts on $V$ by conjugation
\[
\gamma.x=\gamma x\gamma^{-1},\qquad \gamma\in \Gl_2(\Q),
\]
leaving the quadratic form invariant. This induces an isomorphism
$H=\GSpin(V)\cong \Gl_2$. The domain $\D$ can be identified with
$\H\cup \bar\H$ via
\begin{equation}
\label{eqY7.2}
z=x+i y  \mapsto
\R\Re \zxz {z} {-z^2} {1} {-z}+ \R\Im\zxz {z} {-z^2}
{1} {-z} \in \D.
\end{equation}
Under this identification, the action of
$H(\mathbb R)$ on $\mathbb D$ becomes the usual linear fractional
action.

 Let $L$ be the lattice
\begin{align}
\label{latticeN}
 L=\left\{\zxz{b}{-a/N}{c}{-b};\quad a,b,c\in
\Z\right\}.
\end{align}
The dual lattice is given by
\begin{align}
\label{latticeN2}
 L'=\left\{\zxz{b/2N}{-a/N}{c}{-b/2N};\quad
\text{$a,b,c\in \Z$} \right\}.
\end{align}
We frequently identify $\Z/2N\Z$ with $L'/L$ via $r \mapsto
\mu_r=\diag(r/2N, -r/2N)$. Here the quadratic form on $L'/L$ is
identified with the quadratic form $x\mapsto -x^2$ on $\Z/4N\Z$. The
level of $L$ is $4N$.
%
%
For $m\in \Q$ and $\mu\in L'/L$, we define
\[
L_{m,\mu}:=\{ x\in \mu+L;\; Q(x)=m\}.
\]
Notice that $L_{m,\mu}$ is empty unless $Q(\mu)\equiv m \pmod{1}$.

Let $K_p\subset H(\Q_p)$  be the compact open subgroup
\[
K_p=\left\{\abcd\in \Gl_2(\Z_p);\;c\in N \Z_p\right\} ,
\]
and let $K=\prod_{p} K_p\subset H(\A_f)$. Then $K$ takes the lattice
$L$ to itself and acts trivially on the discriminant group $L'/L$.
Since $H(\A_f)=H(\Q) K$, it is easily seen that
$$
\alpha: \Gamma_0(N)\backslash  \mathbb H \rightarrow X_K= H(\Q)
\backslash \mathbb D \times H(\A_f) /K,  \quad \Gamma_0(N) z \mapsto
H(\Q) (z, 1) K
$$
is an isomorphism.


Let $m\in \Q_{>0}$ and let $\mu\in L'/L$ such that $Q(\mu) \equiv m
\pmod{1}$. Then $D:=-4Nm\in \Z$ is a negative discriminant. If $r
\in \mathbb Z$ with $\mu =\mu_r \pmod{L}$, then $D \equiv r^2 \pmod{
4N}$, and
\begin{align}
\label{eq:standard} x= \zxz{\frac{ r}{2N}
}{\frac{1}{N}}{\frac{D-r^2}{4N}}{-\frac{ r}{2N}}\in L_{m,\mu}.
\end{align}
Conversely, for a pair of integers $D<0$ and $r$ with $D\equiv r^2
\pmod{4N}$, let $m =-D/4N$ and $\mu =\mu_r$. Then $m \in Q(\mu)
+\mathbb Z$ is positive.
We will use this correspondence in this section freely
without mentioning it. Moreover, it is easy to check from Lemma
\ref{lemY4.1} that
\begin{equation}
Z(m, \mu) = P_{D, r} +  P_{D, -r}
\end{equation}
where $P_{D, r}$ is the Heegner divisor defined in \cite{GKZ}.

%
For a positive norm vector $x$ as in \eqref{eq:standard} we put
\begin{align}
\label{eq:defu}
V_+&=\Q x,&  U&=V\cap x^\perp , \\
\label{eq:defp} \calP&=  L\cap V_+, & \calN&= L\cap U.
\end{align}
Then $V_+$ is a positive definite line and $U$ is a $2$-dimensional
negative definite subspace in $V$. Here we use $\calN$ instead of
$N$ as in the previous section to avoid confusion with the level
$N$. An easy computation gives
\begin{equation} \label{eqnew7.9}
\mathcal N = \mathbb Z \zxz {1} {0} {-r} {-1} \oplus \mathbb Z \zxz
{0} {1/N} {\frac{r^2-D}{4N}} {0}.
\end{equation}
In particular,
 the determinant of $\calN$ is $-D$.
It is  also easy to check that
\begin{equation}
\calP= \mathbb Z \kzxz {r} {2} {\frac{D-r^2}2} {-r}=\mathbb Z
\frac{2N}{t} x, \quad \calP'= \mathbb Z \frac{t}{D} x.
\end{equation}
with $t =\gcd(r, 2N)$.
We consider the  ideal $\mathfrak n=[N, \frac{r+\sqrt D}2]$ of $\calO_D=\mathbb Z[\frac{D+\sqrt D}2]$. The norm of $\frakn$ is equal to $N$.
We define a quadratic form $Q$ on $\frakn$
via
\begin{equation}
Q(z) = -\frac{z \bar z}{N}=-\frac{\norm(z)}{\norm(\mathfrak n)}.
\end{equation}

\begin{lemma}
\label{lemY2.2}
Assume that $D$ is the fundamental discriminant of $\kay=\mathbb Q(\sqrt D)$. Then the following map  gives an isomorphism of
quadratic lattices:
$$
f: (\frakn, Q) \rightarrow (\calN, Q), \quad x N + y \frac{r+\sqrt
D}2 \mapsto \{x, y\}:= \kzxz {x} {-\frac{y}N} {-r x -y
\frac{r^2-D}{4N}} {-x}.
$$
Moreover, both are equivalent to the integral quadratic form $[-N,
-r, -\frac{r^2-D}{4N}]= -N x^2 -r xy -\frac{r^2-D}{4N} y^2$.
\end{lemma}
\begin{proof} Clearly,
$$
Q(x N +y \frac{r+\sqrt D}2) =-\frac{1}N \left( (xN +\frac{yr}2)^2
-\frac{D}4 y^2\right) = -N x^2 -r x y -\frac{r^2-D}4 y^2,
$$
so $(\frakn, Q)$ is equivalent to $[-N, -r, -\frac{r^2-D}4]$. On the
other hand,
by definition we have
$$
Q(\{x, y\}) = N \det \kzxz {x} {-\frac{y}N} {-r x - y
\frac{r^2-D}{4N}} {-x}=-Nx^2 -r xy -\frac{r^2-D}{4N} y^2.
$$
By means of (\ref{eqnew7.9}), one sees then  that  $\calN$ is equivalent to
$[-N, -r, -\frac{r^2-D}4]$, too. This proves the lemma.
\end{proof}

By Lemma \ref{lemY2.2}, we see that $T=\Gspin(U) \cong \kay^*$ with
$\kay=\mathbb Q(\sqrt D)$. It is easily checked that   $K_T\cong
\hat{\OO}_\kay^*$.

\begin{proposition}
\label{cor:zcomp}
Assume that $D$ is a fundamental discriminant
coprime to $N$. Then
\[
Z(U)=Z(m,\mu).
\]
\end{proposition}

\begin{proof}
We claim that under the assumption on $D$ we have
\begin{align}
\label{eq:hyp} \Omega_m(\A_f)\cap \supp(\phi_\mu)=K x.
\end{align}
Then the assertion follows directly from the definitions of the
cycles \eqref{eq:zvarphi} and \eqref{eq:zu}. To prove the claim, we
have to show for all primes $p$ that $\Omega_m(\Q_p)\cap
\supp(\phi_\mu)=K_p x$. This is a direct computation which we omit.
\end{proof}

\subsection{The Shimura lifting and Hecke eigenforms}

Recall from \cite[\S5]{EZ}  that for the lattice $L$ (defined in
Section \ref{sect:5.1}) the space of cusp forms $S_{3/2,\rho_L}$ is
isomorphic to the space $J_{2,N}$ of Jacobi forms of weight $2$ and
index $N$. There is a Hecke theory and a newform theory for
$J_{2,N}$ which give rise to the corresponding notions on
$S_{3/2,\rho_L}$. Let $S^{-}_2(N)$ denote the space of cusp forms of
weight $2$ for $\Gamma_0(N)$ which are invariant under the Fricke
involution. Note that the Hecke $L$-function of any $G\in S^{-}_2(N)
$ satisfies a functional equation with root number $-1$ and
therefore vanishes at the central critical point. According to
\cite{SZ}, the subspace of newforms $J^{new}_{2,N}$ of $J_{2,N}$ is
isomorphic to the subspace of newforms $S^{new,-}_2(N)$ of
$S^{-}_2(N)$ as a module over the Hecke algebra. The isomorphism is
given by the Shimura correspondence.

More precisely, let $m_0\in \Q_{>0}$ and $\mu_0\in L'/L$ such that
$m_0\equiv Q(\mu_0)\pmod{1}$. Assume that $D_0:=-4Nm_0\in \Z$ is a
fundamental discriminant. Let $x\in L_{m_0,\mu_0}$  be as in
(\ref{eq:standard}) and let $U$ be defined by \eqref{eq:defu}.
There is a linear map $\calS_{m_0,\mu_0}:S_{3/2,\rho_L}\to S_2(N)$
defined by
\begin{align} \label{eqY7.12}
g=\sum_\mu \sum_{m>0} b(m,\mu) q^m\phi_\mu \mapsto
\calS_{m_0,\mu_0}(g)= \sum_{n=1}^\infty \sum_{d\mid n}
\left(\frac{D_0}{d}\right) b\left(m_0 \frac{n^2}{d^2},
\mu_0\frac{n}{d} \right) q^n,
\end{align}
see \cite[ Section II.3]{GKZ}, or \cite[Section 2]{Sk}. If we denote
the Fourier coefficients of $\calS_{m_0,\mu_0}(g)$ by $B(n)$, then
we may rewrite the formula for the image as the Dirichlet series
identity
\begin{align}
\label{eq:dirid} L\left(\calS_{m_0,\mu_0}(g), s\right)= \sum_{n>0}
B(n)n^{-s} = L(\chi_{D_0},s)\cdot \sum_{n>0} b\left(m_0 n^2, \mu_0
n\right)
 n^{-s}.
\end{align}
The maps $\calS_{m_0,\mu_0}$ are Hecke-equivariant and there is a
linear combination of them which provides the above isomorphism of
$S_{3/2,\rho_L}^{new}$ and $S^{new,-}_2(N)$. Notice that if $g\in
S^{new}_{3/2,\rho_L}$ is a newform that corresponds to the
normalized newform $G\in S_2^{new,-}(N)$ under the Shimura
correspondence, then
\begin{align}
\label{eq:newform}
 L\left(\calS_{m_0,\mu_0}(g), s\right) =
b\left(m_0,\mu_0\right)\cdot L(G,s).
\end{align}

\begin{lemma}
\label{prop:l1} Let $m_0$, $\mu_0$, $D_0$, $U$ be as above.
If $g\in S_{3/2,\rho_L}$, then
\[
L(g, U, s)= 2^{-s}\left( \pi
m_0\right)^{-(s+1)/2}\Gamma\left(\frac{s+1}{2}\right)
L(\chi_{D_0},s+1)^{-1} L\big(\calS_{m_0,\mu_0}(g),s+1\big).
\]
In particular,
\begin{align*}
L'(g, U, 0) &=\frac{ \sqrt N\vol(K_T)}{\pi} b(m_0, \mu_0) L'(G,
1),
\end{align*}
 if $g \in S^{new}_{3/2, \rho_{L}}$ and $G\in S_2^{new, -}(N)$ are
 further related by (\ref{eq:newform}).
\end{lemma}

\begin{proof}
In view of \eqref{eq:luf} we have
\[
L(g, U, s)=\left(4\pi
\right)^{-(s+1)/2}\Gamma\left(\frac{s+1}{2}\right) \sum_{\lambda\in
\calP'} b(Q(\lambda), \lambda)  Q(\lambda)^{-(s+1)/2},
\]
where we view $g$ as a modular form with representation
$\rho_{\calP\oplus \calN}$ via Lemma \ref{sublattice}. Using the fact that
$b(Q(\lambda), \lambda)=0$ for $\lambda \in \calP'$ unless
$\lambda\in  \calP'\cap L'=\Z x $, the assertion follows by a straightforward computation.
\end{proof}

 Let $G\in
S_{2}^{new,-}(N)$ be a normalized newform of weight $2$, and write
$F_G$ for the totally real number field generated by the eigenvalues
of $G$. There is a newform $g\in S_{3/2,\rho_L}^{new}$ mapping to
$G$ under the Shimura correspondence. We normalize $g$ such that all
its coefficients $b(m,\mu)$ are contained in $F_G$.

\begin{lemma}
\label{goodf}
There is a $f \in H_{1/2,\bar\rho_L}$ with Fourier
coefficients $c^\pm(m,\mu)$ such that
\begin{enumerate}
\item[(i)]
$\xi(f)=\| g\|^{-2} g$,
\item[(ii)]
the coefficients of the principal part $P_f$ lie in $F_G$,
\item[(iii)]
the constant term $c^+(0,0)$ vanishes.
\end{enumerate}
\end{lemma}


\begin{proof}
The existence of an
$f\in
H_{1/2,\bar\rho_L}$ satisfying (i) and (ii) follows from Lemma 7.3
in \cite{BO}.
We may in addition attain (iii) by adding a suitable multiple of the
theta series in $M_{1/2,\bar\rho_L}$ for the lattice $\Z$ with the
quadratic form $x\mapsto Nx^2$.
\end{proof}

\begin{lemma}
\label{lemY7.8}
\label{lem:non-vanishing}
Let $g\in S_{3/2,\rho_L}^{new}$ be a newform with Fourier coefficients
$b(m,\mu)$
as above.  Let $S$ be a finite set of primes including all
those dividing $N$.  There exist infinitely many fundamental
discriminants $D<0$ such that
\begin{itemize}
\item[(i)]
$q$ splits in $\Q(\sqrt{D})$ for all primes $q\in S$,
\item[(ii)]
$b(m,\mu)\neq 0$ for $m=-\frac{D}{4N}$ and any $\mu\in L'/L$
such that $m\equiv Q(\mu)\pmod{1}$.
\end{itemize}
\end{lemma}

\begin{proof}
This is a consequence of the non-vanishing theorem for the central critical values of quadratic twists of Hecke $L$-functions proved in \cite{BFH}, together with the Waldspurger type formula for Jacobi forms,
see \cite[Chapter II, Corollary 1]{GKZ} and \cite{Sk}.
\end{proof}
An alternative proof could probably be given by employing the relationship between vector valued modular forms (respectively Jacobi forms) and scalar valued modular forms and using the non-vanishing result proved in \cite{BrHmbg}.

\subsection{The Gross-Zagier Formula}

\label{sect:gz}

Let $\mathcal Y_0(N)$ (respectively $\calX_0(N)$) be the moduli stack
over $\mathbb Z$ of cyclic isogenies of degree $N$ of elliptic
curves (respectively generalized elliptic curves) $\pi: E \rightarrow
E'$ such that $\ker \pi$ meets every irreducible component of each
geometric fiber as in \cite{KM}.
Then ${\calX}_0(N)(\mathbb C) =X_0(N)$. The stack ${\mathcal X}_0(N)$
is a proper flat curve over $\Z$. It is smooth over $\Z[1/N]$ and
regular except at closed supersingular points $\underline{x}$ in
characteristic $p$ dividing $N$ where $\Aut(\underline{x})\neq \{\pm
1\}$ (see \cite[Chapter 3, Proposition 1.4]{GZ}).

Let  $\mathcal Z(m, \mu)$ be the DM-stack representing  the moduli
problem which assigns to  a base scheme $S$ over $\mathbb Z$ the set
of  pairs $(\pi: E \rightarrow E',\iota)$ where
\begin{itemize}
\item[(i)]
$\pi: E \rightarrow E'$ is a cyclic isogeny of two elliptic curves $E$ and
$E'$ over $S$ of degree $N$,
\item[(ii)]
$\iota: \OO_D \hookrightarrow \End(\pi)=\{
\alpha \in \End(E);\;  \pi  \alpha \pi^{-1} \in \End(E')\}
$
is an $\OO_D$ action on $\pi$ such that
$\iota(\mathfrak n) \ker \pi =0$.
\end{itemize}
Here $\mathfrak n=[N, \frac{r+\sqrt D}2]$ is  one ideal  of
$k=\mathbb Q(\sqrt D)$ above $N$ and $\mu_r =\mu$ (recall that
$D=-4Nm$ and $\mu_r =\diag (\frac{r}{2N}, -\frac{r}{2N})$).
Moreover, $\OO_D$ denotes the order of discriminant $D$ in $k$.

 The  forgetful map $(\pi: E \rightarrow E',
\iota) \mapsto (\pi:E \rightarrow E')$ is a finite \'etale map from
$\mathcal Z(m, \mu)$ into $\mathcal Y_0(N)$, which is generically
$2$ to $1$,
and its direct image is the flat closure
of $Z(m, \mu)$ in $\mathcal X_0(N)$.
It does not intersect with the
boundary $\mathcal X_0(N)\bs \mathcal Y_0(N)$, and lies in the
regular locus of $\mathcal X_0(N)$ (see \cite[Lemma 2.2 and Remark
2.3]{Co}).
In
particular, we may use intersection theory for these divisors and
for cuspidal divisors on $\mathcal X_0(N)$ even though $\mathcal
X_0(N)$ is not regular.

Let $f\in H_{1/2,\bar\rho_L}$, and denote the Fourier expansion of
$f$ as in \eqref{deff}. Assume that the principal part of $f$ has
coefficients in $\R$
and that $c^+(0,0)=0$. There is a divisor
$C(f)$ on $X_0(N)$ supported at the cusps such that $\Phi(z,h,f)$ is
a Green function for the divisor
\[
Z^c(f)=Z(f)+C(f)
\]
of degree $0$ on $X_0(N)$. Let $\calZ^c(f)$ be the flat  closure of
$Z^c(f)$  in $\mathcal X_0(N)$.
We write $\hat{\calZ}^c(f)$ for the
arithmetic divisor given by the pair
\[
\big(\calZ^c(f), \Phi(\cdot,f)\big)\in \cha^1(\mathcal
X_0(N))_\R.
\]
For $m\in \Q_{>0}$ and $\mu\in L'/L$
we define
\begin{align}
y(m,\mu)=Z(m,\mu)-\frac{\deg Z(m,\mu)}{2}( (\infty)+(0)).
\end{align}
This divisor has degree $0$ and is invariant under the Fricke
involution. Moreover, using the principal part of the weak Maass
form $f$, we put
\begin{align}
y(f)=\sum_{\mu\in L'/L}\sum_{m>0} c^+ (-m,\mu) y(m,\mu).
\end{align}
We let $\mathcal Y(m, \mu)$ and $\mathcal Y(f)$ denote their
flat closures in $\mathcal X_0(N)$. Note that for primes $p$ not dividing the discriminant $D=-4Nm$, the divisor $\mathcal Y(m,\mu)$ has zero intersection with every fibral component of $\calX_0(N)$ over $\F_p$, see e.g.~\cite[Chapter IV.4, Proposition 1]{GKZ}.

Let
$J=J_0(N)$ be the Jacobian of $X_0(N)$, and let $J(F)$ denote its
points over any number field $F$. They correspond to divisor classes
of degree zero on $X_0(N)$ which are rational over $F$.
Note that  $y(f)$ is a divisor
of degree $0$ which differs from $Z^c(f)$ by a divisor of degree
zero on $X_0(N)$ which is supported at the cusps. By the
Manin-Drinfeld theorem, $Z^c(f)$ and $y(f)$ define the same point in
the Mordell-Weil space $J(\Q)\otimes_\Z\C$.

We now fix some notation for the rest of this subsection.
Let $G\in S_2^{new,-}(N)$ be a normalized newform defined over the
number field $F_G$. Let $g\in S_{3/2,\rho_L}^{new}$ be a cusp form
corresponding to $G$ under the Shimura correspondence with
coefficients $b(m,\mu)\in F_G$.  Let $f\in H_{1/2,\bar\rho_L}$ be a harmonic
weak Maass form as in
Lemma \ref{goodf}.

We now consider the generating series
\begin{equation}
A(\tau) =
\sum_{\mu\in L'/L} \sum_{m>0
} y(m,\mu)q^m\phi_\mu.
\end{equation}
By the Gross-Kohnen-Zagier theorem, $A(\tau)$ is a modular form with
values in $J(\Q)\otimes_\Z\C$. Borcherds gave a different proof for this
result using Borcherds products associated to weakly holomorphic
modular forms in $M_{1/2,\bar\rho_L}^!$, see \cite{Bo2}.

We may look at
the projection $A^G(\tau)$  of $A(\tau)$ to the $G$-isotypical
component of $J(\Q)\otimes_\Z\C$. So the coefficients of $A^G(\tau)$
are the projections $y^G(m,\mu)$ of the Heegner divisors $y(m,\mu)$
to the $G$-isotypical component.
\cite[Theorem 7.7]{BO} describes this generating series as follows.

\begin{theorem}
\label{thm:isogen}
Let $f$, $g$, and $G$ be as above.
We have
the identity
\[
A^G(\tau)= g(\tau)\otimes y(f)\in S_{3/2,\rho_L}\otimes_\Z J(\Q).
\]
In particular, the divisor $y(f)$ lies in the $G$-isotypical
component of  $J(\Q)\otimes_\Z\C$.
\end{theorem}
The proof is based on a comparison of the action of the Hecke
algebra on the Jacobian and on harmonic weak Maass forms, and on
multiplicity one for the space $ S_{3/2,\rho_L}^{new}$.


\begin{theorem} \label{theoY7.8}
\label{GZ1} Let $G$ be a normalized  cuspidal new form of weight
$2$, level $N$ whose $L$-function has an odd functional equation.
Let
$f$ and $g$ be associated to $G$ as above.
Then  the Neron-Tate height of $y(f)$ is given by
\[
\langle y(f), y(f)\rangle_{NT}=  \frac{2\sqrt{N}}{\pi\|g\|^2}
L'\big(G,1).
\]
\end{theorem}

\begin{proof}
According to Theorem
\ref{thm:isogen}, we have $b(m,\mu) y (f)=y^G(m,\mu)$, and therefore
\begin{align*}
\langle y(f), y(f) \rangle_{NT} b(m,\mu)
&=\langle y(f),
y(m,\mu)\rangle_{NT}\\
&=\langle Z^c(f),
y(m,\mu)\rangle_{NT}
\end{align*}
for all $(m, \mu)$. Here we have also used the Manin-Drinfeld theorem.
Set $d(m, \mu)=\deg Z(m, \mu)$. For two pairs $(m_0,\mu_0)$ and $(m_1,\mu_1)$ which we will specify appropriately later, we put
$$
c=c(m_0, m_1, \mu_0, \mu_1) = d(m_1, \mu_1) b(m_0, \mu_0) - d(m_0,
\mu_0) b(m_1, \mu_1).
$$
We consider the degree zero divisor
\begin{align*}
Z&=
d(m_1, \mu_1) y(m_0, \mu_0) -d(m_0, \mu_0) y(m_1,
\mu_1)\\
& =d(m_1, \mu_1) Z(m_0, \mu_0) -d(m_0, \mu_0) Z(m_1, \mu_1).
\end{align*}
%
on $X_0(N)$. This divisor is supported outside the cusps.
Let $M$ be the least common multiple of the discriminants of the Heegner divisors in the support of $Z(f)$.
We assume that  $D_i=-4Nm_i$ is coprime to $MN$. This implies that $Z^c(f)$ and $Z$ are relatively prime. Moreover, it implies
that for every prime $p$, the divisor $\calZ^c(f)$ or the flat closure of $Z$ has zero intersection with every fibral component of $\calX_0(N)$ over $\F_p$.
By means of
\cite[Section~3]{Gr}, we find
\begin{align*}
c \langle y(f) , y(f) \rangle_{NT}
 &=\langle Z^c(f), d(m_1, \mu_1) Z(m_0, \mu_0)-d(m_0, \mu_0)  Z(m_1, \mu_1)\rangle_{NT}
   \\
   &=d(m_1, \mu_1)\langle \hat{\mathcal Z}^c(f), \calZ(m_0, \mu_0)\rangle_{Fal}
   -d(m_0, \mu_0) \langle \hat{\mathcal Z}^c(f), \calZ(m_1, \mu_1)
   \rangle_{Fal}.
\end{align*}
Notice that   $\mathcal C(f)$,   the flat closure in $\mathcal X_0(N)$ of the cuspidal
part $C(f)$,   lies in the cuspidal part of
$\mathcal X_0(N)$ and thus does not intersect with $\calZ(m, \mu)$.
One has by Theorem \ref{theoY4.6}, and Lemma \ref{prop:l1}:
\begin{align}
\label{eqY7.18}
&\langle \hat{\mathcal Z}^c(f), \calZ(m_0, \mu_0)\rangle_{Fal}
\\
&=\frac{1}{2} \Phi(Z(m_0, \mu_0), f)
  + \langle \calZ(f), \calZ(m_0, \mu_0)\rangle_{fin} +\langle \mathcal C(f),
  \calZ(m_0, \mu_0)\rangle_{fin} \notag
  \\
  &=\frac{2}{\vol(K_{T_0})}  L'( \xi(f),U_0, 0)
   + \frac{2}{\vol(K_{T_0})} \CT\langle f^+, \theta_{\calP_0}\otimes
   \calE_{\calN_0}\rangle  + \langle \calZ(f), \calZ(m_0,
   \mu_0)\rangle_{fin} \notag
   \\
   &=\frac{2\sqrt N}{\pi \|g\|^2} b(m_0, \mu_0) L'(G, 1)+\frac{2}{\vol(K_{T_0})} \CT\langle f^+, \theta_{\calP_0}\otimes
   \calE_{\calN_0}\rangle  + \langle \calZ(f), \calZ(m_0,
   \mu_0)\rangle_{fin}.  \notag
  \end{align}
Here  the subscript $0$ in $\calP_0$, $\calN_0$, $U_0$, and $T_0$ indicates
its relation to $D_0$.
So we see that
$$
c \langle y(f), y(f) \rangle_{NT} =c\frac{2\sqrt N}{\pi \|g\|^2}
L'(G, 1) + \sum_{\text{$p$ prime}} \alpha_p \log p
$$
with coefficients $\alpha_p \in F_G$.

We claim that we can choose $D_0$ and $D_1$ such that $c\neq 0$.
In fact, according to Lemma~\ref{lemY7.8}, we may fix a pair $(m_0,\mu_0)$ such that $D_0$ is coprime to $MN$ and such that $b(m_0,\mu_0)\neq 0$.
We let $(m_1,\mu_1)$ run through the pairs such that $D_1$ is is a square modulo $4N$ and coprime to $MN$.
By Siegel's lower bound for the class numbers we have for any $\eps>0$ that
\[
d(m_1,\mu_1)\gg_\eps m_1^{1/2-\eps}, \quad m_1\to \infty.
\]
On the other hand, by Iwaniec's bound for the coefficients of half integral weight modular forms as refined by Duke \cite{Iw}, \cite{Du}, we have
\[
b(m_1,\mu_1)\ll_\eps m_1^{1/2-1/28+\eps}, \quad m_1\to \infty.
\]
This implies that $c\neq 0 $ for $m_1$ sufficiently large.
Hence we find hat
\begin{equation}
\label{eqY7.19}
\langle y(f), y(f) \rangle_{NT} =\frac{2\sqrt N}{\pi \|g\|^2} L'(G,
1) + \sum_{p} \beta_p \log p
\end{equation}
for some coefficients $\beta_p \in F_G$ independent of all choices that we made above.

Now
we prove that $\beta_p=0$ for every $p$. Let $p$ be any fixed prime.
According to Lemma~\ref{lemY7.8}, we may fix a pair $(m_0,\mu_0)$ such that $D_0$ is coprime to $MN$, $p$ splits in $\Q(\sqrt{D_0})$, and such that $b(m_0,\mu_0)\neq 0$.
We let $(m_1,\mu_1)$ run through the pairs such that $D_1$ is is a square modulo $4N$ coprime to $MN$, and such that $p$ splits in $\Q(\sqrt{D_1})$.
As above, we have $c\neq 0$ when $m_1$ is sufficiently large.
In view of (\ref{eqY7.18}) and
(\ref{eqY7.19}), one has
$$
\beta_p =\frac{2d(m_1, \mu_1)}{c\vol(K_{T_0})} a_{0, p} - \frac{2d(m_0,
\mu_0)}{c\vol(K_{T_1})} a_{1, p}+\frac{d(m_1, \mu_1)}c b_{0, p} -
\frac{d(m_0, \mu_0)}c b_{1, p}.
$$
Here we write
$$
\CT\langle f^+, \theta_{\calP_i}\otimes
\calE_{\calN_i}\rangle =\sum_{\text{$q$ prime}} a_{i, q} \log q
$$
by Theorem \ref{theoY2.5}, and
$$
\langle \calZ(f), \calZ(m_i, \mu_i) \rangle_{fin} =\sum_{\text{$q$ prime}} b_{i, q}
\log q
$$
by definition. By Theorem \ref{theoY2.5}, one sees immediately that
$a_{i,p}=0$ since $p$ is split in $\kay_i=\mathbb Q(\sqrt{D_i})$. On
the other hand, if $x=(\pi: E \rightarrow E', \iota) \in \mathcal
Z(m_0, \mu_0)(\bar{\mathbb F}_p)$,
then $E$ and $E'$ are ordinary since $p$ is split in $\kay_0$. This
means that $\iota$ is an isomorphism. So  there is no action of
$\OO_{D}$ on $E$ if $D/D_0$ is not a square. This implies
$$
\langle \calZ(m, \mu), \calZ(m_0, \mu_0) \rangle_p=0
$$
if $m/m_0=D/D_0$ is not a square. Consequently,
$$
\langle \calZ(f), \calZ(m_0, \mu_0)\rangle_p=0,
$$
that is, $b_{0, p}=0$.
For the same reason, $b_{1, p}=0$ and thus $\beta_p=0$. This proves
the theorem.
\end{proof}

\begin{corollary} (Gross-Zagier formula \cite[Theorem I.6.3]{GZ})
\label{GZ2} For any  any $\mu\in L'/L$ and any positive $m\in Q(\mu)+\Z$ we have
\[
\langle y^G(m,\mu),y^G(m,\mu) \rangle_{NT} = \frac{\sqrt{|D|
}}{4\pi^2\|G\|^2} L(G,\chi_D,1) L'\big(G,1).
\]
Here $D=-4Nm$ and $\|G\|$ denotes the Petersson norm of $G$.
\end{corollary}


\begin{proof}
This follows from Theorem \ref{GZ1} using the fact that
$y^G(m,\mu)=b(m,\mu) y(f)$ and the Waldspurger type formula
\[
b(m,\mu)^2=\frac{\|g\|^2}{8\pi
\sqrt{N}\|G\|^2}\sqrt{|D|}L(G,\chi_D,1),
\]
see \cite[Chapter II, Corollary 1]{GKZ}, and \cite{Sk}. Here we
have also used the fact that the Petersson norm $\|g\|$ is equal to
$2N^{1/4}\|\phi\|$, where $\phi$ is the Jacobi form of weight $2$
corresponding to $g$ and $ \|\phi\|$ is its Petersson norm, see
\cite[Theorem 5.3]{EZ}. (Notice that a factor of $2$ is missing in \cite{EZ} which is due to the fact that the element $(-1,0)$ of the Jacobi group acts as $(\tau,z)\mapsto (\tau,-z)$ on $\H\times \C$.)
\end{proof}

\subsection{Pull-back of Heegner divisors}

We continue to use the notation of Section \ref{sect:5.1}. Given two
cycles $\calZ (m_i, \mu_i)$ in $\mathcal Y_0(N)$, let $D_i=-4Nm_i$
and $r_i \in \mathbb Z/2N\Z$ with $\mu_i =\mu_{r_i}$ as before. We
assume that $D_0$ is prime to $2N$ and is  fundamental, and that
$D_0D_1$ is not a square so that $\mathcal Z(m_0, \mu_0)$ and
$\mathcal Z(m_1, \mu_1)$ intersect properly.
%
In this setting,
Conjecture \ref{conj4.8} is just the following theorem.

\begin{theorem} \label{theoY7.11}
\label{thm:finite}
Under the above assumptions on $D_0$ and $D_1$, the finite
intersection pairing   $\langle
\calZ(m_1,\mu_1),\calZ(m_0,\mu_0)\rangle_{fin}$ is equal to the
$(m_1,\mu_1)$-th coefficient of $\frac{-2}{\vol(K_T)}
\theta_{\calP_0}(\tau)\otimes \calE_{\calN_0}(\tau)$. That is,
\[
\langle
\calZ(m_1,\mu_1),\calZ(m_0,\mu_0)\rangle_{fin}=-\frac{2}{\vol(K_T)}
\sum_{\substack{l\in \Q\\ lx_0\in \calP_0'\\l^2 m_0 \leq m_1}}
\sum_{\substack{\nu\in \calN_0'/\calN_0\\
\nu +lx_0\equiv \mu_1\;(L)}} \kappa(m_1-l^2m_0,\nu).
\]
\end{theorem}
In this subsection, we prove the result by pulling the intersection
back to the CM stack $\mathcal C$ studied in Section \ref{sect6} and
using Theorem \ref{theoY2.1}.
Let $\mathfrak n_i=[N,
\frac{r_i+\sqrt{D_i}}2]$. Let $\mathcal C$ be the moduli stack of CM
elliptic curves associated to the quadratic field $k_0=\mathbb
Q(\sqrt{D_0})$ defined in Section~\ref{sect6}. For a CM elliptic
curve $(E, \iota) \in \mathcal C(S)$, let $E_{\mathfrak n_0}
=E/E[\mathfrak n_0]$ and let $\pi: E\rightarrow E_{\mathfrak n_0}$
be the natural map. Write
$$
\OO_{E, \mathfrak n_0} =\End_S(\pi)=\{ \alpha \in \OO_E; \; \pi
\alpha \pi^{-1} \in \End_{S} (E_{\mathfrak n_0})\}.
$$
 The starting point is

\begin{lemma} \label{lemY7.12} There is a natural isomorphism of stacks
$$
j: \mathcal C \rightarrow \mathcal Z(m_0, \mu_0), \quad j(E, \iota)
= (\pi: E \rightarrow E_{\mathfrak n_0}, \iota).
$$
\end{lemma}
\begin{proof} Since $\iota(\mathfrak n_0) \ker \pi =\iota(\mathfrak
n_0) E[\mathfrak n_0]=0$, and $\iota(\OO_{D_0}) \subset \OO_{E,
\mathfrak n_0}$, one has for $(E, \iota) \in \mathcal C(S)$ that $j(E, \iota) \in \mathcal Z(m_0, \mu_0)(S)$. The map $j$ is obviously a
bijection. It is also easy to check that $\Aut_S(E, \iota) = \Aut_S(j(E,
\iota))$. So $j$ is an isomorphism.
\end{proof}

Combining this map with the natural map from $\mathcal Z(m_0, \mu_0)$
to $\mathcal X_0(N)$, we obtain a natural map from $\mathcal C$ to
$\mathcal X_0(N)$, still denoted by $j$. Its direct image is the
cycle $\mathcal Z(m_0, \mu_0)$. So
\begin{equation*}
\langle \mathcal Z(m_1, \mu_1), \mathcal Z(m_0, \mu_0) \rangle_{fin} =
\widehat{\deg} (j^* \mathcal Z(m_1, \mu_1)).
\end{equation*}

Looking at the fiber product diagram
\begin{equation*}
\xymatrix{
 j^*\mathcal Z(m_1, \mu_1) =\mathcal Z(m_1, \mu_1) \times_{\mathcal X_0(N)} \mathcal C \ar[r] \ar[d]  &\mathcal C\ar[d]
 \cr
 \mathcal Z(m_1, \mu_1)\ar[r] &\mathcal X_0(N)\cr},
 \end{equation*}
one sees that $j^*\mathcal Z(m_1, \mu_1)(S)$ consists of triples $(E, \iota,
\phi)$ where $(E, \iota)\in \mathcal C(S)$, and
\begin{equation*}
\phi: \, \OO_{D_1}   \hookrightarrow \OO_{E, \mathfrak n_0}
\end{equation*}
such that
\begin{equation*}
\phi(\mathfrak n_1) E[\mathfrak n_0] =0.
\end{equation*}

\begin{proposition}
\label{propY7.14}
One has
$$
\langle \mathcal Z(m_1, \mu_1), \mathcal Z(m_0, \mu_0)\rangle_{fin}
= -\frac{2}{\vol
(K_T)} \sum_{\substack{n \equiv r_0 r_1 \pmod{2N}\\
n^2 \le D_0 D_1}} \kappa(\frac{D_0D_1-n^2}{4N |D_0|},\frac{\tilde
 2 n }{ \sqrt{D_0}}).
$$
Here $\tilde 2  \in \mathbb Z/D_0\Z$ is determined by the condition
$2\cdot \tilde 2
\equiv 1 \pmod{D_0}$.
\end{proposition}

\begin{proof}
First we look at geometric points $(E, \iota, \phi) \in
j^* \mathcal Z(m_1, \mu_1)(F)$, with $F=\mathbb C$ or
$F=\bar{\mathbb F}_p$. Then  $\OO_{E, \mathfrak n_0}$ contains
$\iota( \OO_{\kay_0})$ and $\phi(\OO_{\kay_1})$, and is thus  at
least of rank four over $\mathbb Z$. This implies $p$ is non-split
in $k_i$, $i=0,1$, and $E$ is supersingular. Assuming this,  let
$\mathbb B$ be the quaternion algebra over $\mathbb Q$ ramified
exactly at $p$ and $\infty$, and let
$$
\iota_0:  \kay_0 \hookrightarrow \mathbb B
$$
be a fixed embedding. Choose a prime $p_0\nmid 2pD_0$ such that (as
in Section \ref{sect6})
$$
\hbox{inv}_l \mathbb B= \begin{cases}
    (D_0, -p_0 p)_l &\ff p \hbox{ inert in  } \kay_0,
    \\
     (D_0, -p_0)_l  &\ff p \hbox{ ramified in  } \kay_0
     \end{cases}
     $$
     for every prime $l$.
     In particular, $p_0=\mathfrak p_0 \overline{\mathfrak p}_0$ is
     split in $\kay_0$. Let $\kappa_\mathbb B = -p_0 p$ or $-p_0$
     depending on whether $p$ is inert or ramified in $\kay_0$, and
     let $\delta_\mathbb B\in \mathbb B^*$ such that
     $\delta^2=\kappa_\mathbb B$ and $\delta_\mathbb B \alpha
     =\bar\alpha \delta_\mathbb b$ for $\alpha \in \kay_0$. Here we
     identify $\alpha \in \kay_0$ with $\iota_0(\alpha) \in \mathbb
     B$.
Then $\OO_E=\End E$ is a maximal  order of $\mathbb B$. Write
\begin{equation}
\label{Yeq2.7}
\phi(\frac{r_1+ \sqrt{D_1}}2) =\alpha +\boldbeta \in \OO_{E,
\mathfrak n_0}
\end{equation}
with $\alpha \in \kay_0$ and $\boldbeta \in \delta_{\mathbb
B}\kay_0$. The condition $\phi(\mathfrak n_1) E[\mathfrak n_0]=0$ is
the same as
\begin{equation*}
\phi(\frac{r_1+ \sqrt{D_1}}2)
E[\mathfrak n_0] =0,
\end{equation*}
which is the same  as
\begin{equation*}
\alpha +\boldbeta \in \OO_{E} \mathfrak n_0.
\end{equation*}
 In particular, $\alpha \in \partial_0^{-1}
\mathfrak n_0$. One sees from (\ref{Yeq2.7}) that
$$
\phi(\sqrt{D_1}) = \alpha_1 + 2 \boldbeta
$$
with $\alpha_1 = -r_1 + 2 \alpha \in \partial_0^{-1}$ and $\tr
\alpha_1=0$.
We write
$$\alpha_1 =\frac{n}{\sqrt{D_0}}, \quad \alpha
=\frac{1}{\sqrt{D_0}}(a N + b \frac{r_0+\sqrt{D_0}}2).
$$
Then we
see
$$
n=-r_1 \sqrt{D_0} + 2a N + br_0 +b \sqrt{D_0}
$$
and thus $b=r_1$, and
$$n =2aN+r_0 r_1 \equiv r_0 r_1 \pmod{2N}.
$$
Moreover,
$$
D_1 = \alpha_1^2 -4 \norm(\boldbeta)=\frac{n^2}{D_0} -4 \norm(\boldbeta),
$$
and so
$$
\norm(\boldbeta) =\frac{D_0D_1 -n^2}{4|D_0|}=\frac{D_0D_1-n^2}{4N
|D_0|} \norm(\mathfrak n_0) \in \frac{1}{|D_0|} \mathbb Z_{>0}.
$$
This implies that
\begin{equation}
(E, \iota, \boldbeta) \in \mathcal Z(\frac{D_0D_1-n^2}{4N |D_0|},
\mathfrak n_0, \frac{n+ r_1\sqrt{D_0}}{2 \sqrt{D_0}})(\bar{\mathbb
F}_p).
\end{equation}

Conversely, if $(E, \iota, \boldbeta) \in \mathcal
Z(\frac{D_0D_1-n^2}{4N }, \mathfrak n_0, \frac{n+ r_1\sqrt{D_0}}{2
\sqrt{D_0}})(\bar{\mathbb F}_p)$ for some $n \equiv r_0 r_1 \pmod{2N}$, then
$$\boldbeta \in \delta_\mathbb B
\partial_0^{-1}\mathfrak n_0, \quad  \norm(\boldbeta) =
\frac{D_0D_1-n^2}{4N|D_0| }\norm(\mathfrak n_0)
$$
and
$$
\alpha + \boldbeta \in \OO_E \mathfrak n_0
$$
with $\alpha =\frac{n+r_1\sqrt{D_0}}{2 \sqrt{D_0}}$.
If we write $n = r_0
r_1 +2a N$,
then
$$
\alpha =\frac{n+ r_1\sqrt{D_0}}{2 \sqrt{D_0}}
=\frac{1}{\sqrt{D_0}}(aN +r_1 \frac{r_0 +\sqrt{D_0}}2) \in
\partial_0^{-1} \mathfrak n_0.
$$
So $\phi(\frac{r_1 +\sqrt{D_1}}2) =\alpha + \boldbeta \in
\OO_E\mathfrak n_0$ gives $(E, \iota, \phi) \in j^*\mathcal Z(m_1,
\mu_1)(\bar{\mathbb F}_p)$. Hence we have proved an isomorphism
\begin{equation}
j^*\mathcal Z(m_1, \mu_1) (\bar{\mathbb F}_p) \cong
\bigsqcup_{\substack{n \equiv r_0 r_1 \pmod{2N}\\
n^2 \le D_0 D_1}}\mathcal Z(\frac{D_0D_1-n^2}{4N |D_0|}, \mathfrak
n_0, \frac{n+ r_1\sqrt{D_0}}{2 \sqrt{D_0}})(\bar{\mathbb F}_p),
\end{equation}
given by $( E, \iota, \phi) \mapsto (E, \iota, \boldbeta)$ via the
relation
$$
\phi(\frac{r_1 +\sqrt{D_1}}2) =\frac{n+ r_1\sqrt{D_0}}{2 \sqrt{D_0}}
+\boldbeta.
$$
Let $W=W(\bar{\mathbb F}_p)$ be the Witt ring of $\bar{\mathbb
F}_p$. It is not hard to check that for any locally complete
$W$-algebra $R$ with residue field $\bar{\mathbb F}_p$, $(E, \iota,
\phi)$ lifts to an element in $j^*\mathcal Z(m_1, \mu_1)(R)$ if and
only if $(E, \iota, \boldbeta)$ lifts to an element in $\mathcal
Z(\frac{D_0D_1-n^2}{4N }, \mathfrak n_0, \frac{n+2 r_1\sqrt{D_0}}{2
\sqrt{D_0}})(R)$. So we have by Theorem \ref{theoY2.1} that
\begin{align*}
\langle \mathcal Z(m_1, \mu_1), \mathcal Z(m_0, \mu_0) \rangle
 &=\widehat{\deg} (j^*\mathcal Z(m_1, \mu_1))
 \\
 &= \sum_{\substack{n \equiv r_0 r_1 \pmod{2N}\\
n^2 \le D_0 D_1}} \widehat{\deg}\mathcal Z(\frac{D_0D_1-n^2}{4N
|D_0|}, \mathfrak n_0, \frac{n+ r_1\sqrt{D_0}}{2 \sqrt{D_0}})
\\
 &= -\frac{2}{\vol(K_T)} \sum_{\substack{n \equiv r_0 r_1 \pmod{2N}\\
n^2 \le D_0 D_1}} \kappa (\frac{D_0D_1-n^2}{4N |D_0|},\frac{n+
r_1\sqrt{D_0}}{2 \sqrt{D_0}}).
\end{align*}
Since $\frac{n+ r_1\sqrt{D_0}}{2 \sqrt{D_0}} \equiv \frac{\tilde 2
n}{\sqrt{D_0}} \pmod{\OO_{D_0}}$, this concludes the proof of the proposition.
\end{proof}

Now Theorem \ref{theoY7.11} follows from the above proposition and the following lemma.

\begin{lemma} Let the notation be as above. Then one has
\begin{equation}
\sum_{\substack{ l\in \mathbb Q \\ l x_0 \in \calP_0' \\ l^2 m_0 \le
m_1}} \sum_{\substack{ \nu \in \calN_0'/\calN_0 \\ \nu +  l x_0
\equiv \mu_1 \;(L)}} \kappa_\nu(m_1-l^2m_0) = \sum_{\substack{n\in
\mathbb Z, \, n^2 \le  D_0 D_1\\ n \equiv r_0 r_1 \;(2N)}}
\kappa(\frac{D_0D_1-n^2}{4 N |D_0|}, {\frac{\tilde 2
n}{\sqrt{D_0}}}).
\end{equation}
\end{lemma}

\begin{proof}
It is clear  that
$l x_0 \in \calP_0'$ if and only if $l =\frac{n}D$ with $n \in \mathbb
Z$. The inequality $l^2 m_0 \le m_1$ is the same as $n^2 \le D_0
D_1$. By Lemma \ref{lemY2.2}, one sees that
\begin{equation}
\label{eqY2.3}
\nu(a) =f(\frac{Na}{\sqrt D_0})=\frac{a}{D_0} \kzxz {-r_0} {-2}
{\frac{D_0+r_0^2}2} {r_0}, \quad  a \in \mathbb Z/D_0\Z
\end{equation}
gives a complete set of representatives of $\calN_0'/\calN_0$.
Write
$$
\nu(a) +l x_0 =  \kzxz {\frac{ur_0}{2N}} {\frac{u}N} {aD_0
+\frac{D_0-r_0^2}{4N} u} {- \frac{ur_0}{2N}}
$$
with  $ u =\frac{n-2N
a}{D_0}$. So $\nu(a) + l x_0 \in  \mu_1 +L $  if and only if
$$
u=\frac{n-2N a}{D_0} \in \mathbb Z, \quad \hbox{ and } \quad  ur_0
\equiv r_1 \pmod{2N}.
$$
Since $(D_0, 2N)=1$, and $D_0 \equiv r_0^2 \pmod{4N}$, one sees that
the above condition is equivalent to
$$
n \equiv 2 N a \pmod{D_0}, \quad  n\equiv r_0 r_1 \pmod{2N}.
$$
So $\nu(a) + l x_0 \in \mu_1 + L$  if and only if $n \equiv r_0
r_1 \pmod{2N}$ and $Na  \equiv \tilde 2 n \pmod{D_0}$. In such a case,
Lemma \ref{lemY2.2} implies
$$
\kappa(t, {\nu(a)}) = \kappa(t, {\frac{Na}{\sqrt {D_0}}})=\kappa
(t, {\frac{\tilde 2 n}{\sqrt{D_0}}}).
$$
Finally, one checks
$$
m_1 -l^2 m_0 =-\frac{D_1}{4N} +\frac{n^2}{D_0^2} \frac{D_0}{4N}
=\frac{D_0D_1-n^2}{4N|D_0|}.
$$
Putting this together, one proves the proposition.
\end{proof}

Now Conjecture \ref{conjY5.2} becomes the following theorem in our
setting.

\begin{theorem}
\label{thm:fh}

Assume that $D_0$ is a fundamental discriminant coprime to $2N$. Let $f$ be any element of
$
H_{1/2, \bar{\rho}_L}$.
Then
\begin{align}
\label{eq:7.15}
\langle \hat{\mathcal Z}^c(f),   \mathcal Z(m_0,\mu_0)
\rangle_{Fal}= \frac{2}{\vol(K_T)}\left( c^+(0,0)\kappa(0,0)+ L'( \xi(f),U, 0)\right).
\end{align}
\end{theorem}

\begin{proof}
We first assume that $\mathcal Z(f)$ and $\mathcal
Z(m_0, \mu_0)$ intersect properly.
According to Proposition \ref{cor:zcomp}, the assumption on $D_0$
implies that $Z(U)=Z(m_0,\mu_0)$. Hence Theorem \ref{thm:fund}
says that
\begin{align*}
\langle \hat{\mathcal Z}^c(f),   \mathcal Z(m_0,\mu_0) \rangle_\infty &=\frac{1}{2}\Phi(Z(U),f)\\
&=\frac{2}{\vol(K_T)}\left( \CT\left(\langle f,\,
\theta_{\calP_0}(\tau)\otimes \calE_{\calN_0}(\tau)\rangle\right) +
L'(
\xi(f),U,0)\right) .
\end{align*}
According to Theorem
\ref{thm:finite}, we have
\begin{align*}
\langle \mathcal Z^c(f),   \mathcal Z(m_0,\mu_0) \rangle_{fin} &=
\sum_{m>0, \mu \in L'/L}  c^+(-m, \mu) \langle \mathcal Z(m,\mu),
\mathcal Z(m_0,\mu_0) \rangle_{fin}\\
&= -\frac{2}{\vol(K_T)} \CT\left(\langle f,\,
\theta_{\calP_0}(\tau)\otimes \calE_{\calN_0}(\tau)\rangle\right).
\end{align*}
Adding the two identities together, we obtain the assertion in the case when $\mathcal Z(f)$ and $\mathcal
Z(m_0, \mu_0)$ intersect properly.
Finally, for general $f$, we notice that there always exists a weakly holomorphic modular form $f'\in M^!_{1/2,\bar\rho_L}$ with vanishing constant term
such that
$\mathcal Z(f+f')$ and $\mathcal
Z(m_0, \mu_0)$ intersect properly.
For $f'$ both sides of the claimed identity \eqref{eq:7.15} vanish. Hence, the general case follows from the linearity of
\eqref{eq:7.15} in $f$.
\end{proof}

Notice that Theorem \ref{thm:fh} and Lemma \ref{prop:l1} can be used to give another proof of
the Gross-Zagier  formula in Theorem \ref{GZ1}.

\section{The case $n=2$}

\label{sect:8}

In this section, we verify a very special case of Conjecture
\ref{conj4.8} when $n=2$.
We plan to study the case $n=2$
systematically in a sequel to this paper.

Let $F=\Q(\sqrt\Delta)$ be a real quadratic field with prime
discriminant $\Delta \equiv 1 \pmod 4$.  We denote by $\calO_F$ the ring of integers in $F$, and write $\partial_F$ for the different of $F$.
Let $V$ be the quadratic space
\begin{equation}
V =\{ A \in M_2(F);\; A'=A^t\} =\{ A=\kzxz {a} {\lambda} {\lambda'}
{b};\; a, b \in \Q,\; \lambda \in F \}
\end{equation}
with the quadratic form $Q(A) =\det A$, which has signature $(2,
2)$. We consider the even lattice
$L=V\cap M_2(\OO_F)$. The dual lattice is
$$
L'= \{ A = \kzxz {a} {\lambda} {\lambda'} {b};\;  a, b \in \mathbb
Z,\;  \lambda \in \partial_F^{-1}\}.
$$

In this case, $$H=\Gspin(V)=\{ g \in \GL_2(F); \; \det g \in \mathbb
Q^*\}$$  acts on $V$ via
$$
 g.A=\frac{1}{\det g} g
A\, {}^tg'.
$$
Take
$$
K = H (\hat{\mathbb Z}) =\{ g \in \GL_2(\hat{\OO}_F);\;  \det g \in
\hat{\mathbb Z}\}.
$$
The following identification is well-known
\begin{equation} \label{eqn=21.2}
(\mathbb H^\pm)^2 \rightarrow \mathbb D, \quad z=(z_1, z_2) \mapsto
U=\mathbb  R \kzxz {1} {x_1} {x_2} {x_1x_2 -y_1y_2} \oplus
    \mathbb R \kzxz {0} {-y_1} {-y_2} {-x_1y_2-x_2y_1}.
    \end{equation}
Since $H(\A_f) = H(\mathbb Q)^+ K$, one can show that
$$
X_K=H(\mathbb Q) \backslash \mathbb D\times H(\A_f)/K \cong
\SL_2(\OO_F)\backslash \mathbb H^2,
$$
which  we will denote simply by $X$ in this section.

\subsection{The CM cycle $Z(U)$}

There are many CM $0$-cycles $Z(U)$. Here we choose a special one
for simplicity.  Let $k_D=\mathbb Q(\sqrt D)$ be an imaginary
quadratic field with fundamental discriminant $D$ and assume
$(D, 2\Delta)=1$.
The oriented negative $2$-plane associated to
the CM point $z=(\frac{D+\sqrt D}2, \frac{D+\sqrt D}2)
\in X$  via (\ref{eqn=21.2})
is actually rational and is given by
\begin{equation}
U=\mathbb Q f_1 \oplus \mathbb Q f_2, \quad f_1 =\kzxz {0} {1} {1}
{D},  \quad f_2 = \kzxz {2} {D} {D} {\frac{D^2+D}2}.
\end{equation}
The lattice $N=U\cap L$ is isomorphic to
\begin{equation}
(N, Q) \cong (\OO_{D}, -\norm ), \quad  f_1 \mapsto 1,\quad  f_2 \mapsto
\sqrt D,
\end{equation}
where $\calO_D$ is the ring of integers in $k_D$.
It is easy to check that
\begin{align*}
V_+&=U^\perp= \mathbb Q e_1 \oplus \mathbb Q e_2,\\
P&=V_+ \cap L
=\mathbb Z e_1 \oplus \mathbb Z e_2 \cong (\mathfrak{\Delta},
 \frac{\norm}{\Delta}),
\end{align*}
where
$$
e_1=\kzxz {0} {\sqrt\Delta} {-\sqrt\Delta} {0}, \quad e_2 =\kzxz {1}
{\frac{D+\sqrt\Delta}2} {\frac{D-\sqrt\Delta}2}  {\frac{D^2-D}4},
$$
and $\mathfrak\Delta$ is the ideal of $k_{D\Delta}=\mathbb
Q(\sqrt{D\Delta})$ over $\Delta$. Let $P_i=\mathbb Z e_i$ and
$$
M=P_1^\perp \cap L =\{ A= \kzxz {a} {b} {b} {c};\; a, b, c \in
\mathbb Z\} \cong \{ A= \kzxz {b} {a} {c} {-b} ;  \; a , b, c \in
\mathbb Z \}.
$$
So the cycle $Z(e_1, 1)$ defined in (\ref{eqY4.3}) is naturally isomorphic
 to the modular curve $Y_0(1)$ defined in Section \ref{sect7}. The inclusion
$N \subset M \subset L$ gives rise to natural morphisms
\begin{equation}
\xymatrix{
  Z(U)\ar[r]^-{j_0} &Y_0(1) \ar[r]^-{j_1} &X.}
\end{equation}
In terms of coordinates in upper half planes, they are given by
$$
j_0( [z_U^\pm, h]) = \frac{b+\sqrt D}{2a},  \quad j_1(z) =(z, z),
$$
if $[a, \frac{b+\sqrt D}2]$ is the ideal of $k_D$ associated to
$h$. The morphism $j_0$ is two-to-one, and $j_1$ is an injection. It is not
hard to check  for a non-square  integer $m>0$ that
\begin{equation} \label{eq8.7}
j_1^*Z(m, \mu)  = \sum_{\substack{ \mu_1 \in P_1'/P_1, \, \mu_2
\in M'/M \\ \mu_1 + \mu_2 \equiv \mu \;(L) \\ m_1 +m_2 =m, m_i
\ge 0}} r_{P_1}(m_1, \mu_1)  Z(m_2, \mu_2)_{Y_0(1)}.
\end{equation}
Here we use the subscript $P_1$ to indicate the dependence of
Fourier coefficients $r_{P_1}(m, \mu)$ on $P_1$, and the subscript
$Y_0(1)$ to indicate the cycles in $Y_0(1)$. We remark that $Z(m,
\mu)$ is basically the Hirzebruch-Zagier divisor $T_{m\Delta}$.

\subsection{Integral Model.}

Let  $\mathcal X$ be the Hilbert moduli stack  assigning to a
base scheme  $S$ over $\mathbb Z$   the set of the triples $(A,
\iota, \lambda)$, where
\begin{enumerate}
\item[(i)]
$A$ is a abelian surface over $S$.
\item[(ii)]
$\iota: \OO_F \hookrightarrow \End_S(A)$ is real
multiplication of $\OO_F$ on $A$.
\item[(iii)]
$\lambda: \partial_F^{-1} \rightarrow P(A)=\Hom_{\OO_F}(A,
A^\vee)^{\hbox{sym}}$ is a $\partial_F^{-1}$-polarization (in the
sense of Deligne-Papas) satisfying the condition:
$$
\partial_F^{-1}\otimes A  \rightarrow A^\vee, \quad  r \otimes a
\mapsto \lambda(r)(a)
$$
is an isomorphism.
\end{enumerate}
(See \cite[Chapter 3]{Go} and \cite[Section 3]{Vo}.)
Then it is well-known that $\mathcal X(\mathbb C) =X$.   Let $\mathcal Z(m,
 \mu)$ be the flat closure of $Z(m, \mu)$ in $\mathcal
 X$, and let $\mathcal Z(U)$ be the flat closure of $Z(U)$ in
 $\mathcal X$. Let $\mathcal C$ and $\mathcal Y_0(1)$ be as in
 Sections \ref{sect6} and \ref{sect7}.
Let $j_0: \mathcal C \to \mathcal Y_0(1)$ be the map defined in Lemma
\ref{lemY7.12} (with abuse of notation).
The map
  $j_1$ extends integrally to a closed immersion  $j_1$ from $\mathcal
 Y_0(1)$ to $\mathcal X$ defined in \cite[Lemma 2.2]{Yam=1}. Let
 $j=j_1 \circ j_0$ be the map from $\mathcal C$ to $\mathcal X$. Then the direct image of $\mathcal C$ is
 $\mathcal Z(U)$, so $j$ can be viewed as the integral extension
of the map $j$ defined in (8.5).
Taking the flat closures on both sides of
 (\ref{eq8.7}), one sees that (\ref{eq8.7}) holds also integrally.
 So we have by Theorem~\ref{theoY7.11} and Lemma~\ref{lemY7.12} that
 \begin{align*}
 \langle \mathcal Z(U), \mathcal Z(m, \mu) \rangle_{\mathcal X}
  &=\langle  (j_{0})_*\mathcal C, j_1^*\mathcal Z(m, \mu) \rangle_{\mathcal Y_0(1)}
  \\
   &= \sum_{\substack{ \mu_1 \in P_1'/P_1,\, \mu_2 \in
M'/M \\ \mu_1 + \mu_2 \equiv \mu \pmod L \\ m_1 +m_2 =m,\, m_i \ge
0}} r_{P_1}(m_1, \mu_1) \langle \mathcal Z(-\frac{D}4, \frac{D}2 ),
\mathcal Z(m_2, \mu_2) \rangle_{\mathcal Y_0(1)}
\\
 &=c\sum_{\substack{ \mu_1 \in P_1'/P_1,\,
\mu_2 \in M'/M \\ \mu_1 + \mu_2 \equiv \mu \pmod L \\ m_1 +m_2 =m,\,
m_i \ge 0}} r_{P_1}(m_1, \mu_1) \sum_{\substack{ \mu_3 \in
P_2'/P_2,\, \mu_4 \in N'/N\\ \mu_3 + \mu_4 \equiv \mu \pmod M \\ m_3
+m_4 =m_2,\, m_i \ge 0}} r_{P_2}(m_3, \mu_3) \kappa_N(m_4, \mu_4)
\\
 &=c\sum_{\substack{ \mu_1 \in (P_1+P_2)'/(P_1+P_2),\,
\mu_2 \in N'/N \\ \mu_1 + \mu_2 \equiv \mu \pmod L \\ m_1 +m_2 =m,\,
m_i \ge 0}} r_{P_1+P_2}(m_1, \mu_1) \kappa_N(m_2, \mu_2).
\end{align*}
Here $c=-\frac{2}{\vol(K_T)} = -\frac{2h_D}{w_D}$ as in Lemma
\ref{lemY6.1}. Since
$$
P_1 \oplus P_2 \oplus N \subset P \oplus N
\subset L \subset L' \subset P' \oplus N' \subset (P_1 \oplus P_2)'
\oplus N',
$$
it is easy to see that for $\mu_1 \in (P_1 \oplus P_2)'$ and $\mu_2
\in N'$, the condition $\mu_1 + \mu_2 \in L'$ implies that $\mu_1 \in P'$. So we
 have proved Conjecture \ref{conj4.8} in this special
case, which we state as a theorem.

\begin{theorem}
\label{thm:n=2}
Let $F=\mathbb Q(\sqrt\Delta)$ be a real quadratic
field with prime discriminant $\Delta \equiv 1 \pmod{4}$, and let
$\mathcal X$ be the associated Hilbert modular surface. Let $U$ be
as above, and assume that $m >0$ is not a square.  Then
$$
\langle \mathcal Z(U), \mathcal Z(m, \mu) \rangle_{fin}
 = -\frac{2h_D}{w_D}\sum_{\substack{ \mu_1 \in P'/P,\,
\mu_2 \in N'/N \\ \mu_1 + \mu_2 \equiv \mu \pmod L \\ m_1 +m_2 =m,\,
m_i \ge 0}} r_{P}(m_1, \mu_1) \kappa_N(m_2, \mu_2)
$$
is $-\frac{2h_D}{w_D}$ times the $(m, \mu)$-th Fourier coefficient
of $\theta_P(\tau) \otimes  \mathcal E_N(\tau)$.
\end{theorem}

As discussed in Section \ref{sect:fal}, this implies Conjecture
\ref{conjY5.2}. We also remark that the $L$-series  $L(\xi(f), U,
s)$ is the Rankin-Selberg $L$-function of a cusp  form of weight
$2$, level $\Delta$ and {\it non-trivial} Nebentypus $\chi_{\Delta}$
with a theta function of weight $1$.  This is new in the sense that
it is associated to the Jacobian of $X_1(\Delta)$.


\begin{thebibliography}{[BCDT}
\bibitem[Bo1]{Bo1}
\emph{R. Borcherds}, Automorphic forms with singularities on
Grassmannians, Inv. Math. \textbf{132} (1998), 491--562.

\bibitem[Bo2]{Bo2}
{\em R. E. Borcherds}, The Gross-Kohnen-Zagier
  theorem in higher dimensions, Duke Math. J. {\bf 97} (1999),
  219--233.

\bibitem[BGS]{BGS} {\em J.-B. Bost, H. Gillet, and C. Soul\'e}, Heights of projective varieties and positive Green forms.  J. Amer. Math. Soc.  {\bf 7}  (1994), 903--1027.

\bibitem[BCDT]{BCDT} {\em C. Breuil, B. Conrad, F. Diamond, R. Taylor}, On the modularity of elliptic curves over $\Q$: wild 3-adic exercises,
J. Amer. Math. Soc. {\bf 14} (2001), 843--939.

\bibitem[Br1]{BrHmbg} \emph{J. H. Bruinier}, On a theorem of Vign\'eras, Abh. Math. Sem. Univ. Hamburg {\bf 68} (1998), 163--168.

\bibitem[Br2]{Br} \emph{J. H. Bruinier}, Borcherds products on
  $\Orth(2,l)$ and Chern classes of Heegner divisors,  Springer Lecture Notes in Mathematics {\bf 1780}, Springer-Verlag (2002).

\bibitem[Br3]{Brcurve} \emph{J. H. Bruinier}, Two applications of the curve lemma for orthogonal groups, Math. Nachr. {\bf 274--275} (2004), 19--31.

\bibitem[BF]{BF} {\em J. H. Bruinier and J. Funke}, On two geometric theta lifts,
Duke Math. Journal. {\bf 125} (2004), 45--90.

\bibitem[BK]{BK} \emph{J. H. Bruinier and U. K\"uhn},
Integrals of automorphic Green's functions associated to Heegner
divisors, Int. Math. Res. Not. {\bf 2003:31} (2003), 1687--1729.


\bibitem[BrO]{BO} \emph{J. H. Bruinier and K. Ono}, Heegner divisors, $L$-functions and harmonic weak Maass forms, preprint (2007).

\bibitem[BBK]{BBK} {\em J. H. Bruinier, J. Burgos, and U. K\"uhn}, Borcherds products and
arithmetic intersection theory on Hilbert modular surfaces, Duke
Math. J. {\bf 139} (2007), 1--88.

\bibitem[BKK]{BKK} {\em J. Burgos, J. Kramer, and U. K\"uhn},
Cohomological arithmetic Chow groups,  J. Inst. Math. Jussieu.
{\bf 6}, 1--178 (2007).

\bibitem[BFH]{BFH}
{\em D. Bump, S. Friedberg, and J. Hoffstein},
Nonvanishing theorems for $L$-functions of modular forms and their derivatives.
Invent. Math. {\bf 102} (1990), 543--618.




 \bibitem[Col]{Col} {\em P. Colmez}, P\'eriods des vari\'et\'es
 ab\'eliennes \`a multiplication complex, Ann. Math., 138(1993),
 625-683.


\bibitem[Co]{Co} {\em B. Conrad}, Gross-Zagier revisited. With an appendix by W. R. Mann. Math. Sci. Res. Inst. Publ. {\bf 49},  Heegner points and Rankin $L$-series,  67--163, Cambridge Univ. Press, Cambridge (2004).



\bibitem[Du]{Du}  \emph{W. Duke},
Hyperbolic distribution problems and half-integral weight Maass forms.
Invent. Math. {\bf 92} (1988), 73--90.


\bibitem[EZ]{EZ} \emph{M. Eichler and D. Zagier}, The Theory of Jacobi
  Forms, Progress in Math. {\bf 55}, Birkh\"auser (1985).



\bibitem[Go]{Go} {\em E. Goren}, Lectures on Hilbert modular
varieties and modular forms, CRM monograph series 14, 2001.

\bibitem[Gr]{Gr} \emph{B. Gross}, Local heights on curves. In: Arithmetic Geometry. G. Cornell and J. Silvermann (eds.), 327--339, Springer-Verlag (1986).

\bibitem[GK]{GK} \emph{B. Gross and K. Keating},
On the intersection of modular correspondences,
Invent. Math. {\bf 112} (1993), 225--245.

\bibitem[GKZ]{GKZ} \emph{B. Gross, W. Kohnen, and D. Zagier}, Heegner
  points and derivatives of $L$-series. II.  Math. Ann.  {\bf 278}
  (1987), 497--562.

\bibitem[GZ]{GZ} {\em B. Gross and D. Zagier}, Heegner points and derivatives of $L$-series, Invent. Math. {\bf 84} (1986), 225--320.

\bibitem[Iw]{Iw} {\em
H. Iwaniec}
Fourier coefficients of modular forms of half-integral weight.
Invent. Math. {\bf 87} (1987), 385--401.

\bibitem[KM]{KM} {\em N. Katz and B. Mazur}, Arithmetic moduli of elliptic curves, Ann. Math. Stud. {\bf 108}, Princeton University Press (1985).



\bibitem[Ku1]{Ku1} \emph{S. Kudla}, Some extensions of the Siegel-Weil
formula, unpublished manuscript (1992). Available at
{\tt www.math.umd.edu/\~{}ssk} .


\bibitem[Ku2]{Ku2} \emph{S. Kudla},
Central derivatives of Eisenstein series and height pairings. Ann.
of Math. (2) {\bf 146} (1997), 545--646.

\bibitem[Ku3]{KuDuke}  \emph{S. Kudla},
Algebraic cycles on Shimura varieties of orthogonal type.  Duke
Math. J.  {\bf 86}  (1997),  no. 1, 39--78.

\bibitem[Ku4]{Ku4}
\emph{S. Kudla}, Integrals of Borcherds forms, Compositio Math.
\textbf{137} (2003), 293--349.

\bibitem[Ku5]{Ku:MSRI} {\em S. Kudla}, Special cycles and derivatives of Eisenstein series,
in {\em Heegner points and Rankin $L$-series}, Math. Sci. Res.
Inst. Publ. {\bf 49}, Cambridge University Press, Cambridge
(2004).

\bibitem[KR1]{KR1}
\emph{S. Kudla and S. Rallis}, On the Weil-Siegel formula,  J.
Reine Angew. Math. \textbf{387} (1988), 1--68.

\bibitem[KR2]{KR2}
\emph{S. Kudla and S. Rallis}, On the Weil-Siegel formula II,  J.
Reine Angew. Math. \textbf{391} (1988), 65--84.

\bibitem[KY1]{KY1} \emph{S. Kudla and T. Yang}, Pull-back of
arithmetic theta functions, in preparation.

\bibitem[KY2]{KY2} \emph{S. Kudla and T. Yang},
Derivatives of Eisenstein series,
in preparation.

\bibitem[KRY1]{KRY1} \emph{S. Kudla, M. Rapoport, and T. Yang},
On the derivative of an Eisenstein series of weight one, Internat.
Math. Res. Notices  {\bf 1999:7} (1999), 347--385.


\bibitem[KRY2]{KRY2}  {\em S. Kudla, M. Rapoport, and T.H. Yang},  Modular forms and special cycles on
Shimura curves,  Annals of Math. Studies series, vol 161, Princeton
Univ. Publ., 2006.


\bibitem[OT]{OT} {\em T. Oda and M. Tsuzuki}, Automorphic Green functions associated with the secondary spherical functions.
Publ. Res. Inst. Math. Sci.  {\bf 39}  (2003), 451--533.

\bibitem[Sche]{Sche} {\em N. R. Scheithauer}, Moonshine for Conway's group, Habilitation, University of Heidelberg (2004).

\bibitem[Scho]{Scho} {\em J. Schofer}, Borcherds forms and generalizations of singular moduli, J. Reine Angew. Math., to appear.



\bibitem[Sk]{Sk} \emph{N.-P. Skoruppa}, Explicit formulas for the
  Fourier coefficients of Jacobi and elliptic modular forms.  Invent.
  Math.  {\bf 102} (1990), 501--520.

\bibitem[SZ]{SZ} \emph{N.-P. Skoruppa and D. Zagier}, Jacobi forms
and a certain space of modular forms,  Invent. Math.  {\bf 94}
(1988), 113--146.

\bibitem[SABK]{SABK} {\em C. Soul\'e, D. Abramovich, J.-F. Burnol, and J. Kramer}, Lectures on Arakelov Geometry, Cambridge Studies in Advanced Mathematics {\bf 33}, Cambridge University Press, Cambridge (1992).

\bibitem[Vo]{Vo} {\em I. Vollaard},  On the Hilbert-Blumenthal moduli
problem, J. Inst. Math. Jussieu  {\bf 4}  (2005),
653--683.

\bibitem[We]{Weil}
\emph{A. Weil}, Sur la formule de Siegel dans la th\'eorie des
groupes classiques, Acta Math. \textbf{113} (1965) 1--87.

\bibitem[Wi]{Wi} {\em A. Wiles},
Modular elliptic curves and Fermat's last theorem,
Ann. of Math. {\bf 141} (1995), 443--551.

\bibitem[Ya1]{YaColmez} {\em T. H. Yang},  Chowla-Selberg formula
and Colmez's conjecture, preprint (2007), pp 15.

\bibitem[Ya2]{Yam=1} {\em T. H. Yang}, An arithmetic intersection
formula on Hilbert  modular surface. Preprint 2007, pp32.


\bibitem[Zh1]{Zh1} {\em S. Zhang},
Heights of Heegner cycles and derivatives of $L$-series,
Invent. Math. {\bf 130} (1997), 99--152.

\bibitem[Zh2]{Zh2} {\em S. Zhang},
Heights of Heegner points on Shimura curves,
Ann. of Math. {\bf 153} (2001),  27--147.
\end{thebibliography}
\end{document}